\documentclass[11pt,a4paper]{article}

\usepackage[utf8]{inputenc}
\usepackage[T1]{fontenc}
\usepackage[english]{babel} 
\usepackage{amsmath}
\usepackage{amsfonts}
\usepackage{amsthm}
\usepackage{amssymb}
\usepackage{makeidx}
\usepackage{graphicx}
\usepackage{lmodern}
\usepackage[left=2cm,right=2cm,top=2cm,bottom=2cm]{geometry}

\usepackage{color}
\usepackage{comment}
\usepackage[dvipsnames]{xcolor}
\usepackage{graphicx}
\usepackage{framed}
\usepackage{tikz}
\usepackage{calrsfs}
\DeclareMathAlphabet{\pazocal}{OMS}{zplm}{m}{n}
\usepackage{bm}

\usepackage{fourier-orns}
\usepackage{mathtools}
\usepackage{relsize}
\usepackage{dsfont}
\usepackage{comment}
\usepackage[linktocpage = true]{hyperref}
\hypersetup{
	colorlinks = true,
	linkcolor = {RoyalBlue},
	urlcolor = {RoyalBlue},
	citecolor = {Green}}

\newcommand{\R}{\mathbb{R}}

\renewcommand{\H}{\mathbb{H}}
\newcommand{\J}{\mathbb{J}}

\newcommand{\F}{\pazocal{F}}
\newcommand{\G}{\pazocal{G}}
\newcommand{\U}{\pazocal{U}}
\newcommand{\K}{\pazocal{K}}
\newcommand{\V}{\pazocal{V}}
\newcommand{\M}{\pazocal{M}}

\newcommand{\Cpazo}{\pazocal{C}}

\newcommand{\Hpazo}{\pazocal{H}}
\newcommand{\Npazo}{\pazocal{N}}
\newcommand{\Lpazo}{\pazocal{L}}
\newcommand{\Ppazo}{\pazocal{P}}

\newcommand{\Rpazo}{\pazocal{R}}
\newcommand{\Spazo}{\pazocal{S}}

\newcommand{\Wpazo}{\pazocal{W}}

\newcommand{\T}{\mathcal{T}}

\newcommand{\Lcal}{\mathcal{L}}
\newcommand{\Vcal}{\mathcal{V}}
\newcommand{\Pcal}{\mathcal{P}}
\newcommand{\Ccal}{\mathcal{C}}
\newcommand{\Scal}{\mathcal{S}}

\newcommand{\Id}{\textnormal{Id}}
\newcommand{\dn}{\textnormal{d}}
\newcommand{\D}{\textnormal{D}}

\newcommand{\Tan}{\textnormal{Tan}}

\newcommand{\supp}{\textnormal{supp}}

\newcommand{\Lip}{\textnormal{Lip}}
\newcommand{\AC}{\textnormal{AC}}
\newcommand{\Graph}{\textnormal{Graph}}
\newcommand{\loc}{\textnormal{loc}}
\newcommand{\Div}{\textnormal{div}}

\newcommand{\textbn}[1]{\textnormal{\textbf{#1}}}

\newcommand{\xb}{\boldsymbol{x}}
\newcommand{\yb}{\boldsymbol{y}}

\newcommand{\vb}{\boldsymbol{v}}

\newcommand{\Bgamma}{\boldsymbol{\gamma}}

\newcommand{\Bnu}{\boldsymbol{\nu}}

\newcommand{\Beta}{\boldsymbol{\eta}}
\newcommand{\Bmu}{\boldsymbol{\mu}}
\newcommand{\Bpartial}{\boldsymbol{\partial}}
\newcommand{\BGamma}{\boldsymbol{\Gamma}}
\newcommand{\vt}{\texttt{t}}

\newcommand{\INTDom}[3]{\int_{#2} #1 \textnormal{d} #3}
\newcommand{\INTSeg}[4]{\int_{#3}^{#4} #1 \textnormal{d} #2}
\newcommand{\NormL}[3]{\parallel \hspace{-0.1cm} #1 \hspace{-0.1cm} \parallel _ {L^{#2}(#3)}}
\newcommand{\NormC}[3]{\left\| #1  \right\| _ {C^{#2}(#3)}}
\newcommand{\Norm}[1]{\parallel \hspace{-0.1cm} #1 \hspace{-0.1cm} \parallel}
\newcommand{\derv}[2]{\frac{\textnormal{d} #1}{ \textnormal{d} #2}}
\newcommand{\tderv}[2]{\tfrac{\textnormal{d} #1}{ \textnormal{d} #2}}

\newtheorem{Def}{Definition}[section]
\newtheorem{thm}[Def]{Theorem}
\newtheorem{prop}[Def]{Proposition}
\newtheorem{rmk}[Def]{Remark}
\newtheorem{lem}[Def]{Lemma}
\newtheorem{cor}[Def]{Corollary}

\newtheorem{taggedhypx}{Hypotheses}
\newenvironment{taggedhyp}[1]
    {\renewcommand\thetaggedhypx{#1}\taggedhypx}
    {\endtaggedhypx}
    
\newtheorem{taggedhypsingx}{Hypothesis}
\newenvironment{taggedhypsing}[1]
    {\renewcommand\thetaggedhypsingx{#1}\taggedhypsingx}
    {\endtaggedhypsingx}

\title{Semiconcavity and Sensitivity Analysis in Mean-Field Optimal Control and Applications}

\author{Benoît Bonnet\footnote{CNRS,  IMJ-PRG,  UMR  7586,  Sorbonne  Université, 4  place  Jussieu,  75252  Paris,  France. \hfill \hspace{3.5cm} \textit{E-mail}: \texttt{benoit.bonnet@imj-prg.fr} (Corresponding author)} , Hélène Frankowska\footnote{CNRS,  IMJ-PRG,  UMR  7586,  Sorbonne  Université,  4  place  Jussieu,  75252  Paris,  France. \hfill \hspace{3.5cm} \textit{E-mail}: \texttt{helene.frankowska@imj-prg.fr}}}

\begin{document}

\maketitle

\begin{abstract}
In this article, we investigate some of the fine properties of the value function associated with an optimal control problem in the Wasserstein space of probability measures. Building on new interpolation and linearisation formulas for non-local flows, we prove semiconcavity estimates for the value function, and establish several variants of the so-called sensitivity relations which provide connections between its superdifferential and the adjoint curves stemming from the maximum principle. We subsequently make use of these results to study the propagation of regularity for the value function along optimal trajectories, as well as to investigate sufficient optimality conditions and optimal feedbacks for mean-field optimal control problems.
\end{abstract}

{\footnotesize

\vspace{0.25cm}

\textbf{MSC2020 Subject Classification :} 30L99, 49K27, 49K40, 49Q12, 49Q22, 58E25
}

\tableofcontents


\section{Introduction}

During the past fifteen years, the mathematical analysis of collective dynamics and multi-agent systems has undergone astonishingly rapid developments. The interest for such topics was historically initiated in communities working at large on the modelling of agent-based dynamics \cite{Jadbabaie2003,Moreau2005}, social choices \cite{Friedkin1990,Hegselmann2002} and aggregation patterns in biological systems \cite{Mogilner1999,Topaz2006}. In all likelihood, what propelled these lines of investigation at the foreground of several branches of modern mathematical analysis are, on the one hand, the works of Cucker and Smale \cite{CS1,CS2} on the mathematics of emergence, and on the other hand the simultaneous introduction of the theory of mean-field games by Lasry, Lions \cite{Lasry2007} and Huang, Caines, Malhamé \cite{Huang2006}. Incidentally, these developments and some of their outlets \cite{Carrillo2010,HaLiu} contributed to sparking a wide interest for multi-agent systems, studied in the so-called \textit{mean-field approximation} framework. In the latter, large deterministic systems of interacting particles are approximated by curves of densities, whose evolutions are described by transport equations in the space of measures (see e.g. \cite{golse} for a theoretical-physics flavoured overview of this topic). Concomitantly to the maturation of these research trends, several major progresses were made in the theory of \textit{optimal transport}. Some of the most notable ones lied in the identification of intrinsic geodesic \cite{McCann1997} and differential \cite{Otto2001} structures, which were amenable to computation while providing sound interpretations of various evolution problems arising in physics (see also \cite{Jordan1998}). These newly discovered concepts were further installed in the reference monograph \cite{AGS} (see also \cite{OTAM,Villani1}), and strongly contributed to establishing the so-called \textit{Wasserstein spaces} of probability measure as the natural framework for studying variational problems involving deterministic collective dynamics. 

Since then, a growing research effort has been devoted to the investigation of \textit{mean field control problems}, i.e. control problems formulated on mean-field approximations of discrete multi-agent systems (see e.g. the survey \cite{Choi2014}). This family of models refers broadly to situations in which a centralised policy-making entity emits a control signal at the macroscopic level, in order to stir an underlying microscopic multi-agent system towards a desired goal (see the introduction of \cite{LipReg} for more details on this general scheme).  While a few results have been dealing in this context with controllability issues \cite{Duprez2019,Duprez2020} as well as the explicit synthesis of control laws for consensus and alignment models \cite{Caponigro2013,Caponigro2015,Piccoli2021,ControlKCS}, the core of the existing contributions on this topic pertains to \textit{mean-field optimal control}. In this setting, a first series of articles have aimed at studying rigorously the mean-field limit of solutions of optimal control problems formulated on discrete particle-like systems \cite{Cavagnari2021,Fornasier2019,FornasierPR2014,Fornasier2014}. More recently, the depiction of optimality conditions in the form of Hamilton-Jacobi-Bellman equations \cite{achdou2,Cavagnari2018,Cavagnari2020,Jimenez2020} or variants of the Pontryagin Maximum Principle  \cite{MFPMP,PMPWassConst,SetValuedPMP,PMPWass,Pogodaev2016,Pogodaev2020} has also attracted a lot of attention. From a quite distinct standpoint, a few numerical schemes have been proposed for multi-agent control problems \cite{achdou1,AlbiPareschiZanella,Burger2020}. Let us stress that while many connections have been made between mean-field control problems and mean-field games (see e.g. \cite{Bensoussan2013,Carmona2015,Carmona2018,Lasry2007}), they are not fully reducible to each other due to the differences in their mathematical structures and application scopes.   

\bigskip

In this article -- which is a continuation of our previous works \cite{ContInc,SetValuedPMP} --, we investigate some of the fine regularity and structure properties of the value function $\Vcal : [0,T] \times \Pcal_c(\R^d) \rightarrow \R$ associated with a general Mayer mean-field optimal control problem in the space of measures, defined by 
\begin{equation}
\label{eq:ValueFunctionIntro}
\Vcal(\tau,\mu_{\tau}) :=  \left\{
\begin{aligned}
\inf_{u(\cdot) \in \U} & \, \big[ \varphi(\mu(T)) \big] \\
\text{s.t.} ~ & \left\{
\begin{aligned}
& \partial_t \mu(t) + \Div_x \big( v(t,\mu(t),u(t)) \mu(t) \big) = 0, \\
& \mu(\tau) = \mu_{\tau},
\end{aligned}
\right.
\end{aligned}
\right.
\end{equation}
for any $(\tau,\mu_{\tau}) \in [0,T] \times \Pcal_c(\R^d)$. Here, the minimisation is taken over the set $\U := \{ u : [0,T] \rightarrow U  ~\text{s.t. $u(\cdot)$ is $\Lcal^1$-measurable} \}$ of admissible controls with $(U,d_U)$ being a compact metric space. The time-evolution of the system is prescribed by the \textit{controlled non-local velocity field} $v : [0,T] \times \Pcal_c(\R^d) \times U \times \R^d \rightarrow \R^d$, while $\varphi : \Pcal_c(\R^d) \rightarrow \R$ represents a final cost.

\begin{rmk}[On the general equivalence between Mayer and Bolza problems]
As discussed in Section \ref{subsection:MFOC}, it is a known fact in optimal control theory that the so-called Bolza problems -- which comprise both a running and a final cost -- can be reduced to Mayer problems (see \cite[Section 4.2]{SetValuedPMP} in the context of mean-field control). Thus without loss of generality and for the sake of conciseness, we chose to restrict our subsequent developments to Mayer problems.
\end{rmk}

\begin{rmk}[On the choice of admissible controls]
Throughout this article, we will consider optimal control problems driven by \textnormal{open-loop} controls $u : [0,T] \rightarrow U$, which do not depend on the space variable $x \in \R^d$. This choice is motivated by the following two important facts. Firstly, on the application side, there exist a wealth of collective dynamics models -- the most commonly encountered being \textnormal{leader-follower} dynamics -- in which the dynamics of the agents is governed by a small number of control signals which are uniformly chosen for the whole system \cite{MFPMP,BonnetFCDC2020,Burger2021,Burger2020,Cavagnari2018,FornasierPR2014,Jimenez2020,Pogodaev2016}.  These classes of dynamical evolutions on measures are also becoming increasingly prominent in the mathematical branch of the machine learning literature, which focuses on the so-called NeurODE models of deep neural networks \cite{PMPNeurODEs,E2019,Jabir2021}. Secondly, on the technical side, while the methods developed in this article would be applicable to controlled vector fields $u : [0,T] \times \R^d \rightarrow U$, proving the corresponding results would require extra regularity assumptions on the admissible controls (see e.g. \cite{PMPWassConst,SetValuedPMP,PMPWass} for more details). This would also lead to heavy formulas involving the space derivatives of the admissible controls, without providing new insights on the results. Therefore, to lighten the presentation, we restrict our developments to open-loop controls.
\end{rmk}

Our goal in the aforedescribed context is to prove that the value function is  \textit{semiconcave} in a precise sense, as well as to derive the so-called \textit{sensitivity relations} which link the Hamiltonian and costates of the maximum principle to its superdifferentials. These latter have both been thoroughly studied in the context of classical control theory as further detailed below, and provide many interesting structural insights on optimal trajectories. 

Semiconcavity estimates have been known to appear in many problems coming from control theory or the calculus of variations for a fairly long time. In the general context of Hamilton-Jabobi-Bellman equations, it was noted as early as \cite{Capuzzo1984} that semiconcavity properties yield powerful quantitative stability estimates on the underlying solutions, a fact that is still frequently used to investigate various kinds of asymptotic properties in the context of mean-field games \cite{Cardaliaguet2017,Cardaliaguet2013} (see also the monograph \cite{Gomes2016}). We also point to the articles \cite{Gangbo2020,Gangbo2015}, devoted to the regularity theory for the so-called \textit{master equation}, in which semiconcavity plays a key role. More generally, semiconcave functions  have been used in other branches of control theory, notably to design stabilising feedbacks by Lyapunov methods \cite{Rifford2002} or to study the regularity of multivalued optimal feedbacks \cite{Cannarsa1991}. They also play an important part in the investigation of various problems in the field of sub-Riemannian geometry \cite{Rifford2014}, owing to the very structured nature of their singular sets. In addition to its numerous applications to variational problems, the notion of semiconcavity is frequently used in non-smooth analysis, as it ensures that several kinds of superdifferentials (Fr\'echet, Clarke, etc...) are non-empty and coincide \cite{CannarsaS2004}. This distinguishing feature also found its way into the theoretical foundations of subdifferential calculus in Wasserstein spaces. Indeed, it is shown throughout \cite{AGS} that the mirror notion of semiconvexity -- defined in a suitable sense along interpolating curves -- is the most natural one to ensure that extended measure subdifferentials are non-empty, as well as to derive quantitative decay estimates on gradient flows formulated in general metric spaces and in the space of probability measures. 

Sensitivity relations are a somewhat more specific -- but no less rich -- topic in optimal control theory. Established originally in \cite{Flemming1975} for $C^2$ value functions and subsequently extended in \cite{Clarke1987} to Lipschitz continuous value functions, they provide a link between the Pontryagin and Hamilton-Jacobi approaches to optimal control, by stating that the opposite of the adjoint variables of the maximum principle belong to the superdifferential of the value function. In the seminal paper \cite{Cannarsa1991} -- which served as a guiding thread for our present developments --, it was shown that sensitivity relations could be used in various ways to provide sufficient optimality conditions for Pontryagin extremals, see also \cite{Cannarsa2000}. Thus, sensitivity relations can be seen as the milestone supporting several connections bridging between the local uniqueness of optimal trajectories, the expression of optimal feedbacks, and differentiability properties of the value function. They have also been used in conjunction with semiconcavity estimates on several occasions, to investigate the regularity of generalised feedbacks mappings \cite{Cannarsa1991}, as well as to study propagations of regularity along optimal trajectories for the value function \cite{Cannarsa2013,Cannarsa2014}.

\bigskip 

This fruitful interplay between semiconcavity and sensitivity relations in the context of optimal control is what motivated the contributions of this paper, which can be summarised as follows. In Section \ref{section:Semiconcavity}, we show that the value function defined in \eqref{eq:ValueFunctionIntro} is semiconcave in three distinct ways. More specifically, we prove in Theorem \ref{thm:Semiconcavity1} that the latter is \textit{geodesically semiconcave} with respect to its second argument in the sense of \cite[Chapter 9]{AGS}, whenever the dynamics and cost functionals of the problem satisfy adequate interpolation inequalities along geodesics. We then show in Theorem \ref{thm:Semiconcavity2} that when the data of the optimal control problem satisfy similar interpolation estimates along arbitrary interpolating curves, these properties are bestowed upon the value function which is then \textit{strongly semiconcave} (see Definition \ref{def:Semiconcavity} below). Finally in Theorem \ref{thm:Semiconcavity3}, we prove that the value function is also semiconcave in the classical sense with respect to its first argument, provided that the non-local velocity field driving the problem is uniformly Lipschitz with respect to the time, measure and space variables. These regularity properties rely on the fine geodesic interpolation estimates between non-local flows established in Lemma \ref{lem:GeodesicInterpolationFlows}, which are based on novel structural results borrowed from \cite[Lemma 1]{ContInc}. In Section \ref{section:Sensitivity}, we shift our focus to sensitivity relations, which involve the intrinsic state-costate curves satisfying the maximum principle in Wasserstein spaces studied in \cite{MFPMP,PMPWassConst,SetValuedPMP,PMPWass}. In Theorem \ref{thm:Sensitivity1}, we prove that plans which are defined as an appropriate opposite of the costate measure belong to a localisation on compact sets of the extended Fr\'echet superdifferentials of the value function, defined in the sense of \cite[Chapter 10]{AGS}. The proof strategy subtending this result is then adapted in Theorem \ref{thm:Sensitivity2} to show that the \textit{barycentric projections} (see Theorem \ref{thm:Disintegration} below) of the state-costate curves with respect to their first marginals belong to the \textit{Dini superdifferential} of the value function (see Definition \ref{def:Value_GateauSuperDiff} below). In Section \ref{section:Applications}, we make use of these new results to investigate three topics that were previously mentioned for their relevance in the study of optimal control problems: the propagation of regularity for the value function along optimal trajectories, sufficient optimality conditions, and the regularity of optimal feedbacks. 

\begin{rmk}[On the choice of stating results both for geodesic and strong interpolations]
The reason why we study semiconcavity and sensitivity results involving both the intrinsic geodesic and optimal transport structures of Wasserstein spaces on the one hand, and arbitrary interpolations curves and perturbation directions on the other hand, is the following. While the former setting is more geometrically meaningful and potentially better suited to investigating problems arising in the calculus of variations, the latter is definitely more adapted to the analysis of control systems. Indeed, as contextually underlined in Remark \ref{rmk:FrechetSubdiff} below and further illustrated in Section \ref{section:Applications}, the class of variations that usually appear in control theory are generically not optimal displacement directions. Thus, while semiconcavity and sensitivity results expressed in terms of the intrinsic structures e.g. of \cite{AGS} are of high interest in themselves, they are not always applicable to control problems. 
\end{rmk}

Concerning the bibliographical positioning of our work, we would like to mention first that the results of Section \ref{section:Semiconcavity} can be loosely connected to existing contributions in the literature of mean-field control \cite{CavagnariM2019,CavagnariMP2018} and mean-field games \cite{Gangbo2020,Gangbo2015}. However, to the best of our knowledge, this article seems to be the first one to investigate semiconcavity properties at this level of generality for mean-field optimal control problems. In contrast, the sensitivity relations of Section \ref{section:Sensitivity} along with the applications presented in Section \ref{section:Applications} were not explored previously, and are therefore completely new. We point out that the results of Section \ref{section:Sensitivity} rely on general linearisation properties for non-local flows, which are exposed in Appendix \ref{section:AppendixFlowDiff}. While quite naturally expected, these results were not available at this degree of generality in the literature, and should constitute a useful addition to the optimal transport toolbox that is being developed for dynamical and variational problems studied in the space of probability measures. 

\bigskip

The structure of the article is the following. In Section \ref{section:Preliminaries}, we recollect classical notions pertaining to optimal transport theory, subdifferential calculus in Wasserstein spaces, continuity equations with non-local velocities and mean-field optimal control problems. While the corresponding concepts are mostly well-known, we stress that the notion of measure subdifferentials presented in Section \ref{subsection:SubdiffWass} is not exactly the one of \cite[Chapter 10]{AGS}, but rather its adaptation to compactly supported measures, following some recent results from \cite{SetValuedPMP}. In Section \ref{section:Semiconcavity}, we introduce two notions of semiconcavity in the spirit of \cite[Chapter 9]{AGS} for functionals defined over compactly supported measures, and establish semiconcavity estimates for the value function. We subsequently prove the Fr\'echet- and Dini-type sensitivity relations involving the Pontryagin costates in Section \ref{section:Sensitivity}, and leverage these latter together with the semiconcavity estimates to investigate various structural properties of mean-field optimal control problems in Section \ref{section:Applications}. Appendix \ref{section:AppendixFlowDiff} is devoted to the derivation of general linearisation formulas for non-local flows, while Appendices \ref{section:AppendixThm}, \ref{section:AppendixGronwall} and \ref{section:AppendixLem} contain technical results required at different stages in our arguments. 


\section{Preliminaries}
\label{section:Preliminaries}

We start by introducing several concepts that will appear throughout the manuscript in the formulations and proofs of our main results.


\subsection{Measure theory and optimal transport}

In this section, we recall a few notions of measure theory, functional analysis and optimal transport. We point to the reference monographs \cite{AmbrosioFuscoPallara,Aubin1990,DiestelUhl} for the two former and to \cite{AGS,OTAM,Villani1} for the latter.

For $d \geq 1$, we denote by $\Lcal^d$ the standard $d$-dimensional Lebesgue measure defined over $\R^d$. Given a separable Banach space $(X,\Norm{\cdot}_X)$ and a real number $p \in [1,+\infty]$, we use the notation $L^p(\Omega,X)$ for the space of $\Lcal^d$-integrable maps from a subset $\Omega \subset \R^d$ into $X$, defined in the sense of Bochner \cite{DiestelUhl}. It is then a well-known result in measure theory (see e.g. \cite[Chapter II - Theorem 9]{DiestelUhl}) that for every Bochner-integrable map $f \in L^1([0,T],X)$, there exists a subset of \textit{Lebesgue points} $\T_f \subset (0,T)$ of full $\Lcal^1$-measure, such that 
\begin{equation}
\label{eq:LebesguePoint_Banach}
\frac{1}{h} \INTSeg{\Norm{f(t) - f(\tau)}_X }{t}{\tau}{\tau+h} ~\underset{h \rightarrow 0^+}{\longrightarrow}~ 0, 
\end{equation}
for every $\tau \in \T_f$. In what follows, $C^0(\Scal,X)$ will stand for the vector space of \textit{continuous functions} from a metric space $(\Scal,d_{\Scal})$ into $(X,\Norm{\cdot}_X)$, and we shall denote by $\Lip(\phi(\cdot) ; \Scal)$ the Lipschitz constant of a map $\phi : \Scal \rightarrow X$. In the particular case where $\Omega := [0,T] \subset \R$ for some $T >0$, we will also use the notation $\AC([0,T],\Scal)$ for the metric space of \textit{absolutely continuous arcs} from $[0,T]$ into $\Scal$. Finally, we will denote by ``$\circ$'' the standard composition operation between functions.

Throughout this article, we denote by $\Pcal(X)$ the space of \textit{Borel probability measures} over a Banach space $(X,\Norm{\cdot}_X)$. The latter is endowed with the standard \textit{narrow topology}, induced by the weak-$^*$ convergence of measures 
\begin{equation}
\label{eq:NarrowTopo}
\mu_n ~\underset{n \rightarrow +\infty}{\rightharpoonup^*}~ \mu \qquad \text{if and only if} \qquad \INTDom{\xi(x)}{X}{\mu_n(x)} ~\underset{n \rightarrow +\infty}{\longrightarrow}~ \INTDom{\xi(x)}{X}{\mu(x)},
\end{equation}
in duality with continuous and bounded maps $\xi : X \rightarrow \R$. We will also denote by $\Pcal_p(X)$ the subset of measures $\mu \in \Pcal(X)$ whose \textit{momentum of order $p$} is finite, i.e.
\begin{equation*}
\M_p(\mu) := \bigg( \INTDom{|x|^p}{X}{\mu(x)}  \bigg)^{1/p} < +\infty.
\end{equation*}
In the sequel, we will often work with the subset $\Pcal_c(X)$ of elements $\mu \in \Pcal(X)$, whose \textit{supports}
\begin{equation*}
\supp(\mu) := \bigg\{ x \in X ~\text{s.t.}~ \mu(\Npazo_x) > 0 ~\text{for any neighbourhood $\Npazo_x$ of $x \in X$} \bigg\},
\end{equation*}
are compact. Finally in the case where $(X,\Norm{\cdot}_X) := (\R^d,|\cdot|)$,  we will use the notation $L^p(\Omega,\R^d;\mu)$ for the space of maps from $\Omega \subset \R^d$ into $\R^d$, and which are $p$-integrable against $\mu \in \Pcal(\R^d)$.

\begin{Def}[Pushforward of probability measures]
Given $\mu \in \Pcal(X)$ and a Borel map $f : X \rightarrow X$, we denote by $f_{\#} \mu \in \Pcal(X)$ the \textnormal{pushforward of $\mu$ through $f$}, which is the measure defined by $f_{\#} \mu(B) := \mu(f^{-1}(B))$ for any Borel set $B \subset X$. 
\end{Def}

\begin{Def}[Transport plans between measures]
Given $\mu,\nu \in \Pcal(X)$, the set of \textnormal{transport plans} $\Gamma(\mu,\nu)$ between $\mu$ and $\nu$ is the subset of elements $\gamma \in \Pcal(X \times X)$ such that $\pi^1_{\#} \gamma = \mu$ and $\pi^2_{\#} \gamma = \nu$, where $\pi^1,\pi^2 : X \times X \rightarrow X$ stand for the projection operators onto the first and second factor respectively. 
\end{Def}

Throughout the remainder of this section, we assume that $(X,\Norm{\cdot}_X) := (\R^d,|\cdot|)$. Given a real number $p \in [1,+\infty)$, it is a well-known fact in optimal transport theory that
\begin{equation*}
W_p(\mu,\nu) := \min \bigg\{ \Big( \INTDom{|x-y|^p}{\R^{2d}}{\gamma(x,y)} \Big)^{1/p} ~\text{s.t.}~ \gamma \in \Gamma(\mu,\nu) \bigg\}, 
\end{equation*}
defines a distance over $\Pcal_p(\R^d)$. We will henceforth denote by $\Gamma_o(\mu,\nu)$ the corresponding set of \textit{$p$-optimal transport plans} for which this minimum is attained\footnote{The fact that $\Gamma_o(\mu,\nu)$ is non-empty follows from the direct method of the calculus of variations.}. In the following proposition, we recall several interesting properties of the so-called \textit{Wasserstein spaces} $(\Pcal_p(\R^d),W_p)$.

\begin{prop}[Some properties of the Wasserstein spaces]
\label{prop:WassProp}
The spaces $(\Pcal_p(\R^d),W_p)$ are complete separable metric spaces, on which the $W_p$-distance metrises the narrow topology \eqref{eq:NarrowTopo}, in the sense that 
\begin{equation*}
W_p(\mu_n,\mu) ~\underset{n \rightarrow +\infty}{\longrightarrow}~ 0 \qquad \text{if and only if} \qquad \left\{
\begin{aligned}
& \hspace{2.1cm} \mu_n ~\underset{n \rightarrow +\infty}{\rightharpoonup^*}~ \mu, \\
& \INTDom{|x|^p}{\R^d}{\mu_n(x)} ~\underset{n \rightarrow +\infty}{\longrightarrow}~ \INTDom{|x|^p}{\R^d}{\mu(x)}, 
\end{aligned}
\right.
\end{equation*}
for every $(\mu_n) \subset \Pcal_p(\R^d)$ and $\mu \in \Pcal_p(\R^d)$. The Wasserstein distances between elements $\mu,\nu \in \Pcal(\R^d)$ are ordered, namely $W_{p_1}(\mu,\nu) \leq W_{p_2}(\mu,\nu)$ whenever $p_1 \leq p_2$, and in the particular case where $\mu,\nu \in \Pcal_c(\R^d)$, the following \textnormal{Kantorovich-Rubinstein duality} formula holds
\begin{equation}
\label{eq:KantorovichDuality}
W_1(\mu,\nu) = \sup \bigg\{ \INTDom{\phi(x)}{\R^d}{(\mu-\nu)(x)} ~\text{s.t.}~  \Lip(\phi(\cdot) ; \R^d) \leq 1 \bigg\}.
\end{equation}
\end{prop}

The Wasserstein metrics, in addition to their interesting topological and geometric properties, allow for various useful estimates. For instance, if $\mu \in \Pcal_p(\R^d)$ and $\zeta,\xi \in L^p(\R^d,\R^d;\mu)$, one has
\begin{equation}
\label{eq:WassEst1}
W_p \big( \zeta_{\#} \mu, \xi_{\#} \mu \big) \, \leq ~ \NormL{\zeta - \xi}{p}{\mu}.
\end{equation}
In addition, given a compact set $K \subset \R^d$, two elements $\mu,\nu \in \Pcal(K)$ and a Lipschitz map $\phi : K \rightarrow \R^d$, it holds
\begin{equation}
\label{eq:WassEst2}
W_p(\phi_{\#} \mu,\phi_{\#}\nu) \leq \Lip \big( \phi(\cdot) ;K \big) W_p(\mu,\nu).
\end{equation}
We end this series of prerequisites by recalling a variant of the well-known \textit{disintegration theorem} in the context of optimal transport, which we will use extensively in the sequel.

\begin{thm}[Disintegration]
\label{thm:Disintegration}
Let $\mu,\nu \in \Pcal_p(\R^d)$ and $\gamma \in \Gamma(\mu,\nu)$. Then, there exists a $\mu$-almost uniquely determined Borel map $x \in \R^d \mapsto \gamma_x \in \Pcal_p(\R^d)$ such that 
\begin{equation*}
\INTDom{\xi(x,y)}{\R^{2d}}{\gamma(x,y)} = \INTDom{\INTDom{\xi(x,y)}{\R^d}{\gamma_x(y)}}{\R^d}{\mu(x)}, 
\end{equation*}
for any $\xi \in L^1(\R^{2d},\R;\gamma)$. The family of measures $\{ \gamma_x \}_{x \in \R^d} \subset \Pcal_p(\R^d)$ is called the \textnormal{disintegration} of $\gamma$ against its first marginal $\pi^1_{\#} \gamma = \mu$, which will be denoted by $\gamma := \INTDom{\gamma_x}{\R^d}{\mu(x)}$. The \textnormal{barycentric projection} $\bar{\gamma} \in L^p(\R^d,\R^d;\mu)$ of $\gamma$ onto $\pi^1_{\#} \gamma = \mu$ is then defined as
\begin{equation}
\label{eq:Barycenter}
\bar{\gamma}(x) := \INTDom{y \,}{\R^d}{\gamma_x(y)},
\end{equation}
for $\mu$-almost every $x \in \R^d$. 
\end{thm}


\subsection{Subdifferential calculus in $(\Pcal_c(\R^d),W_2)$}
\label{subsection:SubdiffWass}

We adapt here some of the definitions and results of the theory of $\Pcal_2$-subdifferential calculus developed in \cite[Chapter 10]{AGS} to the setting of compactly supported measures. Throughout this section, we will write $\phi : \Pcal_c(\R^d) \rightarrow \R$ to mean the restriction of an extended real-valued functional $\phi : \Pcal_2(\R^d) \rightarrow \R \cup \{ \pm \infty \}$ such that $\Pcal_c(\R^d) \subset D(\phi) := \{ \mu \in \Pcal_2(\R^d) ~\text{s.t.}~ \phi(\mu) \neq \pm \infty \}$.   

In the following definition, we propose a localisation for compactly supported measures of the extended Fr\'echet subdifferential from \cite[Definition 10.3.1]{AGS}, in the spirit of \cite{SetValuedPMP}. Given an element $\mu \in \Pcal_c(\R^d)$ and $R > 0$, we denote by $B_R(\mu) := \cup_{x \in \supp(\mu)} B(x,R)$ the $R$-fattening of $\supp(\mu)$. 

\begin{Def}[Localised extended Fr\'echet sub and superdifferentials]
\label{def:Generalised_Subdiff}
We say that a plan $\Bgamma \in \Pcal_2(\R^{2d})$ belongs to the \textnormal{localised extended Fr\'echet subdifferential} $\Bpartial^-_{\loc} \phi(\mu)$ of $\phi(\cdot)$ at $\mu \in \Pcal_c(\R^d)$ provided that $\pi^1_{\#} \Bgamma = \mu$, and 
\begin{equation*}
\phi(\nu) - \phi(\mu) \geq \inf_{\Bmu \in \BGamma_o^{1,3}(\Bgamma,\nu)} \INTDom{\langle r,y-x \rangle}{\R^{3d}}{\Bmu(x,r,y)} + o_R(W_2(\mu,\nu)),
\end{equation*}
for every $R > 0$ and any $\nu \in \Pcal(B_R(\mu))$, where 
\begin{equation}
\label{eq:Bgamma_o13Def}
\BGamma_o^{1,3}(\Bgamma,\nu) := \Big\{ \Bmu \in \Pcal_2(\R^{3d}) ~\text{s.t.}~ \pi^{1,2}_{\#} \Bmu = \Bgamma ~\text{and}~ \pi^{1,3}_{\#} \Bmu \in \Gamma_o(\mu,\nu) \Big\}.
\end{equation}
Similarly, we define the \textnormal{localised extended Fr\'echet superdifferential} $\Bpartial_{\loc}^+ \phi(\mu)$ as the set of plans $\Bgamma \in \Pcal_2(\R^{2d})$ such that $(\pi^1,-\pi^2)_{\#} \Bgamma \in \Bpartial_{\loc}^-(-\phi)(\mu)$.
\end{Def}

\begin{Def}[Localised classical Fr\'echet sub and superdifferentials]
\label{def:Classical_Subdiff}
We say that a map $\xi \in L^2(\R^d,\R^d;\mu)$ belongs to the \textnormal{localised classical subdifferential} $\partial_{\loc}^- \phi(\mu)$ of $\phi(\cdot)$ at $\mu \in \Pcal_c(\R^d)$ if
\begin{equation*}
\phi(\nu) - \phi(\mu) \geq \inf_{\gamma \in \Gamma_o(\mu,\nu)} \INTDom{\langle \xi(x) , y-x \rangle}{\R^{2d}}{\gamma(x,y)} + o_R(W_2(\mu,\nu)), 
\end{equation*}
for every $R > 0$ and any $\nu \in \Pcal(B_{R}(\mu))$. Analogously, we define the \textnormal{localised classical superdifferential} $\partial_{\loc}^+ \phi(\mu)$ of $\phi(\cdot)$ as the set of maps $\xi \in L^2(\R^d,\R^d;\mu)$ such that $(-\xi) \in \partial_{\loc}^-(-\phi)(\mu)$.
\end{Def}

\begin{rmk}[Comparison between Definition \ref{def:Generalised_Subdiff} and Definition \ref{def:Classical_Subdiff}]
\label{rmk:FrechetSubdiff}
Originally, the notion of classical subdifferentiability was primarily used for measures $\mu \in \Pcal_2(\R^d)$ that are absolutely continuous with respect to $\Lcal^d$. Indeed, it can be shown in this context that the notions introduced in Definition \ref{def:Generalised_Subdiff} and Definition \ref{def:Classical_Subdiff} coincide when one replaces localised subdifferentials by standard ones, in the sense that $\Bgamma \in \Bpartial^- \phi(\mu)$ if and only if $\bar{\Bgamma} \in \partial^- \phi(\mu)$ (see \cite[Remark 10.3.3]{AGS}). In practice however, a wide range of functionals can be handled using the simpler notion of Definition \ref{def:Classical_Subdiff}, even for arbitrary measures $\mu \in \Pcal_2(\R^d)$ (see e.g. \cite{PMPWassConst,SetValuedPMP,CDLL,Gangbo2019} and references therein). In the present manuscript, we will make use of the concept of measure-valued subdifferentials provided by Definition \ref{def:Generalised_Subdiff} to state Fr\'echet-type sensitivity relations for the state-costate curves $t \in [0,T] \mapsto \nu^*(t) \in \Pcal_c(\R^{2d})$ satisfying the maximum principle of Theorem \ref{thm:PMP} below. On the other hand, the simpler structure displayed in Definition \ref{def:Classical_Subdiff} will be used to investigate stronger differentiability properties of dynamics and cost functionals defined over $\Pcal_c(\R^d)$.
\end{rmk}

We recall next a notion of localised differentiability for functionals defined over $\Pcal_c(\R^d)$. The latter was introduced by the authors of the present manuscript in \cite{SetValuedPMP}, taking inspiration from \cite{Gangbo2019}, and its formulation involves the so-called \textit{analytical tangent space} 
\begin{equation*}
\Tan_{\mu} \Pcal_2(\R^d) := \overline{\big\{ \nabla \xi(\cdot) ~\text{s.t.}~ \xi \in C^{\infty}_c(\R^d,\R) \big\}}^{L^2(\mu)}, 
\end{equation*}
to $(\Pcal_2(\R^d),W_2)$ at $\mu \in \Pcal_2(\R^d)$ (see \cite[Sections 8.4 and 12.4]{AGS}). We point in particular to \cite[Section 5]{SetValuedPMP} for some illustrations of the relevance of this notion of differentiability when studying smooth integral functionals.

\begin{Def}[Locally differentiable functionals]
\label{def:Classical_Diff} 
A map $\phi : \Pcal_c(\R^d) \rightarrow \R$ is said to be \textnormal{locally differentiable} at $\mu \in \Pcal_c(\R^d)$ if there exists an element $\nabla \phi(\mu) \in \Tan_{\mu} \Pcal_2(\R^d)$ -- called the \textnormal{Wasserstein gradient} of $\phi(\cdot)$ at $\mu \in \Pcal_c(\R^d)$ --, such that $\partial^-_{\loc} \phi(\mu) \cap \partial^+_{\loc} \phi(\mu) = \{ \nabla \phi(\mu) \}$. 

Similarly, given $m \geq 1$, a map $\phi : \Pcal_c(\R^d) \rightarrow \R^m$ is said to be locally differentiable at $\mu \in \Pcal_c(\R^d)$ if its components $(\phi_i(\cdot))_{1 \leq i \leq m}$ are locally differentiable in the sense just defined.
\end{Def}

From now on, we will drop the ``loc'' subscript in the sub and superdifferentials, as the only notions that we will use in the sequel are the localised ones introduced above. We recall in the next proposition a result derived in \cite[Proposition 3.6]{SetValuedPMP}, which provides a general first-order expansion formula along arbitrary plans for the Wasserstein gradient.

\begin{prop}[Chain rule along arbitrary transport plans]
\label{prop:GradientChainrule}
Let $\phi : \Pcal_c(\R^d) \rightarrow \R$ be locally differentiable at $\mu \in \Pcal_c(\R^d)$. Then for any $R >0$ and $\nu \in \Pcal(B_R(\mu))$, it holds
\begin{equation}
\label{eq:ChainruleStatement}
\phi(\nu) - \phi(\mu) = \INTDom{\langle \nabla \phi(\mu)(x) , y-x \rangle}{\R^{2d}}{\Bmu(x,y)} + o_R(W_{2,\Bmu}(\mu,\nu)), 
\end{equation}
for every $\Bmu \in \Gamma(\mu,\nu)$, where 
\begin{equation}
\label{eq:WeightedWass}
W_{2,\Bmu}(\mu,\nu) := \bigg( \INTDom{|x-y|^2}{\R^{2d}}{\Bmu(x,y)} \bigg)^{1/2}.
\end{equation}
Conversely, if there exists a map $\nabla\phi(\mu) \in \Tan_{\mu} \Pcal_2(\R^d)$ such that for every $R > 0$, any $\nu \in \Pcal(B_R(\mu))$ and each optimal transport plan $\Bmu \in \Gamma_o(\mu,\nu)$ the identity \eqref{eq:ChainruleStatement} is satisfied, then $\phi(\cdot)$ is locally differentiable at $\mu$, and $\nabla\phi(\mu)$ is its Wasserstein gradient.
\end{prop}

We end this section by a vector-valued version of the chain rule of Proposition \ref{prop:GradientChainrule}.

\begin{cor}[Chain rule for vector-valued maps]
\label{cor:DiffChainrule}
Suppose that $\phi : \Pcal_c(\R^d) \rightarrow \R^m$ is locally differentiable. Then, for any $R>0$, every $\nu \in \Pcal(B_R(\mu))$ and each $\Bmu \in \Gamma(\mu,\nu)$, it holds
\begin{equation*}
\phi(\nu) = \phi(\mu) + \INTDom{\D \phi(\mu)(x)(y-x)}{\R^{2d}}{\Bmu(x,y)} + o_R(W_{2,\Bmu}(\mu,\nu)).
\end{equation*}
Here, $x \in \R^d \mapsto \D \phi(\mu)(x) := (\nabla \phi_i(\mu)(x))_{1 \leq i \leq m} \in \R^{m \times d}$ is the matrix-valued map whose rows are the Wasserstein gradients of the components $(\phi_i(\cdot))_{1 \leq i \leq m}$ of $\phi(\cdot)$ at $\mu \in \Pcal_c(\R^d)$. 
\end{cor}


\subsection{Continuity equations with non-local velocities in $\R^d$}
\label{subsection:NonLocalCE}

In this section, we recollect several fundamental results on non-local continuity equations defined over the metric space $(\Pcal_c(\R^d),W_1)$. We refer to \cite{AmbrosioC2014} and \cite[Chapter 8]{AGS} as well as to their references for an exhaustive treatment of continuity equations with measure-independent driving fields, and to \cite{AmbrosioGangbo,ContInc,Pedestrian} for the main well-posedness results on non-local continuity equations.

Given a \textit{non-local velocity field} $v : [0,T] \times \Pcal_c(\R^d) \times \R^d \rightarrow \R^d$ and a measure $\mu^0 \in \Pcal_c(\R^d)$, we consider the Cauchy problem
\begin{equation}
\label{eq:NonLocal_CE}
\left\{
\begin{aligned}
& \partial_t \mu(t) + \Div_x \big( v(t,\mu(t)) \mu(t) \big) = 0, \\
& \mu(0) = \mu^0,
\end{aligned}
\right.
\end{equation}
where the first line of \eqref{eq:NonLocal_CE} needs to be understood in the sense of distribution against smooth and compactly supported functions, namely
\begin{equation}
\label{eq:NonLocal_Distrib1}
\INTSeg{\INTDom{\Big( \partial_t \xi(t,x) + \langle \nabla_x \xi(t,x) , v(t,\mu(t),x) \rangle \Big)}{\R^d}{\mu(t)(x)}}{t}{0}{T} = 0, 
\end{equation}
for every $\xi \in C^{\infty}_c((0,T) \times \R^d,\R)$. This identity can be equivalently rewritten\footnote{The equivalence between these expressions follows by intregrating by parts against test functions of the form $(t,x) \in [0,T] \times \R^d \mapsto \zeta(t) \phi(x) \in \R^d$ with $\zeta \in C^{\infty}_c((0,T),\R)$ and $\psi \in C^{\infty}_c(\R^d,\R)$, see e.g. \cite{AmbrosioC2014} and \cite[Chapter 8]{AGS}.} as 
\begin{equation}
\label{eq:NonLocal_Distrib2}
\derv{}{t} \INTDom{\psi(x)}{\R^d}{\mu(t)(x)} = \INTDom{\langle \nabla \psi(x) , v(t,\mu(t),x) \rangle}{\R^d}{\mu(t)(x)}, 
\end{equation}
for any $\psi \in C^{\infty}_c(\R^d,\R)$ and $\Lcal^1$-almost every $t \in [0,T]$. In the sequel, we will often identify non-local velocity fields $(t,\mu,x) \mapsto v(t,\mu,x) \in \R^d$ with vector field valued map $t \mapsto v(t,\cdot,\cdot) \in C^0(\Pcal_c(\R^d) \times \R^d,\R^d)$, whose measurability and integrability are understood in the sense of Definition \ref{def:IntegralC0} below. Throughout this section, we will make use of the following set of assumptions. 

\begin{taggedhyp}{\textbn{(CE)}}
\label{hyp:CE}
Suppose that for any $R > 0$, the following holds with $K := B(0,R)$. 
\begin{enumerate}
\item[$(i)$] The application $t \in [0,T] \mapsto v(t,\mu,x) \in \R^d$ is $\Lcal^1$-measurable for any $(\mu,x) \in \Pcal_c(\R^d) \times \R^d$, and there exists a map $m(\cdot) \in L^1([0,T],\R_+)$ such that 
\begin{equation*}
|v(t,\mu,x)| \leq m(t) \Big( 1 + |x| + \M_1(\mu) \Big), 
\end{equation*}
for $\Lcal^1$-almost every $t \in [0,T]$ and any $(\mu,x) \in \Pcal_c(\R^d) \times \R^d$.
\item[$(ii)$] There exist two maps $l_K(\cdot),L_K(\cdot) \in L^1([0,T],\R_+)$ such that 
\begin{equation*}
|v(t,\mu,x) - v(t,\mu,y)| \leq l_K(t) |x-y| \qquad \text{and} \qquad |v(t,\mu,x) - v(t,\nu,x)| \leq L_K(t) W_1(\mu,\nu), 
\end{equation*}
for $\Lcal^1$-almost every $t \in [0,T]$, any $\mu,\nu \in \Pcal(K)$ and all $x,y \in K$.
\item[$(iii)$] The map $x \in  \R^d \mapsto v(t,\mu,x) \in \R^d$ is Fr\'echet-differentiable for $\Lcal^1$-almost every $t \in [0,T]$ and any $\mu \in \Pcal_c(\R^d)$, and $(\mu,x) \in \Pcal_1(K) \times K \mapsto \D_x v(t,\mu,x) \in \R^{d \times d}$ is continuous. 
\item[$(iv)$] The map $\mu \in \Pcal_c(\R^d) \mapsto v(t,\mu,x) \in \R^d$ is locally differentiable for $\Lcal^1$-almost every $t \in [0,T]$ and any $x \in \R^d$, and $(\mu,x,y) \in \Pcal_1(K) \times K \times K \mapsto \D_{\mu} v(t,\mu,x)(y) \in \R^{d \times d}$ is continuous.
\end{enumerate}
\end{taggedhyp}

Hypotheses \ref{hyp:CE}-$(i),(ii)$ are standard sub-linearity and Cauchy-Lipschitz type regularity assumptions which ensure that \eqref{eq:NonLocal_CE} is well-posed, while \ref{hyp:CE}-$(iii),(iv)$ are needed to formulate the PMP recalled in Theorem \ref{thm:PMP} and the technical linearisation results of Appendices  \ref{section:AppendixFlowDiff}, \ref{section:AppendixThm} and \ref{section:AppendixGronwall}. 

\begin{Def}[Non-local flows of diffeomorphisms] 
\label{def:NonlocalFlows}
Let $v : [0,T] \times \Pcal_c(\R^d) \times \R^d \rightarrow \R^d$ be a non-local velocity field satisfying hypotheses \ref{hyp:CE}-$(i),(ii)$. For any compact set $K \subset \R^d$, we define the \textnormal{non-local flow of diffeomorphisms} $\Phi_{(\tau,\cdot)}[\mu](\cdot) \in C^0([0,T] \times K,\R^d)$ starting from $\mu \in \Pcal(K)$ at time $\tau \in [0,T]$ as the unique solution of the non-local Cauchy problem
\begin{equation}
\label{eq:NonLocalFlow_Def}
\Phi_{(\tau,t)}[\mu](x) = x + \INTSeg{v \Big( s , \Phi_{(\tau,s)}[\mu](\cdot)_{\#} \mu , \Phi_{(\tau,s)}[\mu](x) \Big)}{s}{\tau}{t}, 
\end{equation}
for every $(t,x) \in [0,T] \times K$.
\end{Def}

\begin{rmk}[Existence and uniqueness of non-local flows]
Given $(\tau,\mu) \in [0,T] \times \Pcal(K)$, the existence and uniqueness of $\Phi_{(\tau,\cdot)}[\mu](\cdot) \in C^0([0,T] \times K,\R^d)$ can be obtained under hypotheses \ref{hyp:CE}-$(i),(ii)$ by a fixed point argument that is detailed in the proof of Theorem \ref{thm:FlowDiff} in Appendix \ref{section:AppendixThm} below. 
\end{rmk}

We recall next the main well-posedness, stability and representation results for non-local continuity equations in the Cauchy-Lipschitz setting, for which we refer to \cite{ContInc}. In the sequel, we will write $\Norm{\cdot}_1 := \NormL{\cdot}{1}{[0,T],\R_+}$ to mean the $L^1$-norm of a positive real-valued map defined over $[0,T]$. 

\begin{thm}[Well-posedness and flow representation for solutions of \eqref{eq:NonLocal_CE}]
\label{thm:NonLocalCE}
Let $\mu^0 \in \Pcal_c(\R^d)$ and $v : [0,T] \times \Pcal_c(\R^d) \times \R^d \rightarrow \R^d$ be a non-local velocity field satisfying \ref{hyp:CE}-$(i),(ii)$. 

Then, there exists a unique curve of measures $\mu(\cdot)$ solution of \eqref{eq:NonLocal_CE}. Furthermore for every $r > 0$, any $\mu^0 \in \Pcal(B(0,r))$ and all times $0 \leq \tau \leq t \leq T$, it holds
\begin{equation}
\label{eq:SuppAC_Est}
\supp(\mu(t)) \subset K := B(0,R_r) \qquad \text{and} \qquad W_1(\mu(\tau),\mu(t)) \leq \INTSeg{m_r(s)}{s}{\tau}{t},
\end{equation}
where
\begin{equation*}
R_r := (r \, +\Norm{m(\cdot)}_1) \Big( 1+ T\exp(2\Norm{m(\cdot)}_1) \Big) \qquad \text{and} \qquad m_r(\cdot) := (1+2R_r)m( \cdot) \in L^1([0,T],\R_+).
\end{equation*}
Moreover, the unique measure curve $\mu(\cdot)$ solving \eqref{eq:NonLocal_CE} is such that
\begin{equation}
\label{eq:FlowRep}
\mu(t) = \Phi_{(\tau,t)}[\mu(\tau)](\cdot)_{\#} \mu(\tau),
\end{equation}
for all times $\tau,t \in [0,T]$, where the family of non-local flows $(\Phi_{(\tau,t)}[\mu(\tau)](\cdot))_{\tau,t \in [0,T]}$ satisfies the following \textnormal{semigroup property} 
\begin{equation*}
\Phi_{(\tau,t)}[\mu(\tau)](\cdot) = \Phi_{(s,t)}[\mu(s)] \circ \Phi_{(\tau,s)}[\mu(\tau)](\cdot), 
\end{equation*}
in $C^0(K,\R^d)$ for every $\tau,s,t \in [0,T]$.
\end{thm}

In the next theorem, we recall a variant of the celebrated \textit{superposition principle}, for which we refer the reader to the seminal work \cite{AmbrosioPDE} (see also the more recent contributions of \cite{AmbrosioC2014}). In what follows, we denote by $\Sigma_T := C^0([0,T],\R^d)$ the space of continuous arcs from $[0,T]$ into $\R^d$, by $e_t : (x,\sigma) \in \R^d \times \Sigma_T \mapsto \sigma(t) \in \R^d$ the so-called \textit{evaluation map} defined for all times $t \in [0,T]$, and by $\pi_{\R^d} : \R^d \times \Sigma_T \rightarrow \R^d$ the projection operator onto the first factor.

\begin{thm}[Superposition principle]
\label{thm:Superposition}
Let $\vb : [0,T] \times \R^d \rightarrow \R^d$ be a Lebesgue-Borel velocity field, $(\tau,\mu_{\tau}) \in [0,T] \times \Pcal_c(\R^d)$ and $\mu(\cdot) \in C^0([0,T],\Pcal_c(\R^d))$ be a curve of measures such that 
\begin{equation*}
\INTSeg{\INTDom{\frac{|\vb(t,x)|}{1+|x|}}{\R^d}{\mu(t)(x)}}{\tau}{0}{T} < +\infty. 
\end{equation*}
Then, the curve $\mu(\cdot)$ is a solution of the Cauchy problem
\begin{equation*}
\left\{
\begin{aligned}
& \partial_t \mu(t) + \Div_x \big( \vb(t) \mu(t) \big) = 0,\\
& \mu(\tau) = \mu_{\tau}, 
\end{aligned}
\right.
\end{equation*}
if and only if there exists a \textnormal{superposition measure} $\Beta \in \Pcal(\R^d \times \Sigma_T)$ concentrated on the sets of pairs $(x,\sigma) \in \R^d \times \AC([0,T],\R^d)$ satisfying
\begin{equation*}
\hspace{1.2cm} \sigma(\tau) = x \qquad \text{and} \qquad \dot{\sigma}(t) = \vb(t,\sigma(t)),  
\end{equation*}
for $\Lcal^1$-almost every $t \in [\tau,T]$, and such that
\begin{equation*}
(\pi_{\R^d})_{\#} \Beta = \mu_{\tau} \qquad \text{and} \qquad (e_t)_{\#} \Beta = \mu(t), 
\end{equation*}
for all times $t \in [\tau,T]$.
\end{thm}

We end this section on continuity equations by recalling a useful structural result from \cite{ContInc}, which plays a central role in several of the arguments of Section \ref{section:Semiconcavity}.

\begin{prop}[Superposition measures producing optimal plans]
\label{prop:SuperpositionPlan}
Given $r>0$, let $\mu_1,\mu_2 \in \Pcal(B(0,r))$, fix a time $\tau \in [0,T]$ and consider two non-local velocity fields $v_1,v_2 : [0,T] \times \Pcal_c(\R^d) \times \R^d \rightarrow \R^d$ satisfying hypotheses \ref{hyp:CE}-$(i),(ii)$. For $\imath \in \{1,2\}$, denote by $\mu_\imath(\cdot)$ the solution of 
\begin{equation*}
\left\{ 
\begin{aligned}
& \partial_t \mu_{\imath}(t) + \Div_x \big( v_{\imath}(t,\mu_\imath(t)) \mu_{\imath}(t) \big) = 0, \\
& \mu_{\imath}(\tau) = \mu_{\imath}, 
\end{aligned}
\right.
\end{equation*}
and by $\Beta_{\imath} \in \Pcal(\R^d \times \Sigma_T)$ a superposition measure associated with $\mu_{\imath}(\cdot)$ via Theorem \ref{thm:Superposition}. 

Then, for every $p$-optimal transport plan $\gamma \in \Gamma_o(\mu_1,\mu_2)$, there exists $\hat{\Beta}_{12} \in \Gamma(\Beta_1,\Beta_2)$ such that 
\begin{equation*}
(\pi^1_{\R^d},\pi^2_{\R^d})_{\#} \hat{\Beta}_{12} = \gamma \qquad \text{and} \qquad (e^1_t,e^2_t)_{\#} \hat{\Beta}_{12} \in \Gamma_o(\mu_1(t),\mu_2(t)), 
\end{equation*}
for all times $t \in [\tau,T]$, where for $\imath \in \{1,2\}$ the maps 
\begin{equation*}
\pi^{\imath}_{\R^d} : \big( \R^d \times \Sigma_T \big) \times \big( \R^d \times \Sigma_T \big) \rightarrow \R^d \quad \text{and} \quad e^{\imath}_t : \big( \R^d \times \Sigma_T \big) \times \big( \R^d \times \Sigma_T \big) \rightarrow \R^d, 
\end{equation*}
stand for the projection onto $\R^d$ and the evaluation map restricted to the $\imath$-th factor respectively.
\end{prop}


\subsection{Optimal control in Wasserstein spaces}
\label{subsection:MFOC}

In this last preliminary section, we recollect known facts about optimal control problems formulated on controlled non-local continuity equations. We refer the reader to  \cite{PMPWassConst,ContInc,SetValuedPMP,PMPWass,LipReg} for a detailed account on this topics  (see also  \cite{Bensoussan2013,Carmona2018,CavagnariMP2018,Cavagnari2020,Jimenez2020} for complementary results).

The theory of optimal control is usually developed on \textit{Bolza type} problems inspired by the calculus of variations, which in the absence of constraints can be written in our context as
\begin{equation*}
(\Ppazo_L) ~ \left\{
\begin{aligned}
\min_{u(\cdot) \in \U} & \, \bigg[ \INTSeg{L(t,\mu(t),u(t))}{t}{0}{T} + \varphi(\mu(T)) \bigg] \\
\text{s.t.} ~ & \left\{
\begin{aligned}
& \partial_t \mu(t) + \Div_x \big( v(t,\mu(t),u(t)) \mu(t) \big) = 0, \\
& \mu(0) = \mu^0.
\end{aligned}
\right.
\end{aligned}
\right.
\end{equation*}
Here, the set of admissible controls is defined as $\U := \{ u : [0,T] \rightarrow  U ~\text{s.t. $u(\cdot)$ is $\Lcal^1$-measurable} \}$, where $(U,d_U)$ is a compact metric space. The dynamics is driven by the controlled non-local velocity field $v : [0,T] \times \Pcal_c(\R^d) \times U \times \R^d \rightarrow \R^d$, while the mappings $L : [0,T] \times \Pcal_c(\R^d) \times U \rightarrow \R$ and $\varphi : \Pcal_c(\R^d) \rightarrow \R$ are running and final cost functionals respectively. The following sets of assumptions are quite common when studying smooth unconstrained Bolza problems. 
 
\begin{taggedhyp}{\textbn{(OCP)}}
\label{hyp:OCP}
For every $R > 0$, assume that the following holds with $K := B(0,R)$.
\begin{enumerate}
\item[$(i)$] The non-local velocity field $(t,\mu,x) \in [0,T] \times \Pcal_c(\R^d) \times \R^d \mapsto v(t,\mu,u,x) \in \R^d$ satisfies hypotheses \ref{hyp:CE} with constants that are uniform with respect to $u \in U$. Moreover, the map $u \in U \mapsto v(t,\mu,u,x) \in \R^d$ is continuous for $\Lcal^1$-almost every $t \in [0,T]$ and any $(\mu,x) \in \Pcal_c(\R^d) \times \R^d$.
\item[$(ii)$] The final cost $\varphi : \Pcal_c(\R^d) \rightarrow \R$ is Lipschitz in the $W_1$-metric over $\Pcal(K)$ and locally differentiable. Moreover, the map $x \in \R^d \mapsto \nabla \varphi(\mu)(x) \in \R^d$ is continuous for every $\mu \in \Pcal_c(\R^d)$. 
\end{enumerate}
\end{taggedhyp}

\begin{taggedhypsing}{\textbn{(L)}}
\label{hyp:L}
For every $R > 0$, assume that the following holds with $K := B(0,R)$.  \vspace{0.15cm} \\
The running cost $(t,\mu,u) \in [0,T] \times \Pcal_c(\R^d) \times U \mapsto L(t,\mu,u) \in \R$ is $\Lcal^1$-measurable with respect to $t \in [0,T]$ and continuous with respect to $u \in U$. Moreover, there exists a map $\ell_K(\cdot) \in L^1([0,T],\R_+)$ such that 
\begin{equation*}
|L(t,\nu,u) - L(t,\mu,u)| \leq \ell_K(t)W_1(\mu,\nu), 
\end{equation*}
for $\Lcal^1$-almost every $t \in [0,T]$, all $\mu,\nu \in \Pcal(K)$ and each $u \in U$, and $t \in [0,T] \mapsto \sup_{u \in U} |L(t,\delta_0,u)| \in \R$ is Lebesgue integrable. In addition, the application $\mu \in \Pcal_c(\R^d) \mapsto L(t,\mu,u) \in \R$ is locally differentiable for $\Lcal^1$-almost every $t \in [0,T]$ and all $(\mu,u) \in \Pcal_c(\R^d) \times U$, with continuous gradient $x \in \R^d \rightarrow \nabla_{\mu} L(t,\mu,u)(x) \in \R^d$.
\end{taggedhypsing}

\begin{rmk}[Concerning hypotheses \ref{hyp:OCP} and \ref{hyp:L}]
In the classical theory of the calculus of variations, it is customarily not assumed that the right-hand side of the dynamics is sublinear. Instead, one usually imposes a suitable growth condition at infinity with respect to the velocity variables in the running cost. In this spirit, we stress that all the results exposed in the present paper still hold true if the sublinearity hypothesis \ref{hyp:CE}-$(i)$ on the controlled velocity field is replaced by a Tonelli-type condition on the running cost of the optimal control problem (see e.g. \cite[Chapter 16]{Clarke}). 
\end{rmk}

It is a well-known fact in optimal control that under hypotheses \ref{hyp:OCP} and \ref{hyp:L}, any Bolza problem of the form $(\Ppazo_L)$ can be equivalently rewritten as a \textit{Mayer problem} -- that is a problem with $L(t,\mu,u) \equiv 0$ --, formulated on an extended state space. We refer to \cite[Section 4.2]{SetValuedPMP} for such an adaptation in the setting of mean-field optimal control problems. In this context, all the results derived for Mayer problems can in turn be transposed to Bolza problems. Hence without loss of generality, we will work in the sequel on the Mayer problem 
\begin{equation*}
(\Ppazo) ~ \left\{
\begin{aligned}
\min_{u(\cdot) \in \U} & \, \big[ \varphi(\mu(T)) \big] \\
\text{s.t.} ~ & \left\{
\begin{aligned}
& \partial_t \mu(t) + \Div_x \big( v(t,\mu(t),u(t)) \mu(t) \big) = 0, \\
& \mu(0) = \mu^0,
\end{aligned}
\right.
\end{aligned}
\right.
\end{equation*}
and systematically assume that its data satisfy hypotheses \ref{hyp:OCP}. 

\begin{Def}[Admissible pairs and strong-local minimisers] \hspace{-0.15cm}
A trajectory-control pair $(\mu(\cdot),u(\cdot))$ is said to be \textnormal{admissible} for $(\Ppazo)$ if it solves the controlled Cauchy problem
\begin{equation*}
\left\{
\begin{aligned}
& \partial_t \mu(t) + \Div_x \Big( v(t,\mu(t),u(t)) \mu(t) \Big) = 0, \\
& \mu(0) = \mu^0.
\end{aligned}
\right.
\end{equation*}
Moreover, we say that $(\mu^*(\cdot),u^*(\cdot))$ is a \textnormal{strong local minimiser} for $(\Ppazo)$ if there exists $\epsilon > 0$ such that
\begin{equation*}
\varphi(\mu^*(T)) \leq \varphi(\mu(T)),
\end{equation*}
for every admissible pair $(\mu(\cdot),u(\cdot))$ satisfying $\sup_{t \in [0,T]} W_1(\mu^*(t),\mu(t)) \leq \epsilon$. 
\end{Def}

As a consequence of hypothesis \ref{hyp:OCP}-$(i)$ and Theorem \ref{thm:NonLocalCE}, we have the following useful lemma which provides uniform regularity and support estimates on admissible trajectories for $(\Ppazo)$.

\begin{lem}[Uniform estimates on admissible trajectories]
\label{lem:AdmEst}
Let $\mu \in \Pcal(B(0,r))$ for some $r > 0$, and assume that hypotheses \ref{hyp:OCP} hold. Then, there exist $R_r > 0$ and $m_r(\cdot) \in L^1([0,T],\R_+)$ such that 
\begin{equation*}
\supp(\mu(t)) \subset K := B(0,R_r) \qquad \text{and} \qquad W_1(\mu(\tau),\mu(t)) \leq \INTSeg{m_r(s)}{s}{\tau}{t},
\end{equation*}
for all times $0 \leq \tau \leq t \leq T$, whenever $\mu(\cdot)$ is a solution of the Cauchy problem
\begin{equation*}
\left\{
\begin{aligned}
& \partial_t \mu(t) + \Div_x \Big( v(t,\mu(t),u(t)) \mu(t) \Big) = 0, \\
& \mu(s) = \mu.
\end{aligned}
\right.
\end{equation*}
for any given $s \in [0,T]$ and $u(\cdot) \in \U$. 
\end{lem}

We end this section by recalling the main result of \cite{SetValuedPMP}, which is an adaptation of the celebrated \textit{Pontryagin Maximum Principle} (``PMP'' for short) to problem $(\Ppazo)$. Its statement involves the \textit{symplectic matrix} $\J_{2d}$, given by 
\begin{equation*}
\J_{2d} := \begin{pmatrix}
0 && \Id \\ - \Id && 0
\end{pmatrix},
\end{equation*}
and the \textit{Hamiltonian} $\H : [0,T] \times \Pcal_c(\R^{2d}) \times U \rightarrow \R$ associated with $(\Ppazo)$, defined as
\begin{equation}
\label{eq:PMP_Hamiltonian}
\H(t,\nu,u) := \INTDom{\langle r , v(t,\pi^1_{\#} \nu,u,x) \rangle}{\R^{2d}}{\nu(x,r)},
\end{equation}
for $\Lcal^1$-almost every $t \in [0,T]$ and any $(\nu,u) \in \Pcal_c(\R^{2d}) \times U$. 

\begin{thm}[Pontryagin Maximum Principle for $(\Ppazo)$]
\label{thm:PMP}
Let $\mu^0 \in \Pcal(B(0,r))$ for some $r >0$, assume that hypotheses \ref{hyp:OCP} hold and let $(\mu^*(\cdot),u^*(\cdot)) \in \AC([0,T],\Pcal_1(K)) \times \U$ be a strong local minimiser for $(\Ppazo)$, where $K := B(0,R_r)$ is as in Lemma \ref{lem:AdmEst}. 

Then, there exists a \textnormal{state-costate curve} $t \in [0,T] \mapsto \nu^*(t) \in \Pcal_c(\R^{2d})$ satisfying the followings.
\begin{enumerate}
\item[$(i)$] There exist $R'_r > 0$ and $m'_r(\cdot) \in L^1([0,T],\R_+)$, depending only on $r>0$ and the regularity constants in \ref{hyp:CE}, such that 
\begin{equation}
\label{eq:PMP_Est}
\supp(\nu^*(t)) \subset K' \times K' \qquad \text{and} \qquad W_1(\nu^*(\tau),\nu^*(t)) \leq \INTSeg{m_r'(s)}{s}{\tau}{t},
\end{equation}
for all times $0 \leq \tau \leq t \leq T$, where $K' := B(0,R_r')$.
\item[$(ii)$] The curve $\nu^*(\cdot)$ is a solution of the \textnormal{forward-backward Hamiltonian} continuity equation
\begin{equation}
\label{eq:PMP_Dynamics}
\left\{
\begin{aligned}
&  \, \partial_t \nu^*(t) + \Div_{(x,r)} \big( \J_{2d} \nabla_{\nu} \H(t,\nu^*(t),u^*(t)) \nu^*(t) \big) = 0, \\
& \pi^1_{\#} \nu^*(t) = \mu^*(t) \hspace{2.1cm} \text{for all times $t \in [0,T]$}, \\
& \nu^*(T) = \Big( \Id , -\nabla \varphi(\mu^*(T)) \Big)_{\raisebox{4pt}{$\scriptstyle{\#}$}} \mu^*(T),
\end{aligned}
\right.
\end{equation}
where the Wasserstein gradient of the Hamiltonian is given explicitly by 
\begin{equation}
\label{eq:PMP_HamiltonianGrad}
\begin{aligned}
& \nabla_{\nu} \H(t,\nu^*(t),u^*(t))(x,r) \\
& \hspace{1cm} = \begin{pmatrix}
\, \D_x v \big( t,\mu^*(t),u^*(t),x \big)^{\top} r + \INTDom{\D_{\mu} v \big( t,\mu^*(t),u^*(t),y \big)(x)^{\top} p \,}{\R^{2d}}{\nu^*(t)(y,p)} \\ \\
v \big( t,\mu^*(t),u^*(t),x \big)
\end{pmatrix},
\end{aligned}
\end{equation}
for $\Lcal^1$-almost every $t \in [0,T]$ and any $(x,r) \in \R^{2d}$. 
\item[$(iii)$] The \textnormal{maximisation condition}
\begin{equation}
\label{eq:PMP_Maximisation}
\H(t,\nu^*(t),u^*(t)) = \max_{u \in U} \, \H(t,\nu^*(t),u), 
\end{equation}
holds for $\Lcal^1$-almost every $t \in [0,T]$. 
\end{enumerate} 
\end{thm}


\section{Semiconcavity of the value function}
\label{section:Semiconcavity}

In this section, we investigate fine regularity properties of the \textit{value function} $\Vcal : [0,T] \times \Pcal_c(\R^d) \rightarrow \R$ associated with $(\Ppazo)$, which is defined by  
\begin{equation}
\label{eq:ValueFunction}
\Vcal(\tau,\mu_{\tau}) := \left\{
\begin{aligned}
\inf_{u(\cdot) \in \U} & \, [\varphi(\mu(T))] \\
\text{s.t.} ~ & \left\{
\begin{aligned}
& \partial_t \mu(t) + \Div_x \big( v(t,\mu(t),u(t)) \mu(t) \big) = 0, \\
& \mu(\tau) = \mu_{\tau}, 
\end{aligned}
\right.
\end{aligned}
\right.
\end{equation}
for any $(\tau,\mu_{\tau}) \in [0,T] \times \Pcal_c(\R^d)$. In what follows, we will assume that hypotheses \ref{hyp:OCP} of Section \ref{subsection:MFOC} hold. We start by an elementary -- but nonetheless important -- regularity result for the value function. 

\begin{prop}[Absolute continuity of the value function]
\label{prop:LipRegValue}
For every compact set $K \subset \R^d$, there exist a constant $\Lpazo_K > 0$ and a map $\M_K(\cdot) \in L^1([0,T],\R_+)$, both depending only on $K$, such that 
\begin{equation*}
|\Vcal(\tau_1,\mu_1) - \Vcal(\tau_2,\mu_2)| \leq \INTSeg{\M_K(t)}{t}{\tau_1}{\tau_2} + \Lpazo_K W_1(\mu_1,\mu_2), 
\end{equation*}
for every $0 \leq \tau_1 \leq \tau_2 \leq T$ and all $\mu_1,\mu_2 \in \Pcal(K)$.  
\end{prop}

\begin{proof}
The proof of this result is a standard consequence of \eqref{eq:SuppAC_Est}, together with the local Lipschitz dependence of admissible curves on their initial conditions (see e.g. \cite[Theorem 4]{ContInc}) and the $W_1$-Lipschitz regularity of $\varphi(\cdot)$ over sets of uniformly compactly supported measures. 
\end{proof}

Our aim throughout this section is to derive subtler \textit{semiconcavity} properties of the value function. As amply discussed in the introduction, these latter imply many useful structure results on optimal trajectories, for which we mainly refer to \cite{Cannarsa1991,CannarsaS2004}. 


\subsection{Two local semiconcavity results with respect to the measure variable}
\label{subsection:SemiconcavityEst}

In this first section, we study semiconcavity properties of the value function associated with $(\Ppazo)$ with respect to its second argument. To this end, we introduce below a localised notion of semiconcavity along displacement interpolating curves between compactly supported measures, in the spirit of \cite[Chapter 9]{AGS} (see also the seminal work \cite{McCann1997}). 

\begin{Def}[Local geodesic and strong semiconcavity]
\label{def:Semiconcavity}
We say that a functional $\phi : \Pcal_c(\R^d) \rightarrow \R$ is \textnormal{locally geodesically semiconcave} if for every $K := B(0,R)$ with $R>0$, there exists a constant $\Cpazo_K > 0$ such that for any $\mu_1,\mu_2 \in \Pcal(K)$ and each $\gamma \in \Gamma_o(\mu_1,\mu_2)$, it holds
\begin{equation}
\label{eq:SemiconcavityDef1}
(1-\lambda) \phi(\mu_1) + \lambda \phi(\mu_2) - \phi(\gamma^{1 \rightarrow 2}_{\lambda}) \leq  \Cpazo_K \, \lambda(1-\lambda)W_2^2(\mu_1,\mu_2), 
\end{equation}
for every $\lambda \in [0,1]$, where $\gamma^{1 \rightarrow 2}_{\lambda} := ( (1-\lambda)\pi^1 + \lambda \pi^2)_{\#} \gamma$. 

Similarly, we say that $\phi(\cdot)$ is \textnormal{locally strongly semiconcave} if for any $\mu_1,\mu_2 \in \Pcal(K)$ and all $\Bmu \in \Gamma(\mu_1,\mu_2)$, one has
\begin{equation}
\label{eq:SemiconcavityDef2}
(1-\lambda) \phi(\mu_1) + \lambda \phi(\mu_2) - \phi(\Bmu^{1 \rightarrow 2}_{\lambda}) \leq  \Cpazo_K \, \lambda(1-\lambda)W^2_{2,\Bmu}(\mu_1,\mu_2), 
\end{equation}
for every $\lambda \in [0,1]$, with $\Bmu^{1\rightarrow2}_{\lambda} := ( (1-\lambda)\pi^1 + \lambda \pi^2)_{\#} \Bmu$.
\end{Def}

\begin{rmk}[General semiconcavity moduli]
More generally (see e.g. \cite{Cannarsa1991} or \cite[Chapter 2]{CannarsaS2004}), the notions of local semiconcavity proposed in Definition \ref{def:Semiconcavity} could be formulated as 
\begin{equation*}
(1-\lambda) \phi(\mu_1) + \lambda \phi(\mu_2) - \phi(\Bmu^{1 \rightarrow 2}_{\lambda}) \leq  \Cpazo_K \, \lambda(1-\lambda)W_{2,\Bmu}(\mu_1,\mu_2) \omega_K  \big( W_{2,\Bmu}(\mu_1,\mu_2) \big),
\end{equation*}
where $\omega_K : \R_+ \rightarrow \R_+$ is a non-decreasing modulus of continuity such that $\omega_K(r) \rightarrow 0$ as $r \rightarrow 0^+$. For the sake of simplicity, we will only consider the case $\omega_K(r) := \Cpazo_K \, r$ in the present paper. Indeed, the analysis of general moduli is very similar to that of linear ones -- albeit being slightly more cumbersome --, and the corresponding results are of matching depth. 
\end{rmk}

Based on the notion of local geodesic semiconcavity, we make the following additional assumptions on the dynamics and cost functionals of problem $(\Ppazo)$. 

\begin{taggedhyp}{\textbn{(SC1)}}
\label{hyp:SC1}
Suppose that for any $R > 0$, the following holds with $K:= B(0,R)$. 
\begin{enumerate}
\item[$(i)$] There exists a map $\Cpazo_K(\cdot) \in L^1([0,T],\R_+)$ such that for every $u \in U$, any $\mu_1,\mu_2 \in \Pcal(K)$, each $\gamma \in \Gamma_o(\mu_1,\mu_2)$ and all $x,y \in K$, it holds
\begin{equation*}
\Big| (1-\lambda) v(t,\mu_1,u,x) + \lambda v(t,\mu_2,u,y) - v \big( t , \gamma^{1 \rightarrow 2}_{\lambda},u,(1-\lambda)x + \lambda y \big) \Big| \leq \Cpazo_K(t) \, \lambda(1-\lambda) \Big( |x-y|^2 + W_2^2(\mu_1,\mu_2) \Big), 
\end{equation*}
for $\Lcal^1$-almost every $t \in [0,T]$ and any $\lambda \in [0,1]$.
\item[$(ii)$] The final cost $\varphi : \Pcal_c(\R^d) \rightarrow \R$ is locally geodesically semiconcave. 
\end{enumerate}
\end{taggedhyp}

In the following theorem, we state one of our main results concerning the value function defined in \eqref{eq:ValueFunction}, that is a local geodesic semiconcavity property with respect to the measure variable. 

\begin{thm}[Local geodesic semiconcavity of the value function]
\label{thm:Semiconcavity1}
Suppose that hypotheses \ref{hyp:OCP} and \ref{hyp:SC1} hold. Then for any $K := B(0,r)$ with $r > 0$, there exists a constant $\Ccal_K > 0$ such that for every $\tau \in [0,T]$ and any $\mu_1,\mu_2 \in \Pcal(K)$, it holds
\begin{equation*}
(1-\lambda) \Vcal(\tau,\mu_1) + \lambda \Vcal(\tau,\mu_2) - \Vcal(\tau,\gamma^{1 \rightarrow 2}_{\lambda}) \leq \Ccal_K \, \lambda(1-\lambda) W_2^2(\mu_1,\mu_2),
\end{equation*}
for all $\lambda \in [0,1]$ and each $\gamma \in \Gamma_o(\mu_1,\mu_2)$.
\end{thm}

We split the proof of Theorem \ref{thm:Semiconcavity1} into two steps. We first prove in Step 1 a general interpolation inequality between curves of measures generated by non-local flows, and use this result in Step 2 to recover the geodesic semiconcavity of the value function. In what follows given $u(\cdot) \in \U$, we denote by $(\Phi_{(\tau,t)}^u[\mu](\cdot))_{t \in [0,T]}$ the non-local flows starting from $\mu \in \Pcal(B(0,r))$ at time $\tau \in [0,T]$ and generated by $\vb : (t,\mu,x) \in [0,T] \times \Pcal_c(\R^d) \times \R^d \mapsto v(t,\mu,u(t),x) \in \R^d$, in the sense of Definition \ref{def:NonlocalFlows}.


\paragraph*{Step 1: An interpolation comparison between non-local flows} In this first step, we prove a general distance estimate between, on the one hand, solutions of non-local continuity equations starting from an interpolated initial datum, and, on the other hand, interpolations of solutions starting from the end-points of this interpolated initial datum. 

\begin{lem}[A general interpolation inequality]
\label{lem:GeodesicInterpolationFlows}
Assume that hypotheses \ref{hyp:OCP} and \ref{hyp:SC1} hold. Then for every $K := B(0,r)$ with $r > 0$, there exists a constant $\Spazo_K > 0$ such that for every $\tau \in [0,T]$, any $u(\cdot) \in \U$, all $\mu_1,\mu_2 \in \Pcal(K)$ and each $\gamma \in \Gamma_o(\mu_1,\mu_2)$, there exists a continuous curve of $2$-optimal plans $t \in [\tau,T] \mapsto \gamma(t) \in \Gamma_o(\mu_1(t),\mu_2(t))$ such that for all times $t \in [\tau,T]$ and every $\lambda \in [0,1]$, it holds
\begin{equation}
\label{eq:GeodesicInterpolationFlows}
W_1 \Big(\gamma^{1 \rightarrow 2}_{\lambda}(t) , \Phi_{(\tau,t)}^u[\gamma^{1 \rightarrow 2}_{\lambda}](\cdot)_{\#} \gamma^{1 \rightarrow 2}_{\lambda} \Big) \leq \Spazo_r \, \lambda(1-\lambda) W_2^2(\mu_1,\mu_2).
\end{equation}
Here, the curves of measures $\mu_1(\cdot),\mu_2(\cdot)$ are defined by 
\begin{equation*}
\mu_1(t) := \Phi^u_{(\tau,t)}[\mu_1](\cdot)_{\#} \mu_1 \qquad \text{and} \qquad \mu_2(t) := \Phi^u_{(\tau,t)}[\mu_2](\cdot)_{\#} \mu_2, 
\end{equation*}
for all times $t \in [\tau,T]$. 
\end{lem}

\begin{proof}
Observe that by \eqref{eq:FlowRep} of Theorem \ref{thm:NonLocalCE} and Lemma \ref{lem:AdmEst}, there exists $R_r \geq r > 0$ such that $\mu_1(\cdot),\mu_2(\cdot) \in \AC([\tau,T],\Pcal_1(K))$ up to redefining $K := B(0,R_r)$. Considering two superposition measures $\Beta_1,\Beta_2 \in \Pcal(\R^d \times \Sigma_T)$ associated with $\mu_1(\cdot),\mu_2(\cdot)$ respectively via Theorem \ref{thm:Superposition}, there exists by Proposition \ref{prop:SuperpositionPlan} a transport plan $\hat{\Beta}_{12} \in \Gamma(\Beta_1,\Beta_2)$ such that for all times $t \in [\tau,T]$, it holds
\begin{equation*}
(\pi_{\R^d}^1,\pi_{\R^d}^2)_{\#} \hat{\Beta}_{12} = \gamma \qquad \text{and} \qquad (e_t^1,e_t^2)_{\#} \hat{\Beta}_{12} =: \gamma(t) \in \Gamma_o(\mu_1(t),\mu_2(t)). 
\end{equation*}
In particular, owing to the definition of the optimal plans $\gamma(t) \in \Gamma_o(\mu_1(t),\mu_2(t))$, the measures
\begin{equation*}
\hat{\Bmu}_{\lambda}^{1 \rightarrow 2}(t) := \Big( (1-\lambda)e_t^1 +\lambda e_t^2 \, , \, \Phi^u_{(\tau,t)}[\gamma^{1 \rightarrow 2}_{\lambda}] \circ ((1-\lambda) \pi^1_{\R^d} + \lambda \pi^2_{\R^d}) \Big)_{\raisebox{4pt}{$\scriptstyle{\#}$}} \hat{\Beta}_{12}, 
\end{equation*}
are admissible transport plans between $\gamma^{1 \rightarrow 2}_{\lambda}(t)$ and $\Phi^u_{(\tau,t)}[\gamma^{1 \rightarrow 2}_{\lambda}](\cdot)_{\#} \gamma^{1 \rightarrow 2}_{\lambda}$ for all times $t \in [\tau,T]$ and every $\lambda \in [0,1]$. Recall also that under hypotheses \ref{hyp:CE}, the superposition measures $\Beta_1,\Beta_2 \in \Pcal(\R^d \times \Sigma_T)$ are concentrated on pairs $(x,\sigma_1),(y,\sigma_2) \in \R^d \times \Sigma_T$ of the form
\begin{equation*}
\sigma_1(t) = \Phi_{(\tau,t)}^u[\mu_1](x)  \qquad \text{and} \qquad \sigma_2(t) = \Phi_{(\tau,t)}^u[\mu_2](y),  
\end{equation*}
for all times $t \in [\tau,T]$ and any $(x,y) \in \supp(\gamma)$. This allows us to derive the following distance estimate 
\begin{equation}
\label{eq:CrossedInterpolationEst}
\begin{aligned}
W_1 \Big(\gamma^{1 \rightarrow 2}_{\lambda}&(t) , \Phi^u_{(\tau,t)}[\gamma^{1 \rightarrow 2}_{\lambda}](\cdot)_{\#} \gamma^{1 \rightarrow 2}_{\lambda} \Big) \\
& \leq \INTDom{|\xb - \yb|}{\R^{2d}}{\hat{\Bmu}_{\lambda}^{1 \rightarrow 2}(t)(\xb,\yb)} \\
& = \INTDom{\Big| (1-\lambda)\sigma_1(s) + \lambda \sigma_2(s) - \Phi^u_{(\tau,t)}[\gamma^{1 \rightarrow 2}_{\lambda}]\Big( (1-\lambda) x + \lambda y) \Big) \Big|}{(\R^d \times \Sigma_T)^2}{\hat{\Beta}_{12}(x,\sigma_1,y,\sigma_2)} \\
& = \INTDom{\Big| (1-\lambda) \Phi^u_{(\tau,t)}[\mu_1](x) + \lambda \Phi^u_{(\tau,t)}[\mu_2](y) - \Phi^u_{(\tau,t)}[\gamma^{1 \rightarrow 2}_{\lambda}]\Big( (1-\lambda) x + \lambda y) \Big) \Big|}{\R^{2d}}{\gamma(x,y)}, 
\end{aligned}
\end{equation}
for all times $t \in [\tau,T]$ and every $\lambda \in [0,1]$.

Our goal is now to estimate the right-hand side of \eqref{eq:CrossedInterpolationEst}. Recalling the characterisation \eqref{eq:NonLocalFlow_Def} of non-local flows, one has
\begin{equation}
\label{eq:SemiconcavityFlow1}
\begin{aligned}
\Big| (1&-\lambda) \Phi^u_{(\tau,t)}[\mu_1](x) + \lambda \Phi^u_{(\tau,t)}[\mu_2](y) - \Phi^u_{(\tau,t)}[\gamma^{1 \rightarrow 2}_{\lambda}]((1-\lambda)x + \lambda y) \Big| \\
& \leq \int_{\tau}^t \Big| (1-\lambda) v \Big(s,\mu_1(s),u(s),\Phi_{(\tau,s)}^u[\mu_1](x) \Big) + \lambda v \Big(s,\mu_2(s),u(s),\Phi_{(\tau,s)}^u[\mu_2](y) \Big) \\
& \hspace{5cm} - v \Big(s, \gamma^{1 \rightarrow 2}_{\lambda}(s) ,u(s), (1-\lambda)\Phi^u_{(\tau,s)}[\mu_1](x) + \lambda \Phi^u_{(\tau,s)}[\mu_2](y) \Big) \Big| \textnormal{d}s \\
& \hspace{0.4cm} + \int_{\tau}^t \Big| v \Big(s, \gamma^{1 \rightarrow 2}_{\lambda}(s) ,u(s), (1-\lambda)\Phi^u_{(\tau,s)}[\mu_1](x) + \lambda \Phi^u_{(\tau,s)}[\mu_2](y) \Big) \\ 
& \hspace{4.2cm} - v \Big(s,\Phi^u_{(\tau,s)}[\gamma^{1 \rightarrow 2}_{\lambda}](\cdot)_{\#} \gamma^{1 \rightarrow 2}_{\lambda} ,u(s), \Phi^u_{(\tau,s)}[\gamma^{1 \rightarrow 2}_{\lambda}]((1-\lambda)x + \lambda y) \Big) \Big|  \textnormal{d}s
\end{aligned}
\end{equation}
for any $x,y \in K$. Observing that $(\gamma^{1 \rightarrow 2}_{\lambda}(s))_{\lambda \in [0,1]}$ is an interpolating curve between $\mu_1(s)$ and $\mu_2(s)$ with $\gamma(s) \in \Gamma_o(\mu_1(s),\mu_2(s))$ for all times $s \in [\tau,t]$, it further holds by \ref{hyp:SC1}-$(i)$ that
\begin{equation}
\label{eq:SemiconcavityFlow2}
\begin{aligned}
& \Big| (1-\lambda) v \Big(s,\mu_1(s),u(s),\Phi_{(\tau,s)}^u[\mu_1](x) \Big) + \lambda v \Big(s,\mu_2(s),u(s),\Phi_{(\tau,s)}^u[\mu_2](y) \Big) \\
& \hspace{4.45cm} - v \Big(s, \gamma^{1 \rightarrow 2}_{\lambda}(s) ,u(s), (1-\lambda)\Phi^u_{(\tau,s)}[\mu_1](x) + \lambda \Phi^u_{(\tau,s)}[\mu_2](y) \Big) \Big| \\
& \hspace{2.6cm} \leq \Cpazo_K(s) \, \lambda(1-\lambda) \Big( \big| \Phi_{(\tau,s)}^u[\mu_1](x) -\Phi_{(\tau,s)}^u[\mu_2](y)\big|^2 + W_2^2(\mu_1(s),\mu_2(s)) \Big) \\
& \hspace{2.6cm} \leq \Cpazo_K'(s) \, \lambda(1-\lambda) \Big( |x-y|^2 + W_2^2(\mu_1,\mu_2) \Big),
\end{aligned}
\end{equation}
where the map $\Cpazo_K'(\cdot) \in L^1([0,T],\R_+)$ is given by 
\begin{equation*}
\Cpazo'_K(s) := 4 \sup_{\tau,t \in [0,T]} \hspace{-0.1cm} \Big( \max_{\mu \in \Pcal(K)} \Lip \big( \Phi_{(\tau,t)}[\mu](\cdot); K \big)^2 + \max_{x \in K} \, \Lip \big( \Phi_{(\tau,t)}[\cdot](x); \Pcal_1(K) \big)^2 \, \Big) \Cpazo_K(s), 
\end{equation*}
for $\Lcal^1$-almost every $s \in [\tau,T]$. As a consequence of \ref{hyp:CE}-$(ii)$, one also has 
\begin{equation}
\label{eq:SemiconcavityFlow3}
\begin{aligned}
& \Big| v \Big(s, \gamma^{1 \rightarrow 2}_{\lambda}(s) ,u(s), (1-\lambda)\Phi^u_{(\tau,s)}[\mu_1](x) + \lambda \Phi^u_{(\tau,s)}[\mu_2](y) \Big) \\ 
& \hspace{3.625cm} - v \Big(s,\Phi^u_{(\tau,s)}[\gamma^{1 \rightarrow 2}_{\lambda}](\cdot)_{\#} \gamma^{1 \rightarrow 2}_{\lambda} ,u(s), \Phi^u_{(\tau,s)}[\gamma^{1 \rightarrow 2}_{\lambda}]((1-\lambda)x + \lambda y) \Big) \Big| \\
& \hspace{2.4cm} \leq l_K(s) \Big| (1-\lambda)\Phi^u_{(\tau,s)}[\mu_1](x) + \lambda \Phi^u_{(\tau,s)}[\mu_2](y) - \Phi^u_{(\tau,s)}[\gamma^{1 \rightarrow 2}_{\lambda}]((1-\lambda)x + \lambda y) \Big| \\
& \hspace{2.8cm} + L_K(s) W_1 \Big(\gamma^{1 \rightarrow 2}_{\lambda}(s) , \Phi^u_{(\tau,s)}[\gamma^{1 \rightarrow 2}_{\lambda}](\cdot)_{\#} \gamma^{1 \rightarrow 2}_{\lambda} \Big), 
\end{aligned}
\end{equation}
for $\Lcal^1$-almost every $s \in [\tau,T]$. Whence, by merging \eqref{eq:CrossedInterpolationEst}, \eqref{eq:SemiconcavityFlow1}, \eqref{eq:SemiconcavityFlow2} and \eqref{eq:SemiconcavityFlow3} while integrating the resulting estimate against $\gamma \in \Gamma_o(\mu_1,\mu_2)$ and applying Fubini's theorem, we obtain
\begin{equation*}
\begin{aligned}
& \INTDom{\Big| (1-\lambda) \Phi^u_{(\tau,t)}[\mu_1](x) + \lambda \Phi^u_{(\tau,t)}[\mu_2](y) - \Phi^u_{(\tau,t)}[\gamma^{1 \rightarrow 2}_{\lambda}]((1-\lambda)x + \lambda y) \Big|}{\R^{2d}}{\gamma(x,y)} \\
& \leq \bigg( 2 \INTSeg{\Cpazo_K'(s)}{s}{\tau}{t} \bigg) \lambda(1-\lambda) W_2^2(\mu_1,\mu_2) \\
& \hspace{0.4cm} + \INTSeg{\hspace{0.125cm} l_K(s) \INTDom{\Big| (1-\lambda) \Phi^u_{(\tau,s)}[\mu_1](x) + \lambda \Phi^u_{(\tau,s)}[\mu_2](y) - \Phi^u_{(\tau,s)}[\gamma^{1 \rightarrow 2}_{\lambda}]((1-\lambda)x + \lambda y) \Big|}{\R^{2d}}{\gamma(x,y)}}{s}{\tau}{t} \\
& \hspace{0.4cm} + \INTSeg{L_K(s) \INTDom{\Big| (1-\lambda) \Phi^u_{(\tau,s)}[\mu_1](x) + \lambda \Phi^u_{(\tau,s)}[\mu_2](y) - \Phi^u_{(\tau,s)}[\gamma^{1 \rightarrow 2}_{\lambda}]((1-\lambda)x + \lambda y) \Big|}{\R^{2d}}{\gamma(x,y)}}{s}{\tau}{t}.
\end{aligned}
\end{equation*}
We thus recover by a direct application of Gr\"onwall's Lemma
\begin{equation}
\label{eq:SemiconcavityFlow4}
\INTDom{\Big| (1-\lambda) \Phi^u_{(\tau,t)}[\mu_1](x) + \lambda \Phi^u_{(\tau,t)}[\mu_2](y) - \Phi^u_{(\tau,t)}[\gamma^{1 \rightarrow 2}_{\lambda}]((1-\lambda)x + \lambda y) \Big|}{\R^{2d}}{\gamma(x,y)} \leq \Spazo_K \lambda(1-\lambda) W_2^2(\mu_1,\mu_2), 
\end{equation}
for all times $t \in [\tau,T]$, where $\Spazo_K := 2 \Norm{\Cpazo_K'(\cdot)}_1 \exp(\Norm{l_K(\cdot)}_1 + \Norm{L_K(\cdot)}_1)$ depends only on $r>0$ since by construction $K := B(0,R_r)$ with $R_r > 0$ given by Lemma \ref{lem:AdmEst}. Plugging the estimate \eqref{eq:SemiconcavityFlow4} back into \eqref{eq:CrossedInterpolationEst}, we can finally conclude that
\begin{equation*}
W_1 \Big(\gamma^{1 \rightarrow 2}_{\lambda}(t) , \Phi^u_{(\tau,t)}[\gamma^{1 \rightarrow 2}_{\lambda}](\cdot)_{\#} \gamma^{1 \rightarrow 2}_{\lambda} \Big) \leq \Spazo_K \lambda(1-\lambda) W_2^2(\mu_1,\mu_2),
\end{equation*}
for every $t \in [\tau,T]$ and all $\lambda \in [0,1]$, which ends the proof.
\end{proof}


\paragraph*{Step 2: Local geodesic semiconcavity of the value function} Given $r >0$, let $\mu_1,\mu_2 \in \Pcal(B(0,r))$  and $\gamma \in \Gamma_o(\mu_1,\mu_2)$ be arbitrary, and take $\tau \in [0,T]$. In addition, let $K := B(0,R_r)$ with $R_r > 0$ as in Lemma \ref{lem:AdmEst}.

\begin{proof}[Proof of Theorem \ref{thm:Semiconcavity1}] 
By the definition \eqref{eq:ValueFunction} of the value function, there exists for all $\epsilon >0$ and $\lambda \in [0,1]$ an admissible control $u_{\epsilon}(\cdot) \in \U$ such that 
\begin{equation*}
\varphi \Big( \Phi_{(\tau,T)}^{u_{\epsilon}}[\gamma^{1 \rightarrow 2}_{\lambda}](\cdot)_{\#} \gamma^{1 \rightarrow 2}_{\lambda} \Big) \leq \Vcal(\tau,\gamma^{1 \rightarrow 2}_{\lambda}) + \epsilon.
\end{equation*}
Using again the definition of the value function, we further get
\begin{equation}
\label{eq:ValueFunctionSemiConcavity1}
\begin{aligned}
(1-\lambda) \Vcal (\tau,\mu_1) & + \lambda \Vcal(\tau,\mu_2) - \Vcal(\tau,\gamma^{1 \rightarrow 2}_{\lambda}) \\
& \leq (1-\lambda) \varphi(\mu_1(T)) + \lambda \varphi(\mu_2(T)) - \varphi \Big( \Phi_{(\tau,T)}^{u_{\epsilon}}[\gamma^{1 \rightarrow 2}_{\lambda}](\cdot)_{\#} \gamma^{1 \rightarrow 2}_{\lambda} \Big) + \epsilon,
\end{aligned}
\end{equation}
for all $\epsilon > 0$ and any $\lambda \in [0,1]$,  where 
\begin{equation*}
\mu_1(T) := \Phi_{(\tau,T)}^{u_{\epsilon}}[\mu_1](\cdot)_{\#} \mu_1 \qquad \text{and} \qquad \mu_2(T) := \Phi_{(\tau,T)}^{u_{\epsilon}}[\mu_2](\cdot)_{\#} \mu_2.
\end{equation*}
Invoking the results of Lemma \ref{lem:GeodesicInterpolationFlows}, there exists a constant $\Spazo_K > 0$ that is independent of $u_{\epsilon}(\cdot)$ and a $2$-optimal transport plan $\gamma(T) \in \Gamma_o(\mu_1(T),\mu_2(T))$ such that
\begin{equation*}
W_1 \Big( \Phi_{(\tau,T)}^{u_{\epsilon}}[\gamma^{1 \rightarrow 2}_{\lambda}](\cdot)_{\#} \gamma^{1 \rightarrow 2}_{\lambda} , \gamma^{1 \rightarrow 2}_{\lambda}(T) \Big) \leq \Spazo_K \, \lambda(1-\lambda) W_2^2(\mu_1,\mu_2), 
\end{equation*}
for every $\lambda \in [0,1]$. By hypothesis \ref{hyp:OCP}-$(ii)$, this further implies 
\begin{equation}
\label{eq:ValueFunctionSemiConcavity2}
\varphi(\gamma^{1 \rightarrow 2}_{\lambda}(T)) - \varphi \Big( \Phi_{(\tau,T)}^{u_{\epsilon}}[\gamma^{1 \rightarrow 2}_{\lambda}](\cdot)_{\#} \gamma^{1 \rightarrow 2}_{\lambda} \Big) \leq \Lpazo_K \Spazo_K \, \lambda(1-\lambda) W_2^2(\mu_1,\mu_2), 
\end{equation}
where $\Lpazo_{K} := \Lip(\varphi(\cdot); \Pcal_1(K))$. Recalling that $\varphi(\cdot)$ is locally geodesically semiconcave by hypothesis \ref{hyp:SC1}-$(ii)$, we also get
\begin{equation}
\label{eq:ValueFunctionSemiConcavity3}
\begin{aligned}
(1-\lambda) \varphi(\mu_1(T)) + \lambda \varphi(\mu_2(T)) - \varphi(\gamma^{1 \rightarrow 2}_{\lambda}(T)) & \leq \Cpazo_K \lambda(1-\lambda) W_2^2(\mu_1(T),\mu_2(T)) \\
& \leq \Cpazo_K' \lambda(1-\lambda) W_2^2(\mu_1,\mu_2),
\end{aligned}
\end{equation}
with $\Cpazo_K > 0$ being the local geodesic semiconcavity constant of $\varphi(\cdot)$ over $\Pcal(K)$, and 
\begin{equation*}
\Cpazo_K' := 4 \sup_{t \in [0,T]} \bigg( \max_{\mu \in \Pcal(K)} \Lip \Big( \Phi_{(\tau,t)}^{u_{\epsilon}}[\mu](\cdot); K \Big) + \max_{x \in K} \, \Lip \Big( \Phi_{(\tau,t)}^{u_{\epsilon}}[\cdot](x); \Pcal_1(K) \Big) \bigg) \Cpazo_K.
\end{equation*}
The fact that $\Cpazo_K' \geq 0$ is independent of $(\epsilon,\lambda) \in \R_+^* \times [0,1]$ follows from Lemma \ref{lem:LipFlow}, together with the uniformity with respect to $u \in U$ of the regularity constants appearing in \ref{hyp:OCP}-$(i)$. Thus, plugging \eqref{eq:ValueFunctionSemiConcavity2} and \eqref{eq:ValueFunctionSemiConcavity3} into \eqref{eq:ValueFunctionSemiConcavity1}, we obtain for all $(\epsilon,\lambda) \in \R_+^* \times [0,1]$ that
\begin{equation*}
(1-\lambda) \Vcal(\tau,\mu_1) + \lambda \Vcal(\tau,\mu_2) - \Vcal(\tau,\gamma^{1 \rightarrow 2}_{\lambda}) \leq \Big( \Cpazo_K' + \Lpazo_K \Spazo_K \Big) \, \lambda(1-\lambda) W_2^2(\mu_1,\mu_2) + \epsilon, 
\end{equation*}
which concludes the proof of Theorem \ref{thm:Semiconcavity1} since $\epsilon > 0$ is arbitrary.
\end{proof}

In what follows, we prove that semiconcavity along arbitrary interpolating curves at the level of the dynamics and cost functionals of $(\Ppazo)$ translates into strong local semiconcavity of the value function.

\begin{taggedhyp}{\textbn{(SC2)}}
\label{hyp:SC2}
Suppose that for any $R > 0$, the following holds with $K:= B(0,R)$. 
\begin{enumerate}
\item[$(i)$] There exists a map $\Cpazo_K(\cdot) \in L^1([0,T],\R_+)$ such that for any $\mu_1,\mu_2 \in \Pcal(K)$, each $\Bmu \in \Gamma(\mu_1,\mu_2)$ and all $x,y \in K$, it holds
\begin{equation*}
\Big| (1-\lambda) v(t,\mu_1,x) + \lambda v(t,\mu_2,y) - v \big( t , \Bmu^{1 \rightarrow 2}_{\lambda}, (1-\lambda)x + \lambda y \big) \Big| \leq \Cpazo_K(t) \, \lambda(1-\lambda) \Big( |x-y|^2 + W_{2,\Bmu}^2(\mu_1,\mu_2) \Big), 
\end{equation*}
for $\Lcal^1$-almost every $t \in [0,T]$ and any $\lambda \in [0,1]$.
\item[$(ii)$] The final cost $\varphi : \Pcal_c(\R^d) \rightarrow \R$ is locally strongly semiconcave. 
\end{enumerate}
\end{taggedhyp}

\begin{thm}[Local strong semiconcavity of the value function]
\label{thm:Semiconcavity2}
Suppose that hypotheses \ref{hyp:OCP} and \ref{hyp:SC2} hold. Then for any $K := B(0,r)$ with $r > 0$, there exists a constant $\Ccal_K > 0$ such that for every $\tau \in [0,T]$ and any $\mu_1,\mu_2 \in \Pcal(K)$, it holds
\begin{equation*}
(1-\lambda) \Vcal(\tau,\mu_1) + \lambda \Vcal(\tau,\mu_2) - \Vcal(\tau,\Bmu^{1 \rightarrow 2}_{\lambda}) \leq \Ccal_K \, \lambda(1-\lambda) W_{2,\Bmu}^2(\mu_1,\mu_2),
\end{equation*}
for all $\lambda \in [0,1]$ and each $\Bmu \in \Gamma(\mu_1,\mu_2)$.
\end{thm}

\begin{proof}
The proof of Theorem \ref{thm:Semiconcavity2} relies on the same arguments as those of Theorem \ref{thm:Semiconcavity1}, up to some modifications in the construction of the curve of plans along which the interpolation estimate is satisfied. One can check by reproducing the estimations in the proof of Lemma \ref{lem:GeodesicInterpolationFlows} and using \ref{hyp:SC2}-$(i)$ that there exists a constant $\Spazo_K > 0$ such that for any $\Bmu \in \Gamma(\mu_1,\mu_2)$, the curve of plans 
\begin{equation*}
t \in [0,T] \mapsto \hat{\Bmu}^{1 \rightarrow 2}_{\lambda}(t) := \Big( (1-\lambda) \Phi_{(\tau,t)}^u[\mu_1] \circ \pi^1 + \lambda \Phi_{(\tau,t)}^u[\mu_2] \circ \pi^2 \, , \, \Phi_{(\tau,t)}^u[\Bmu^{1 \rightarrow 2}_{\lambda}] \big( (1-\lambda)\pi^1 +\lambda \pi^2 \big) \Big)_{\raisebox{4pt}{$\scriptstyle{\#}$}} \Bmu, 
\end{equation*}
allows to derive the strong interpolation estimate
\begin{equation*}
W_1 \Big( \Bmu^{1 \rightarrow 2}_{\lambda}(t), \Phi^u_{(\tau,t)}[\Bmu^{1 \rightarrow 2}_{\lambda}](\cdot)_{\#} \Bmu^{1 \rightarrow 2}_{\lambda} \Big) \leq \Spazo_K \, \lambda(1-\lambda) W_{2,\Bmu}^2(\mu_1,\mu_2), 
\end{equation*}
for all times $t \in [0,T]$ and every $\lambda \in [0,1]$. From there on, the proof of Theorem \ref{thm:Semiconcavity2} directly follows by reproducing the estimates of Step 2, while using the fact that under hypotheses \ref{hyp:SC2}-$(ii)$, the final cost also satisfies a strong semiconcavity estimate along arbitrary interpolating curves.
\end{proof}


\subsection{Semiconcavity with respect to both variables}
\label{subsection:TimeMeasSemiconcavity}

In Theorem \ref{thm:Semiconcavity1} and Theorem \ref{thm:Semiconcavity2} above, we have shown that the value function is locally semiconcave with respect to the measure variable, whenever the dynamics and cost functionals satisfy some adequate interpolation inequalities. In Section \ref{subsection:Feedback} below, we will need an extension of these regularity estimates involving semiconcavity properties with respect to the time variable as well, in order to study optimal feedback mappings for $(\Ppazo)$.

Throughout this section, we impose the following tighter regularity conditions on the controlled non-local velocity fields. 

\begin{taggedhyp}{\textbn{(R)}}
\label{hyp:R}
Suppose that for every $R > 0$, the following holds with $K := B(0,R)$. 
\begin{enumerate}
\item[$(i)$] The map $v : [0,T] \times \Pcal_c(\R^d) \times U \times \R^d \rightarrow \R^d$ satisfies \ref{hyp:OCP}-$(i)$ with a time-independent sublinearity constant $m(\cdot) := m > 0$. 
\item[$(ii)$] There exists a constant $\ell_K > 0$ such that for all $u \in U$, it holds
\begin{equation*}
\big| v(t,\mu,u,x) - v(\tau,\nu,u,y) \big| \leq \ell_K \Big( |t-\tau| + W_1(\mu,\nu) + |x-y| \Big),
\end{equation*}
for every $\tau,t \in [0,T]$, any $\mu,\nu \in \Pcal(K)$ and all $x,y \in K$. 

\end{enumerate}
\end{taggedhyp}

Under these stronger regularity requirements, we are able to establish semiconcavity properties of the value function with respect to both time and measure variables.

\begin{thm}[Joint semiconcavity of the value function]
\label{thm:Semiconcavity3}
Suppose that hypotheses \ref{hyp:OCP}, \ref{hyp:SC1} and \ref{hyp:R} hold. Then for every $K := B(0,r)$ with $r > 0$, there exists a constant $\Ccal_K > 0$ such that for every $\tau_1,\tau_2 \in [0,T]$, any $\mu_1,\mu_2 \in \Pcal(B(0,r))$ and each $\gamma \in \Gamma_o(\mu_1,\mu_2)$, it holds 
\begin{equation*}
(1-\lambda) \Vcal(\tau_1,\mu_1) + \lambda \Vcal(\tau_2,\mu_2) - \Vcal \big( \tau_{\lambda},\gamma^{1 \rightarrow 2}_{\lambda} \big) \leq \Ccal_K \, \lambda(1-\lambda) \Big( |\tau_1 - \tau_2|^2 + W_2^2(\mu_1,\mu_2) \Big), 
\end{equation*}
for all $\lambda \in [0,1]$, where $\tau_{\lambda} := (1-\lambda)\tau_1 + \lambda \tau_2$. If the data of $(\Ppazo)$ satisfy the stronger hypotheses of \ref{hyp:SC2}, then one further has
\begin{equation*}
(1-\lambda) \Vcal(\tau_1,\mu_1) + \lambda \Vcal(\tau_2,\mu_2) - \Vcal \big( \tau_{\lambda},\Bmu^{1 \rightarrow 2}_{\lambda} \big) \leq \Ccal_K \, \lambda(1-\lambda) \Big( |\tau_1 - \tau_2|^2 + W_{2,\Bmu}^2(\mu_1,\mu_2) \Big)
\end{equation*}
for all $\lambda \in [0,1]$ and each $\Bmu \in \Gamma(\mu_1,\mu_2)$.
\end{thm} 

\begin{proof}
We will only treat the case in which hypotheses \ref{hyp:SC1} hold, the other scenario being completely similar up to some minor modifications already detailed in the proof of Theorem \ref{thm:Semiconcavity2} above. We also suppose throughout the proof that $0 \leq \tau_2 < \tau_1 \leq T$, since the case $0 \leq \tau_1 < \tau_2 \leq T$ is analogous, and $\tau_1 = \tau_2$ has already been treated in Section \ref{subsection:SemiconcavityEst}. In the sequel, we will use the notation $K := B(0,R_r)$ where $R_r  \geq r > 0$ is defined as in Lemma \ref{lem:AdmEst}.  

Let $\gamma \in \Gamma_o(\mu_1,\mu_2)$ be a $2$-optimal transport plan, and for $\lambda \in [0,1]$ consider $\tau_{\lambda} := (1-\lambda)\tau_1 + \lambda \tau_2$. Observe that for every $\epsilon > 0$, there exists an admissible control $u_{\epsilon}(\cdot) \in \U$ such that 
\begin{equation}
\label{eq:ValueFunctionTimeEst1}
\Vcal \Big( \tau_1 , \Phi_{(\tau_{\lambda},\tau_1)}^{u_{\epsilon}}[\gamma^{1 \rightarrow 2}_{\lambda}](\cdot)_{\#} \gamma^{1 \rightarrow 2}_{\lambda} \Big) \leq \Vcal \big( \tau_{\lambda},\gamma^{1 \rightarrow 2}_{\lambda} \big) + \epsilon.
\end{equation}
We now consider the following right-invertible reparametrisation of the time variable, given by 
\begin{equation}
\label{eq:Reparametrisation}
\vt_{\lambda}(s) := 
\left\{
\begin{aligned}
& \lambda s + (1-\lambda) \tau_1 ~~ \text{if $\tau_2 \leq s \leq \tau_1$}, \\
& ~ s \hspace{3cm} \text{otherwise},
\end{aligned}
\right.
\end{equation}
for every $s \in [0,T]$, and observe that $\vt_{\lambda}(\tau_1) = \tau_1$ and $\vt_{\lambda}(\tau_2) = \tau_{\lambda}$. Using this time variable and the fact that the value function is non-decreasing along admissible trajectories, we further obtain 
\begin{equation}
\label{eq:ValueFunctionTimeEst2}
\begin{aligned}
(1-\lambda) \Vcal(\tau_1,&\mu_1) + \lambda \Vcal(\tau_2,\mu_2) - \Vcal \big( \tau_{\lambda},\gamma^{1 \rightarrow 2}_{\lambda} \big) \\
& \leq (1-\lambda) \Vcal(\tau_1,\mu_1) + \lambda \Vcal \Big( \tau_1 , \Phi_{(\tau_2,\tau_1)}^{u_{\epsilon} \circ \vt_{\lambda}}[\mu_2](\cdot)_{\#} \mu_2 \Big) - \Vcal \Big( \tau_1 , \Phi_{(\tau_{\lambda},\tau_1)}^{u_{\epsilon}}[\gamma^{1 \rightarrow 2}_{\lambda}](\cdot)_{\#} \gamma^{1 \rightarrow 2}_{\lambda} \Big) + \epsilon,
\end{aligned}
\end{equation}
where we used \eqref{eq:ValueFunctionTimeEst1} and the fact that $u_{\epsilon} \circ \vt_{\lambda}(\cdot) \in \U$. From now on, we set 
\begin{equation}
\label{eq:tilde_mu2}
\tilde{\mu}_2(s) := \Phi_{(\tau_2,\vt_{\lambda}^{-1}(s))}^{u_{\epsilon} \circ \vt_{\lambda}} [\mu_2](\cdot)_{\#} \mu_2, 
\end{equation}
for every $s \in [\tau_2,\tau_1]$, so that $\tilde{\mu}_2(\tau_{\lambda}) = \mu_2$ and \eqref{eq:ValueFunctionTimeEst2} can be rewritten in the simpler form
\begin{equation}
\label{eq:ValueFunctionTimeEst3}
\begin{aligned}
(1-\lambda) \Vcal(\tau_1,&\mu_1) + \lambda \Vcal(\tau_2,\mu_2) - \Vcal \big( \tau_{\lambda},\gamma^{1 \rightarrow 2}_{\lambda} \big) \\
& \leq (1-\lambda) \Vcal(\tau_1,\mu_1) + \lambda \Vcal( \tau_1 , \tilde{\mu}_2(\tau_1)) - \Vcal \Big( \tau_1 , \Phi_{(\tau_{\lambda},\tau_1)}^{u_{\epsilon}}[\gamma^{1 \rightarrow 2}_{\lambda}](\cdot)_{\#} \gamma^{1 \rightarrow 2}_{\lambda} \Big) + \epsilon.
\end{aligned}
\end{equation}
Our goal will now be to estimate the right-hand side of \eqref{eq:ValueFunctionTimeEst3} by means of a suitable dynamical interpolation, in the same vein as in the proof of Lemma \ref{lem:GeodesicInterpolationFlows}. 

Let $\Beta_1,\tilde{\Beta}_2 \in \Pcal(\R^d \times \Sigma_T)$ be the superposition measures associated with the (constant) curve of measures $\mu_1 \in \Pcal(K)$ and to $s \in [\tau_{\lambda},\tau_1] \mapsto \tilde{\mu}_2(s) \in \Pcal(K)$ respectively. By Proposition \ref{prop:SuperpositionPlan}, there exists a transport plan $\tilde{\Beta}_{12} \in \Gamma(\Beta_1,\tilde{\Beta}_2)$ such that 
\begin{equation*}
(\pi^1_{\R^d},\pi^2_{\R^d})_{\#} \tilde{\Beta}_{12} = \gamma \qquad \text{and} \qquad  (e_s^1,e_s^2)_{\#} \tilde{\Beta}_{12} =: \tilde{\gamma}(s) \in \Gamma_o(\mu_1,\tilde{\mu}_2(s)), 
\end{equation*}
for all times $s \in [\tau_{\lambda},\tau_1]$. Inserting the curve $\tilde{\gamma}^{1 \rightarrow 2}_{\lambda}(s) := ((1-\lambda) \pi^1 + \lambda \pi^2)_{\#} \tilde{\gamma}(s)$ as a crossed term in the right-hand side of \eqref{eq:ValueFunctionTimeEst3} while invoking Proposition \ref{prop:LipRegValue} and Theorem \ref{thm:Semiconcavity1}, one has
\begin{equation}
\label{eq:ValueFunctionTimeEst4}
\begin{aligned}
(1-\lambda) \Vcal(\tau_1,\mu_1) + & \lambda \Vcal( \tau_1 , \tilde{\mu}_2(\tau_1)) - \Vcal \Big( \tau_1 , \Phi_{(\tau_{\lambda},\tau_1)}^{u_{\epsilon}}[\gamma^{1 \rightarrow 2}_{\lambda}](\cdot)_{\#} \gamma^{1 \rightarrow 2}_{\lambda} \Big) \\
& \leq \Ccal_K \, \lambda (1-\lambda) W_2^2(\mu_1,\tilde{\mu}_2(\tau_1)) + \Lpazo_K W_1 \Big( \tilde{\gamma}^{1 \rightarrow 2}_{\lambda}(\tau_1) , \Phi_{(\tau_{\lambda},\tau_1)}^{u_{\epsilon}}[\gamma^{1 \rightarrow 2}_{\lambda}](\cdot)_{\#} \gamma^{1 \rightarrow 2}_{\lambda} \Big),
\end{aligned}
\end{equation}
where $\Ccal_K, \Lpazo_K > 0$ are the geodesic semiconcavity and Lipschitz constants of $\Vcal(\tau_1,\cdot)$ over $\Pcal(K)$ respectively. By \eqref{eq:tilde_mu2} and hypothesis \ref{hyp:R}-$(i)$, we can further estimate the first term in the right-hand side of \eqref{eq:ValueFunctionTimeEst4} as
\begin{equation}
\label{eq:ValueFunctionTimeEst5}
\begin{aligned}
W_2^2(\mu_1,\tilde{\mu}_2(\tau_1)) & \leq \Big( W_2(\mu_1,\mu_2) + W_2(\mu_2,\tilde{\mu}_2(\tau_1)) \Big)^2 \\ 
& \leq 2 \Big( W_2^2(\mu_1,\mu_2) + \NormC{\Phi_{(\tau_2,\tau_1)}^{u_{\epsilon} \circ \vt_{\lambda}}[\mu_2](\cdot) - \Id}{0}{K,\R^d}^2 \Big) \\
& \leq 2 (1+2R_r)^2m^2 \Big( |\tau_2 - \tau_1|^2 + W_2^2(\mu_1,\mu_2) \Big), 
\end{aligned}
\end{equation}
where we also applied \eqref{eq:WassEst1}. Concerning the second term in the right-hand side of \eqref{eq:ValueFunctionTimeEst4}, notice first that by construction, the measures
\begin{equation*}
\tilde{\Bmu}_{\lambda}^{1 \rightarrow 2}(s) := \Big( (1-\lambda) e_s^1 + \lambda e_s^2 \, , \, \Phi_{(\tau_{\lambda},s)}^{u_{\epsilon}}[\gamma^{1 \rightarrow 2}_{\lambda}] \circ \big( (1-\lambda) \pi^1_{\R^d} + \lambda \pi^2_{\R^d} \big) \Big)_{\raisebox{4pt}{$\scriptstyle{\#}$}} \tilde{\Beta}_{12}
\end{equation*}
are admissible transport plans between $\tilde{\gamma}_{\lambda}^{1 \rightarrow 2}(s)$ and $\Phi_{(\tau_{\lambda},s)}^{u_{\epsilon}}[\gamma^{1 \rightarrow 2}_{\lambda}](\cdot)_{\#} \gamma^{1 \rightarrow 2}_{\lambda}$ for every $s \in [\tau_{\lambda},\tau_1]$. Thus
\begin{equation}
\label{eq:ValueFunctionTimeEst6}
\begin{aligned}
& \hspace{-1.2cm} W_1 \Big( \tilde{\gamma}^{1 \rightarrow 2}_{\lambda}(s) , \Phi_{(\tau_{\lambda},s)}^{u_{\epsilon}}[\gamma^{1 \rightarrow 2}_{\lambda}](\cdot)_{\#} \gamma^{1 \rightarrow 2}_{\lambda} \Big) \\
& \leq \INTDom{|\xb - \yb|}{\R^{2d}}{\tilde{\Bmu}^{1 \rightarrow 2}_{\lambda}(s)(\xb,\yb)} \\
& = \INTDom{\Big| (1-\lambda) \sigma_1(t) + \lambda \sigma_2(t) - \Phi_{(\tau_{\lambda},s)}^{u_{\epsilon}}[\gamma^{1 \rightarrow 2}_{\lambda}] \big( (1-\lambda)x + \lambda y \big) \Big|}{(\R^d \times \Sigma_T)^2}{\tilde{\Beta}_{12}(x,\sigma_1,y,\sigma_2)} \\
& = \INTDom{\Big| (1-\lambda) x + \lambda \Phi^{u_{\epsilon} \circ \vt_{\lambda}}_{(\tau_2,\vt_{\lambda}^{-1}(s))}[\mu_2](y) - \Phi_{(\tau_{\lambda},s)}^{u_{\epsilon}}[\gamma^{1 \rightarrow 2}_{\lambda}] \big( (1-\lambda)x + \lambda y \big) \Big|}{\R^{2d}}{\gamma(x,y)},
\end{aligned}
\end{equation}
where we used the fact that under hypotheses \ref{hyp:CE}, the superposition measures $\Beta_1,\tilde{\Beta}_2 \in \Pcal(\R^d \times \Sigma_T)$ built above are concentrated on the pairs $(x,\sigma_1),(y,\sigma_2) \in \R^d \times \Sigma_T$ given by 
\begin{equation*}
\sigma_1(s) = x \qquad \text{and} \qquad \sigma_2(s) = \Phi^{u_{\epsilon} \circ \vt_{\lambda}}_{(\tau_{\lambda_2},\vt_{\lambda}^{-1}(s))}[\mu_2](y), 
\end{equation*}
for all times $s \in [\tau_{\lambda},\tau_1]$ and any $(x,y) \in \supp(\gamma)$. By using hypotheses \ref{hyp:R}, one can show that there exists a constant $C_K > 0$ such that
\begin{equation}
\label{eq:TimeShiftedInterp}
\begin{aligned}
& \INTDom{\Big| (1-\lambda) x + \lambda \Phi^{u_{\epsilon} \circ \vt_{\lambda}}_{(\tau_2,\vt_{\lambda}^{-1}(s))}[\mu_2](y) - \Phi_{(\tau_{\lambda},s)}^{u_{\epsilon}}[\gamma^{1 \rightarrow 2}_{\lambda}] \big( (1-\lambda)x + \lambda y \big) \Big|}{\R^{2d}}{\gamma(x,y)} \\
& \hspace{1.9cm} \leq C_K \lambda(1-\lambda) \Big( |\tau_2 - \tau_1|^2 + W_2^2(\mu_1,\mu_2) \Big) + \ell_K \INTSeg{W_1 \Big( \tilde{\gamma}_{\lambda}^{1 \rightarrow 2}(\zeta) , \Phi_{(\tau_{\lambda},\zeta)}^{u_{\epsilon}}[\gamma^{1 \rightarrow 2}_{\lambda}](\cdot)_{\#} \gamma^{1 \rightarrow 2}_{\lambda}
\Big)}{\zeta}{\tau_{\lambda}}{s}. 
\end{aligned}
\end{equation}
The proof of this inequality being somewhat heavy and similar to computations already detailed in Lemma \ref{lem:GeodesicInterpolationFlows}, it is postponed to Appendix \ref{section:AppendixGronwall} below. Plugging this last estimate into \eqref{eq:ValueFunctionTimeEst6} and applying Gr\"onwall's Lemma further yields
\begin{equation}
\label{eq:ValueFunctionTimeEst7}
W_1 \Big( \tilde{\gamma}^{1 \rightarrow 2}_{\lambda}(s) , \Phi_{(\tau_{\lambda},s)}^{u_{\epsilon}}[\gamma^{1 \rightarrow 2}_{\lambda}](\cdot)_{\#} \gamma^{1 \rightarrow 2}_{\lambda} \Big)  \leq C_K' \lambda(1-\lambda) \Big( |\tau_2 - \tau_1 |^2 + W_2^2(\mu_1,\mu_2) \Big),
\end{equation}
where $C'_K := C_K \exp(\ell_K T)$. Thus upon combining \eqref{eq:ValueFunctionTimeEst3}, \eqref{eq:ValueFunctionTimeEst4}, \eqref{eq:ValueFunctionTimeEst5} and \eqref{eq:ValueFunctionTimeEst7}, we finally recover
\begin{equation*}
(1-\lambda) \Vcal(\tau_1,\mu_1) + \lambda \Vcal(\tau_2,\mu_2) - \Vcal \big( \tau_{\lambda},\gamma^{1 \rightarrow 2}_{\lambda} \big) \leq \Ccal_K' \lambda(1-\lambda) \Big( |\tau_2 - \tau_1|^2 + W_2^2(\mu_1,\mu_2) \Big) + \epsilon, 
\end{equation*}
where $\Ccal_K' := C'_K \Lpazo_K + 2m^2(1+2R_r)^2 \Ccal_K$, which concludes the proof since $\epsilon > 0$ is arbitrary. 
\end{proof}


\subsection{Examples of functionals satisfying hypotheses \ref{hyp:SC1} and \ref{hyp:SC2}}
\label{subsection:SemiconcavityExamples}

In this short section, we provide several examples of dynamics and cost functionals satisfying the semiconcavity hypotheses of Section \ref{subsection:SemiconcavityEst} on which Theorem \ref{thm:Semiconcavity1}, Theorem \ref{thm:Semiconcavity2} and Theorem \ref{thm:Semiconcavity3} are based. Most of these examples are treated in a slightly different manner in \cite[Chapter 9]{AGS}. 

\begin{prop}[Semiconcavity of potential energies]
Let $V \in C^1(\R^d,\R)$ be such that $\nabla V : \R^d \rightarrow \R^d$ is locally Lipschitz. Then, the functional $\V : \Pcal_c(\R^d) \rightarrow \R$ defined by 
\begin{equation*}
\V(\mu) := \INTDom{V(x)}{\R^d}{\mu(x)},
\end{equation*}
for every $\mu \in \Pcal_c(\R^d)$ is locally strongly semiconcave.
\end{prop}

\begin{proof}
It is a known fact (see e.g. \cite[Proposition 2.1.2]{CannarsaS2004}) that if $V \in C^1(\R^d,\R)$ is such that $\nabla V(\cdot)$ is locally Lipschitz, then for every ball $K := B(0,R)$ with $R>0$ and any $x,y \in K$, it holds
\begin{equation*}
(1-\lambda) V(x) + \lambda V(y) - V((1-\lambda)x + \lambda y) \leq \Cpazo_K \, \lambda(1-\lambda) |x-y|^2, 
\end{equation*}
for every $\lambda \in [0,1]$, where $\Cpazo_K := \Lip(\nabla V(\cdot);K)$. Whence, for any $\mu_1,\mu_2 \in \Pcal(K)$ and each $\Bmu \in \Gamma(\mu_1,\mu_2)$, one has
\begin{equation*}
\begin{aligned}
(1-\lambda) \V(\mu_1) + \lambda \V(\mu_2) - \V(\Bmu^{1 \rightarrow 2}_{\lambda}) & = \INTDom{\Big( (1-\lambda) V(x) + \lambda V(y) - V((1-\lambda)x + \lambda y) \Big)}{\R^{2d}}{\Bmu(x,y)} \\
& \leq \Cpazo_K \, \lambda(1-\lambda) \INTDom{|x-y|^2}{\R^{2d}}{\Bmu(x,y)} \\
& =  \Cpazo_K \, \lambda(1-\lambda) W_{2,\Bmu}^2(\mu_1,\mu_2),  
\end{aligned}
\end{equation*}
for every $\lambda \in [0,1]$, which concludes the proof. 
\end{proof}

This type of local strong semiconcavity regularity is also valid in the broader class of functionals describing interaction energies. 

\begin{cor}[Semiconcavity of interaction energies]
Let $m \geq 1$ be an integer and $W \in C^1(\R^{m \times d},\R)$ be such that $\nabla W : \R^{m \times d} \rightarrow \R^{m \times d}$ is locally Lipschitz. Then, the functional $\Wpazo : \Pcal_c(\R^d) \rightarrow \R$ defined by
\begin{equation*}
\Wpazo(\mu) := \INTDom{W(x_1,\dots,x_m)}{\R^{m \times d}}{(\mu \times \dots \times \mu)(x_1,\dots,x_m)},
\end{equation*}
for every $\mu \in \Pcal_c(\R^d)$ is locally strongly semiconcave.
\end{cor}

Another interesting example of strongly semiconcave function is the squared Wasserstein distance. 

\begin{prop}[Strong semiconcavity of the squared Wasserstein distance]
For any $\nu \in \Pcal_2(\R^d)$ and every compact set $K \subset \R^d$, the map $\mu \in \Pcal(K) \mapsto W_2^2(\mu,\nu) \in \R$ is strongly semiconcave with $\Cpazo_K = 1$.
\end{prop}

\begin{proof}
See \cite[Theorem 7.3.2]{AGS}. 
\end{proof}

In the following proposition, we provide a similar strong semiconcavity result in norm for non-local velocity fields given in the form of smooth convolutions.

\begin{prop}[Norm-semiconcavity of smooth convolutions]
Let $H : (t,x) \in [0,T] \times \R^d \rightarrow \R^d$ be $\Lcal^1$-measurable with respect to $t \in [0,T]$ and continuously Fr\'echet differentiable with respect to $x \in \R^d$. Suppose also that for every $K := B(0,R)$ with $R > 0$, there exists $\Cpazo_K(\cdot) \in L^1([0,T],\R_+)$ such that
\begin{equation*}
|H(t,0)| + \Lip(\D_x H(t,\cdot);K) \leq \Cpazo_K(t), 
\end{equation*}
for $\Lcal^1$-almost every $t \in [0,T]$.

Then, the non-local velocity field $v : [0,T] \times \Pcal_c(\R^d) \times \R^d \rightarrow \R^d$ defined by 
\begin{equation*}
v(t,\mu,x) := \INTDom{H(t,x-z)}{\R^d}{\mu(z)}, 
\end{equation*}
for $\Lcal^1$-almost every $t \in [0,T]$ and all $(\mu,x) \in \Pcal_c(\R^d) \times \R^d$ satisfies \ref{hyp:SC2}-$(i)$.
\end{prop}

\begin{proof}
As for real-valued functionals, it can be checked in this case (see e.g. \cite[Section 5]{Cannarsa1991}) that for $\Lcal^1$-almost every $t \in [0,T]$ and all $x,y \in K$, it holds 
\begin{equation*}
\big| (1-\lambda) H(t,x) + \lambda H(t,y) - H( t,(1-\lambda)x + \lambda y)  \big| \leq \Cpazo_K(t) \, \lambda(1-\lambda) |x-y|^2.
\end{equation*}
Whence, for any $\mu_1,\mu_2 \in \Pcal(K)$, all $x,y \in K$ and each $\Bmu \in \Gamma(\mu_1,\mu_2)$, one can deduce
\begin{equation*}
\begin{aligned}
& \Big| (1-\lambda) v(t,\mu_1,x) + \lambda v(t,\mu_2,y) - v(t,\Bmu^{1 \rightarrow 2}_{\lambda}, (1-\lambda)x + \lambda y) \Big| \\
& = \bigg| \INTDom{ \Big( (1-\lambda) H(t,x-z_1) + \lambda H(t,y-z_2) \Big)}{\R^{2d}}{\Bmu(z_1,z_2)} - \INTDom{H \Big( t , (1-\lambda)x + \lambda y - z_{\lambda} \Big)}{\R^d}{\Bmu_{\lambda}^{1 \rightarrow 2}(z_{\lambda})} \bigg| \\
& \leq \INTDom{\Big| (1-\lambda) H(t,x-z_1) + \lambda H(t,y-z_2) - H \big( t,(1-\lambda)(x-z_1) + \lambda(y-z_2) \big) \Big|}{\R^{2d}}{\Bmu(z_1,z_2)} \\
& \leq 2 \Cpazo_K(t) \, \lambda(1-\lambda) \INTDom{\Big( |x-y|^2 + |z_1 - z_2|^2 \Big)}{\R^{2d}}{\Bmu(z_1,z_2)} \\
& = 2 \Cpazo_K(t) \, \lambda(1-\lambda) \Big( |x-y|^2 + W_{2,\Bmu}^2(\mu_1,\mu_2) \Big), 
\end{aligned}
\end{equation*}
for all $\lambda \in [0,1]$, which concludes the proof of our claim. 
\end{proof}

\begin{rmk}[Comparison between hypotheses \ref{hyp:SC1} and \ref{hyp:SC2}]
Going back to Definition \ref{def:Semiconcavity}, it is clear that strong local semiconcavity implies local geodesic semiconcavitiy, so that every data satisfying hypotheses \ref{hyp:SC2} also satisfies hypotheses \ref{hyp:SC1} by a simple restriction to optimal interpolations.   
\end{rmk}


\section{Sensitivity analysis of the value function}
\label{section:Sensitivity}

In this section, we pursue our investigation of the fine properties of the value function $\Vcal : [0,T] \times \Pcal_c(\R^d) \rightarrow \R$ associated with a mean-field optimal control problem of Mayer type. In the context of mean-field optimal control, the well-known fact of stating that the map $t \in [0,T] \mapsto \Vcal(t,\mu(t)) \in \R$ is non-decreasing whenever $(\mu(\cdot),u(\cdot))$ is an admissible pair still holds true. Furthermore, the application $t \in [0,T] \mapsto \Vcal(t,\mu^*(t)) \in \R$ is constant along a pair $(\mu^*(\cdot),u^*(\cdot))$ if and only if it is optimal for $(\Ppazo)$. In Section \ref{subsection:SufficientOpt}, we will characterise the optimality of admissible trajectory-control pairs in terms of this property. 


\subsection{Sensitivity relations expressed in terms of Fr\'echet superdifferentials}
\label{subsection:FrechetSensitivity}

We state and prove below one of our central results, which is a sensitivity relation between the costates of the maximum principle and the Fr\'echet superdifferential of the value function, defined as follows.

\begin{Def}[Extended Fr\'echet superdifferential of the value function]
\label{def:Superdiff_Value}
We say that a pair $(\delta,\Bgamma) \in \R \times \Pcal_2(\R^{2d})$ belongs to the localised extended Fr\'echet superdifferential $\Bpartial^+ \Vcal(\tau_1,\mu_1)$ of the value function $\Vcal : [0,T] \times \Pcal_c(\R^d) \rightarrow \R$ at $(\tau_1,\mu_1) \in [0,T] \times \Pcal_c(\R^d)$ if $\pi^1_{\#}\Bgamma = \mu_1$, and for every $R >0$, it holds
\begin{equation}
\label{eq:SuperdifferentialValue_Ineq}
\Vcal(\tau_2,\mu_2) - \Vcal(\tau_1,\mu_1) \leq \sup_{\Bmu \in \BGamma_o^{1,3}(\Bgamma,\mu_2)} \INTDom{\langle r ,y-x \rangle}{\R^{3d}}{\Bmu(x,r,y)} \,+ \, \delta (\tau_2 - \tau_1) \, + \, o_R \big( W_2(\mu_1,\mu_2) + |\tau_1 - \tau_2| \big), 
\end{equation}
for any $(\tau_2,\mu_2) \in [0,T] \times \Pcal(B_R(\mu_1))$. Analogously, we will say that $\Bgamma \in \Pcal_2(\R^{2d})$ belongs to the localised extended Fr\'echet superdifferential $\Bpartial^+_{\mu} \Vcal(\tau,\mu_1)$ with respect to the measure variable of $\Vcal(\tau,\cdot)$ at $\mu_1 \in \Pcal_c(\R^d)$ if $\pi^1_{\#} \Bgamma = \mu^1$, and for every $R>0$ it holds
\begin{equation}
\label{eq:SuperdifferentialValue_IneqBis}
\Vcal(\tau,\mu_2) - \Vcal(\tau,\mu_1) \leq \sup_{\Bmu \in \Gamma^{1,3}_o(\Bgamma,\mu_2)} \INTDom{\langle r ,y-x \rangle}{\R^{3d}}{\Bmu(x,r,y)} + \, o_R \big( W_2(\mu_1,\mu_2)\big), 
\end{equation}
for any $\mu_2 \in \Pcal(B_R(\mu_1))$. 
\end{Def}

\begin{thm}[A Fr\'echet-type sensitivity relation]
\label{thm:Sensitivity1}
Let $\mu^0 \in \Pcal(B(0,r))$ for some $r>0$ and suppose that hypotheses \ref{hyp:OCP} hold. Given a minimiser $(\mu^*(\cdot),u^*(\cdot))$ for $(\Ppazo)$, let $R_r'>0$ and $\nu^*(\cdot) \in \AC([0,T],\Pcal_1(K' \times K'))$ be as in Theorem \ref{thm:PMP} where $K' := B(0,R_r')$. 

Then, the following \textnormal{sensitivity relation}
\begin{equation}
\label{eq:Sensitivity_Thm}
\Big( \H(t,\nu^*(t),u^*(t)) \, , \, (\pi^1,-\pi^2)_{\#} \nu^*(t) \Big) \in \Bpartial^+ \Vcal(t,\mu^*(t)), 
\end{equation}
holds for $\Lcal^1$-almost every $t \in [0,T]$.  
\end{thm}

We split the proof of this result into four steps. In Step 1, we prove a differential inequality at time $T >0$ for the value function, which is derived using the Taylor expansion formulas of Appendix \ref{section:AppendixFlowDiff}. In Step 2, we establish a representation formula for the curve $\nu^*(\cdot)$, which allows to uniquely characterise the dynamics of its second marginal by a disintegration argument against $\mu^*(T)$. In Step 3, we use this representation formula to propagate the differential inequality derived in Step 1 backward-in-time, and we finally show in Step 4 how the latter yields the sensitivity relation \eqref{eq:Sensitivity_Thm}.


\paragraph*{Step 1: A differential inequality at time $T > 0$} Let $\T^* \subset (0,T)$ be the full $\Lcal^1$-measure set of Lebesgue points associated with $\vb : t \in [0,T] \mapsto v(t,\cdot,u^*(t),\cdot) \in C^0(\Pcal_c(\R^d) \times \R^d,\R^d)$ as in Lemma \ref{lem:LebesguePoints}, and define $\T := \T^* \cap \T_m \subset (0,T)$ where $\T_m$ is the set of Lebesgue points of $m(\cdot) \in L^1([0,T],\R_+)$.

Let $\tau \in \T$ and $\mu_{\tau} \in \Pcal(B(0,r'))$ for some $r' > 0$, and fix $h \in \R$ such that $\tau+h \in [0,T]$. By Lemma \ref{lem:AdmEst} and Theorem \ref{thm:PMP}, there exists a radius $\Rpazo := \Rpazo(r,r') > 0$ depending on both $r,r' >0$ such that, for all times $t \in [0,T]$, the following support inclusions
\begin{equation*}
\supp(\nu^*(t)) \subset \K \times \K \qquad \text{and} \qquad \supp(\mu(t)) \subset  \K,
\end{equation*}
hold with $\K := B(0,\Rpazo)$, whenever $\mu(\cdot)$ is the solution of the Cauchy problem
\begin{equation*}
\left\{
\begin{aligned}
& \partial_t \mu(t) + \Div_x \Big( v(t,\mu(t),u(t)) \mu(t) \Big) = 0, \\
& \mu(s) = \mu_{\tau}, 
\end{aligned}
\right.
\end{equation*}
for some admissible $u(\cdot) \in \U$ and arbitrary $s \in [0,T]$. In the sequel, when there is no ambiguity, we will frequently use the condensed notation 
\begin{equation}
\label{eq:CondensedNota}
(\Phi^*_{(\tau,t)}(\cdot))_{\tau,t \in [0,T]} := (\Phi_{(\tau,t)}^{u^*}[\mu^*(\tau)](\cdot))_{\tau,t \in [0,T]}, 
\end{equation}
to refer to the unique semigroup of non-local flows defined along $(\mu^*(\cdot),u^*(\cdot))$ as in Section \ref{subsection:SemiconcavityEst}. 

In the following lemma, we derive a differential estimate for the value function. 

\begin{lem}[Differential estimate on the value function]
\label{lem:ValueFunctionDiffEst}
For every $\Bmu_{\tau} \in \Gamma(\mu^*(\tau),\mu_{\tau})$, it holds
\begin{equation}
\label{eq:ValueFunctionEstLem}
\begin{aligned}
\Vcal(\tau+h,\mu_{\tau}&) - \Vcal(\tau,\mu^*(\tau)) \\
& \leq \INTDom{\Big\langle \nabla \varphi(\mu^*(T)) \big( \Phi^*_{(\tau,T)}(x) \big) , \D_x \Phi^*_{(\tau,T)}(x)(y-x) + w_{\Bmu_{\tau}}(T,x) \Big\rangle}{\R^{2d}}{\Bmu_{\tau}(x,y)} \hspace{0.4cm} \\
& \hspace{0.4cm} + h \INTDom{\Big\langle \nabla \varphi(\mu^*(T)) \big( \Phi^*_{(\tau,T)}(x) \big) ,\Psi_{\tau}(T,x) \Big\rangle}{\R^{2d}}{\Bmu_{\tau}(x,y)} + o_{\Rpazo} \Big( W_{2,\Bmu_{\tau}}(\mu^*(\tau),\mu_{\tau}) + h \Big),
\end{aligned}
\end{equation}
where the maps $\D_x \Phi^*_{(\tau,T)}(\cdot)$, $w_{\Bmu_{\tau}}(\cdot,\cdot)$ and $\Psi_{\tau}(\cdot,\cdot)$ are defined as in Proposition \ref{prop:LinSpace_Flows}, Theorem \ref{thm:FlowDiff} and Proposition \ref{prop:LinTime_Flows} respectively, with $\vb : (t,\mu,x) \in [0,T] \times \Pcal_c(\R^d) \times \R^d \mapsto v(t,\mu,u^*(t),x) \in \R^d$. 
\end{lem}

\begin{proof}
We start by observing that by definition \eqref{eq:ValueFunction} of the value function $\Vcal : [0,T] \times \Pcal_c(\R^d) \rightarrow \R$, the following inequality holds 
\begin{equation}
\label{eq:ValueFunctionEst1}
\Vcal(\tau+h,\mu_{\tau}) - \Vcal(\tau,\mu^*(\tau)) \leq \varphi \Big( \Phi^{u^*}_{(\tau+h,T)}[\mu_{\tau}](\cdot)_{\#} \mu_{\tau} \Big) - \varphi(\mu^*(T)), 
\end{equation}
because $\mu^*(\cdot)$ is optimal for $(\Ppazo)$ and $t \in [0,T] \mapsto \Phi^{u^*}_{(\tau+h,t)}[\mu_{\tau}](\cdot)_{\#} \mu_{\tau} \in \Pcal(\K)$ is admissible by construction. Let $\Bmu_{\tau} \in \Gamma(\mu^*(\tau),\mu_{\tau})$ be an arbitrary transport plan, to which we associate $\Bmu_{\tau}(T) \in \Gamma \big( \mu^*(T),\Phi^{u^*}_{(\tau+h,T)}[\mu_{\tau}](\cdot)_{\#} \mu_{\tau} \big)$, defined by
\begin{equation}
\label{eq:BmutauDef}
\Bmu_{\tau}(T) := \Big( \Phi^{u^*}_{(\tau,T)}[\mu^*(\tau)] \circ \pi^1 \, , \, \Phi^{u^*}_{(\tau+h,T)}[\mu_{\tau}] \circ \pi^2 \Big)_{\raisebox{4pt}{$\scriptstyle \#$}} \Bmu_{\tau}.
\end{equation}
Recall now that $\varphi : \Pcal_c(\R^d)\rightarrow \R$ is locally differentiable by hypothesis \ref{hyp:OCP}-$(ii)$, which allows us to write as a consequence of Proposition \ref{prop:GradientChainrule}
\begin{equation}
\label{eq:VarphiDiff1}
\begin{aligned}
\varphi \Big( \Phi^{u^*}_{(\tau+h,T)}[\mu_{\tau}](\cdot)_{\#} \mu_{\tau} \Big) - \varphi(\mu^*(T)) & = \INTDom{\langle \nabla \varphi(\mu^*(T))(x),y-x \rangle}{\R^{2d}}{\Bmu_{\tau}(T)(x,y)} \\
& \hspace{1cm} + o_{\Rpazo} \Big( W_{2,\Bmu_{\tau}(T)} \Big( \mu^*(T), \Phi^{u^*}_{(\tau+h,T)}[\mu_{\tau}](\cdot)_{\#} \mu_{\tau} \Big) \Big).
\end{aligned}
\end{equation}
By Lemma \ref{lem:LipFlow}, there exists a constant $C_{\K,\tau} > 0$ such that for any pair $x,y \in \K$, it holds
\begin{equation*}
\big| \Phi^{u^*}_{(\tau+h,T)}[\mu_{\tau}](y) - \Phi^{u^*}_{(\tau,T)}[\mu^*(\tau)](x) \big| \leq C_{\K,\tau} \Big( |x-y| + W_{2,\Bmu_{\tau}}(\mu^*(\tau),\mu_{\tau}) + h \Big), 
\end{equation*}
where we also used the fact that $\tau \in \T$ is a Lebesgue point of $m(\cdot) \in L^1([0,T],\R_+)$. By integrating this last expression against $\Bmu_{\tau}$ and recalling the definition \eqref{eq:BmutauDef} of $\Bmu_{\tau}(T)$, we obtain
\begin{equation}
\label{eq:Sensitivity_SmallOEst1}
o_{\Rpazo} \Big( W_{2,\Bmu_{\tau}(T)} \Big( \mu^*(T), \Phi^{u^*}_{(\tau+h,T)}[\mu_{\tau}](\cdot)_{\#} \mu_{\tau} \Big) \Big) = o_{\Rpazo} \Big( W_{2,\Bmu_{\tau}}(\mu^*(\tau),\mu_{\tau}) + h \Big).
\end{equation}
We now turn our attention to the integral term in \eqref{eq:VarphiDiff1}. Observe that by the definition \eqref{eq:BmutauDef} of $\Bmu_{\tau}(T)$, the latter can be rewritten as  
\begin{equation}
\label{eq:VarphiDiff2}
\begin{aligned}
& \INTDom{\langle \nabla \varphi(\mu^*(T))(x),y-x\rangle}{\R^{2d}}{\Bmu_{\tau}(T)(x,y)} \\
& \hspace{1cm} = \INTDom{\Big\langle \nabla \varphi (\mu^*(T)) \big( \Phi^{u^*}_{(\tau,T)}[\mu^*(\tau)](x) \big) \, , \, \Phi^{u^*}_{(\tau+h,T)}[\mu_{\tau}](y) - \Phi^{u^*}_{(\tau,T)}[\mu^*(\tau)](x) \Big\rangle}{\R^{2d}}{\Bmu_{\tau}(x,y)}. 
\end{aligned}
\end{equation}
In addition, remark that for any $x,y \in \K$, the following Taylor expansion
\begin{equation}
\label{eq:Sensitivity_FlowDiff1}
\begin{aligned}
\Phi^{u^*}_{(\tau+h,T)}[\mu_{\tau}](y) = \Phi^{u^*}_{(\tau,T)}[\mu^*(\tau)](x) & + \D_x \Phi^{u^*}_{(\tau,T)}[\mu^*(\tau)](x)(y-x) + w_{\Bmu_{\tau}}(T,x) \\
& + h \Psi_{\tau}(T,x) + o_{\Rpazo} \Big( |x-y| + W_{2,\Bmu_{\tau}} (\mu^*(\tau),\mu_{\tau}) + h \Big),
\end{aligned}
\end{equation}
holds as a consequence of Corollary \ref{cor:TotalLinFlows}, where $\D_x \Phi_{(\tau,\cdot)}^{u^*}[\mu^*(\tau)](\cdot)$, $w_{\Bmu_{\tau}}(\cdot,\cdot)$ and $\Psi_{\tau}(\cdot,\cdot)$ are defined as in Proposition \ref{prop:LinSpace_Flows}, Theorem \ref{thm:FlowDiff} and Proposition \ref{prop:LinTime_Flows} respectively. Thus, combining \eqref{eq:VarphiDiff1}, \eqref{eq:VarphiDiff2} and \eqref{eq:Sensitivity_FlowDiff1} with the estimate derived in Lemma \ref{lem:Small-oEst} of Appendix \ref{section:AppendixThm}, we obtain
\begin{equation}
\label{eq:VarphiDiff3}
\begin{aligned}
& \varphi \Big( \Phi^{u^*}_{(\tau+h,T)}[\mu_{\tau}](\cdot)_{\#} \mu_{\tau} \Big) - \varphi(\mu^*(T)) \\
& \hspace{0.8cm} = \INTDom{\Big\langle \nabla \varphi(\mu^*(T)) \big( \Phi_{(\tau,T)}^*(x) \big) \, , \, \D_x \Phi_{(\tau,T)}^*(x)(y-x) + w_{\Bmu_{\tau}}(T,x) \Big\rangle}{\R^{2d}}{\Bmu_{\tau}(x,y)} \\
& \hspace{1.2cm} + h \INTDom{\Big\langle \nabla \varphi(\mu^*(T)) \big( \Phi_{(\tau,T)}^*(x) \big) \, , \, \Psi_{\tau}(T,x) \Big\rangle}{\R^{2d}}{\Bmu_{\tau}(x,y)} + o_{\Rpazo} \Big( W_{2,\Bmu_{\tau}}(\mu^*(\tau),\mu_{\tau}) + h \Big),
\end{aligned}
\end{equation}
where we used the condensed notation $(\Phi_{(\tau,t)}^*(\cdot))_{\tau,t \in [0,T]}$ introduced in \eqref{eq:CondensedNota}. By plugging \eqref{eq:VarphiDiff3} into \eqref{eq:ValueFunctionEst1}, we finally recover \eqref{eq:ValueFunctionEstLem}, which concludes the proof of our claim
\end{proof}


\paragraph*{Step 2: Representation formula for the Pontryagin costate} Our goal is now to prove via a backward time propagation that the differential inequality derived in Step 1 implies \eqref{eq:Sensitivity_Thm} at time $\tau \in \T$. To this end, we first need to isolate the ``costate'' part -- i.e. the second marginal -- of the curve $\nu^*(\cdot)$ satisfying the PMP of Theorem \ref{thm:PMP}. 

\begin{prop}[Disintegration representation of state-costate curves]
\label{prop:RepresentationCostate}
Assume that hypotheses \ref{hyp:OCP} hold and let $\nu^*(\cdot) \in \AC([0,T],\Pcal_1(K' \times K'))$ be a solution of the forward-backward Hamiltonian continuity equation \eqref{eq:PMP_Dynamics}, with $K' := B(0,R_r')$ and $R_r' >0$ being as in \eqref{eq:PMP_Est}. Define the curve of measures
\begin{equation}
\label{eq:nu_TDef}
\nu_T^* : t \in [0,T] \mapsto \Big( \Phi_{(t,T)}^* \circ \pi^1 \, , \,  \pi^2 \Big)_{\raisebox{4pt}{$\scriptstyle{\#}$}} \nu^*(t) \in \Pcal(K' \times K'), 
\end{equation}
and for all $t \in [0,T]$, consider its disintegration $\nu_T^*(t) = \INTDom{\sigma^*_x(t)}{\R^d}{\mu^*(T)(x)}$ against $\pi^1_{\#} \nu^*_T(t) = \mu^*(T)$. 

Then for $\mu^*(T)$-almost every $x \in \R^d$, the curve $\sigma_x^*(\cdot) \in \AC([0,T],\Pcal_1(K'))$ is the unique solution of the \textnormal{backward Cauchy problem}
\begin{equation}
\label{eq:BackwardCauchy}
\left\{
\begin{aligned}
\partial_t \sigma_x^*(t) & + \Div_r \Big( \Wpazo_x(t,\sigma^*_x(t)) \sigma_x^*(t) \Big) = 0, \\
\sigma_x^*(T) & = \delta_{\big( -\nabla \varphi(\mu^*(T))(x) \big)}, 
\end{aligned}
\right.
\end{equation}
where the non-local velocity field $\Wpazo_x : [0,T] \times \Pcal_c(\R^d) \times \R^d \rightarrow \R^d$ is defined by 
\begin{equation}
\label{eq:DualVelocityField}
\begin{aligned}
\Wpazo_x(t,\sigma,r) = ~ & - \D_x v \Big(t,\mu^*(t),u^*(t),\Phi_{(T,t)}^*(x) \Big)^{\raisebox{-2pt}{$\scriptstyle{\top}$}} r \\
& - \INTDom{\INTDom{\D_{\mu} v \Big(t,\mu^*(t),u^*(t),\Phi_{(T,t)}^*(y) \Big)\big( \Phi_{(T,t)}^*(x)\big)^{\top} p \,}{\R^d}{\sigma(p)}}{\R^d}{\mu^*(T)(y)},
\end{aligned}
\end{equation}
for $\Lcal^1$-almost every $t \in [0,T]$, any $\sigma \in \Pcal_c(\R^d)$ and all $r \in \R^d$. Moreover, there exists a map $m_r^{\sigma}(\cdot) \in L^1([0,T],\R_+)$ such that
\begin{equation}
\label{eq:ACRegularitySigma}
W_1(\sigma_x^*(t),\sigma_x^*(\tau)) \leq \INTSeg{m_r^{\sigma}(s)}{s}{\tau}{t}, 
\end{equation}
for all times $0 \leq \tau \leq t \leq T$ and $\mu^*(T)$-almost every $x \in \R^d$.
\end{prop}

\begin{proof}
It can be checked as a consequence of the absolute continuity of $t \in [0,T] \mapsto \nu^*(t) \in \Pcal_1(K' \times K')$ and $t \in [0,T] \mapsto \Phi_{(t,T)}^{u^*}[\mu^*(t)](\cdot) \in C^0(K',\R^d)$ that $\nu^*_T(\cdot) \in \AC([0,T],\Pcal_1(K' \times K'))$. Whence by \eqref{eq:KantorovichDuality}, the map $t \in [0,T] \mapsto \INTDom{\xi(x,r)}{\R^{2d}}{\nu_T^*(t)(x,r)} \in \R$ is absolutely continuous for any $\xi \in C^{\infty}_c(\R^{2d},\R)$, and denoting by $\T_{\xi} \subset (0,T)$ the set of its Lebesgue points, one has
\begin{equation}
\label{eq:PointwiseDer1}
\derv{}{t}{} \INTDom{\xi(x,r)}{\R^{2d}}{\nu^*_T(t)(x,r)} = \derv{}{t}{} \INTDom{\xi \Big( \Phi_{(t,T)}^{u^*}[\mu^*(t)](x) , r \Big)}{\R^{2d}}{\nu^*(t)(x,r)},
\end{equation}
for all times $t \in \T_{\xi}$, i.e. for $\Lcal^1$-almost every $t \in [0,T]$. By repeating the same arguments as in the proof of Proposition \ref{prop:LinTime_Flows} below, it can further be shown that 
\begin{equation}
\label{eq:InitialTimeDer}
\tderv{}{t}{} \Phi^*_{(t,T)}(x) = - \D_x \Phi_{(t,T)}^*(x) v(t,\mu^*(t),u^*(t),x), 
\end{equation}
for $\Lcal^1$-almost every $t \in [0,T]$ and any $x \in K'$. 

By combining \eqref{eq:InitialTimeDer} together with Lebesgue's differentiation theorem under the integral and the distributional characterisation \eqref{eq:NonLocal_Distrib2} of the fact that $\nu^*(\cdot)$ solves \eqref{eq:PMP_Dynamics},  we obtain
\begin{equation}
\label{eq:PointwiseDer2}
\begin{aligned}
& \hspace{-1cm} \derv{}{t}{} \INTDom{\xi \big( \Phi_{(t,T)}^*(x) , r \big)}{\R^{2d}}{\nu^*(t)(x,r)} \\
= & \INTDom{\Big( \tderv{}{t}{} \xi \big( \Phi_{(t,T)}^*(x) , r \big) \Big)}{\R^{2d}}{\nu^*(t)(x,r)} \\
& + \INTDom{\Big\langle \nabla_{(x,r)} \xi \big( \Phi_{(t,T)}^*(x) , r \big) , \J_{2d} \nabla_{\nu} \H(t,\nu^*(t),u^*(t))(x,r) \Big\rangle}{\R^{2d}}{\nu^*(t)(x,r)} \\ 
= & - \INTDom{\Big\langle \nabla_x \xi \big( \Phi_{(t,T)}^*(x) ,  r  \big) \, , \, \D_x \Phi_{(t,T)}^*(x) v(t,\mu^*(t),u^*(t),x) \Big\rangle}{\R^{2d}}{\nu^*(t)(x,r)} \\
& + \INTDom{\Big\langle \D_x \Phi_{(t,T)}^*(x)^{\top} \nabla_x \xi \big( \Phi_{(t,T)}^*(x) , r \big) \, , \, v(t,\mu^*(t),u^*(t),x) \Big\rangle}{\R^{2d}}{\nu^*(t)(x,r)} \\
& - \INTDom{\Big\langle \nabla_r \xi \big( \Phi_{(t,T)}^*(x) , r \big) \, , \, \D_x v(t,\mu^*(t),u^*(t),x)^{\top} r \Big\rangle}{\R^{2d}}{\nu^*(t)(x,r)} \\
& - \INTDom{\bigg\langle \nabla_r \xi \big( \Phi_{(t,T)}^*(x) , r \big) \, , \, \INTDom{\D_{\mu} v(t,\mu^*(t),u^*(t),y)(x)^{\top} p \,}{\R^{2d}}{\nu^*(t)(y,p)} \bigg\rangle}{\R^{2d}}{\nu^*(t)(x,r)}
\end{aligned}
\end{equation}
where we used the explicit expression of $\nabla_{\nu} \H(t,\nu^*(t),u^*(t))(\cdot,\cdot)$ given in \eqref{eq:PMP_HamiltonianGrad} along with the condensed notation \eqref{eq:CondensedNota}. By plugging \eqref{eq:PointwiseDer2} into \eqref{eq:PointwiseDer1} while noticing that the two first integrals in the right-hand side cancel each other out, it further holds
\begin{equation}
\label{eq:PointwiseDer3}
\begin{aligned}
& \derv{}{t}{} \INTDom{\xi(x,r)}{\R^{2d}}{\nu^*_T(t)(x,r)} = - \INTDom{\big\langle \nabla_r \xi(x,r) , \D_x v \Big(t,\mu^*(t),u^*(t),\Phi_{(T,t)}^*(x) \Big)^{\raisebox{-4pt}{$\scriptstyle{\top}$}} r \big\rangle}{\R^{2d}}{\nu_T^*(t)(x,r)} \\
& \hspace{0.6cm} - \INTDom{\bigg\langle \nabla_r \xi(x,r) , \INTDom{ \D_{\mu} v \Big(t,\mu^*(t),u^*(t),\Phi_{(T,t)}^*(y) \Big)\big( \Phi_{(T,t)}^*(x)\big)^{\top} p \,}{\R^{2d}}{\nu_T^*(t)(y,p)} \bigg\rangle}{\R^{2d}}{\nu_T^*(t)(x,r)}, 
\end{aligned}
\end{equation}
for any $\xi \in C^{\infty}_c(\R^{2d},\R)$ and $\Lcal^1$-almost every $t \in [0,T]$. Recalling that $\nu_T^*(t) = \INTDom{\sigma^*_x(t)}{\R^d}{\mu^*(T)(x)}$ for all times $t \in [0,T]$, one can rewrite \eqref{eq:PointwiseDer3} as
\begin{equation}
\label{eq:PointwiseDer4}
\INTDom{\bigg( \derv{}{t}{} \INTDom{\xi(x,r)}{\R^d}{\sigma^*_x(t)(r)} - \INTDom{\langle \nabla_r \xi(x,r) , \Wpazo_x(t,\sigma^*_x(t),r) \rangle}{\R^d}{\sigma^*_x(t)(r)} \bigg)}{\R^d}{\mu^*(T)(x)} = 0,
\end{equation}
for $\Lcal^1$-almost every $t \in [0,T]$ and any $\xi \in C^{\infty}_c(\R^{2d},\R)$, where we used Fubini's theorem and \eqref{eq:DualVelocityField}. In particular, by choosing test functions of the form $\xi(x,r) := \phi(x)\zeta(r)$, we can deduce from \eqref{eq:PointwiseDer4} that
\begin{equation}
\label{eq:PointwiseDer5}
\INTDom{\bigg( \derv{}{t}{} \INTDom{\zeta(r)}{\R^d}{\sigma^*_x(t)(r)} - \INTDom{\big\langle \nabla \zeta(r) , \Wpazo_x(t,\sigma^*_x(t),r) \big\rangle}{\R^d}{\sigma^*_x(t)(r)} \bigg) \phi(x)}{\R^d}{\mu^*(T)(x)} = 0,
\end{equation}
for any $\phi,\zeta \in C^{\infty}_c(\R^d,\R)$. For any Borel set $\Omega \subset \R^d$, consider a sequence $(\phi_n) \subset C^{\infty}_c(\R^d,\R)$ converging strongly towards $\mathds{1}_{\Omega}(\cdot)$ in $L^1(\R^d,\R;\mu^*(T))$. By passing to the limit as $n \rightarrow +\infty$ in \eqref{eq:PointwiseDer5}, we obtain
\begin{equation*}
\INTDom{\bigg( \derv{}{t}{} \INTDom{\zeta(r)}{\R^d}{\sigma^*_x(t)(r)} - \INTDom{\big\langle \nabla \zeta(r) , \Wpazo_x(t,\sigma^*_x(t),r) \big\rangle}{\R^d}{\sigma^*_x(t)(r)} \bigg)}{\Omega}{\mu^*(T)(x)} = 0,
\end{equation*}
for every Borel set $\Omega \subset \R^d$. This in turn implies that for $\mu^*(T)$-almost every $x \in \R^d$, it holds
\begin{equation}
\label{eq:PointwiseDer6}
\derv{}{t}{} \INTDom{\zeta(r)}{\R^d}{\sigma^*_x(t)(r)} = \INTDom{\big\langle \nabla \zeta(r) , \Wpazo_x(t,\sigma^*_x(t),r) \big\rangle}{\R^d}{\sigma^*_x(t)(r)},
\end{equation}
for every $\zeta \in C^{\infty}_c(\R^d,\R)$, which, by \eqref{eq:NonLocal_Distrib2}, equivalently implies that the curve $\sigma_x^*(\cdot)$ solves 
\begin{equation*}
\partial_t \sigma^*_x(t) + \Div_r \Big( \Wpazo_x(t,\sigma^*(t)) \sigma^*_x(t) \Big) = 0,
\end{equation*}
in the sense of distributions. In addition, upon noticing that 
\begin{equation*}
\sigma^*_x(T) = \delta_{\big( - \nabla \varphi(\mu^*(T))(x) \big)},
\end{equation*}
because $\nu^*(T) = (\Id , -\nabla \varphi(\mu^*(T))_{\#} \mu^*(T)$, and by uniqueness of the disintegration representation, we conclude that $\sigma^*_x(\cdot)$ is a solution of \eqref{eq:BackwardCauchy} for $\mu^*(T)$-almost every $x \in \R^d$.  

We end the proof by focusing on the regularity properties of $\sigma_x^*(\cdot)$. The fact that $\sigma_x^*(t) \in \Pcal(K')$ is direct since $\nu^*(t) \in \Pcal(K' \times K')$ for all times $t \in [0,T]$. Moreover, under hypotheses \ref{hyp:OCP}, the non-local velocity fields $\Wpazo_x : [0,T] \times \Pcal_c(\R^d) \times \R^d \rightarrow \R^d$ satisfy hypotheses \ref{hyp:CE}-$(i),(ii)$ with constants that are uniform with respect to $x \in \supp(\mu^*(T))$. By Theorem \ref{thm:NonLocalCE}, this implies that solutions of \eqref{eq:BackwardCauchy} are in fact unique, and there exists a map $m_r^{\sigma}(\cdot) \in L^1([0,T],\R_+)$ such that 
\begin{equation*}
W_1(\sigma_x^*(t),\sigma_x^*(\tau)) \leq \INTSeg{m^{\sigma}_r(s)}{s}{\tau}{t},
\end{equation*}
for all times $0 \leq \tau \leq t \leq T$ and $\mu^*(T)$-almost every $x \in \R^d$. 
\end{proof}

\begin{rmk}[Necessity and sufficiency of the disintegration representation]
In the earlier works \cite{PMPWassConst,SetValuedPMP,PMPWass}, the existence of a state-costate curve $\nu^*(\cdot)$ satisfying the PMP in Wasserstein spaces -- either the version exposed in Theorem \ref{thm:PMP} or its counterpart with feedback controls -- was obtained by constructing explicitly the curves $\sigma_x^*(\cdot)$ for $\mu^*(T)$-almost every $x \in \R^d$, and then defining $\nu^*(\cdot)$ by \eqref{eq:nu_TDef}, with $\nu^*_T(\cdot)$ being itself built via the disintegration formula $\nu^*_T(t) = \INTDom{\sigma_x^*(t)}{\R^d}{\mu^*(T)(x)}$. Therefore, Proposition \ref{prop:RepresentationCostate} can be seen as a reciprocal statement to that of Theorem \ref{thm:PMP}-$(i),(ii)$, as they together establish the uniqueness of state-costate curves with prescribed first marginal which are solutions of the Hamiltonian system \eqref{eq:PMP_Dynamics}. In that sense, this proposition restores the known fact in classical optimal control theory that the costate curve associated with an optimal trajectory-control pair is unique.
\end{rmk}


\paragraph*{Step 3: Time-constancy of two functionals} In Step 1 above, we derived the estimate
\begin{equation}
\label{eq:DiffIneqValue1}
\begin{aligned}
\Vcal(\tau+h&,\mu_{\tau}) - \Vcal(\tau,\mu^*(\tau)) \\
& \leq \INTDom{\Big\langle \nabla \varphi(\mu^*(T)) \big( \Phi_{(\tau,T)}^*(x) \big) \, , \, \D_x \Phi_{(\tau,T)}^*(x)(y-x) + w_{\Bmu_{\tau}}(T,x) \Big\rangle}{\R^{2d}}{\Bmu_{\tau}(x,y)} \\
& \hspace{0.35cm} + h \INTDom{\Big\langle \nabla \varphi(\mu^*(T)) \big( \Phi_{(\tau,T)}^*(x) \big) \, , \, \Psi_{\tau}(T,x) \Big\rangle}{\R^{2d}}{\Bmu_{\tau}(x,y)} + o_{\Rpazo} \Big( W_{2,\Bmu_{\tau}}(\mu^*(\tau),\mu_{\tau}) + h \Big),
\end{aligned}
\end{equation}
for every $\tau \in \T$, any $h \in \R$ such that $\tau + h \in [0,T]$ and all $\mu_{\tau} \in \Pcal(B(0,r'))$, where $\Bmu_{\tau} \in \Gamma(\mu^*(\tau),\mu_{\tau})$ is arbitrary. Our goal in what follows is to show that \eqref{eq:DiffIneqValue1} in fact yields \eqref{eq:SuperdifferentialValue_Ineq} with $(\tau_1,\mu_1) := (\tau,\mu^*(\tau))$ and $(\tau_2,\mu_2) := (\tau+h,\mu_{\tau})$. 

Let us choose $\Bmu_{\tau} := \gamma_{\tau} \in \Gamma_o(\mu^*(\tau),\mu_{\tau})$ and consider the partially transported plan
\begin{equation*}
\gamma_{\tau}^T := \big( \Phi_{(\tau,T)}^* \circ \pi^1, \pi^2 \big)_{\#} \gamma_{\tau} \in \Gamma(\mu^*(T),\mu_{\tau}), 
\end{equation*}
along with its disintegration $\gamma_{\tau}^T := \INTDom{\gamma_{\tau,x}^T}{\R^d}{\mu^*(T)(x)}$ against $\pi^1_{\#} \gamma_{\tau}^T = \mu^*(T)$. Moreover, let $\nu^*(\cdot)$ be the unique state-costate curve satisfying the PMP of Theorem \ref{thm:PMP}, and $\sigma_x^*(\cdot)$ be as in Proposition \ref{prop:RepresentationCostate} for $\mu^*(T)$-almost every $x \in \R^d$. By introducing the curve $\Bnu^*(\cdot) \in \AC([0,T],\Pcal_1(\K \times \K \times \K))$, defined for all times $t \in [0,T]$ by 
\begin{equation}
\label{eq:Bnu_TDef}
\Bnu^*(t) := \Big( \Phi_{(T,t)}^* \circ \pi^1,\pi^2,\pi^3 \Big)_{\raisebox{4pt}{$\scriptstyle{\#}$}} \Bnu^*_T(t) \qquad \text{with} \qquad \Bnu^*_T(t) := \INTDom{\Big( \sigma_x^*(t) \times \gamma_{\tau,x}^T \Big)}{\R^d}{\mu^*(T)(x)}, 
\end{equation}
we can reformulate \eqref{eq:DiffIneqValue1} as
\begin{equation}
\label{eq:DiffIneqValue2}
\begin{aligned}
\Vcal(\tau+h&,\mu_{\tau}) - \Vcal(\tau,\mu^*(\tau)) \\
& \leq \INTDom{\Big\langle - r \, , \, \D_x \Phi^*_{(\tau,T)} \big( \Phi^*_{(T,\tau)}(x) \big) \Big( y - \Phi^*_{(T,\tau)}(x) \Big) + w_{\gamma_{\tau}} \big(T, \Phi^*_{(T,\tau)}(x) \big) \Big\rangle}{\R^{3d}}{\Bnu^*(T)(x,r,y)} \\
& \hspace{0.4cm} + h \INTDom{\big\langle -r \, , \, \Psi_{\tau} \big(T,\Phi^*_{(T,\tau)}(x) \big) \big\rangle}{\R^{2d}}{\nu^*(T)(x,r)} + o_{\Rpazo} \Big( W_2(\mu^*(\tau),\mu_{\tau}) + h \Big).
\end{aligned}
\end{equation}
Here, we also used the fact that $W_{2,\gamma_{\tau}}(\mu^*(\tau),\mu_{\tau}) = W_2(\mu^*(\tau),\mu_{\tau})$ since $\gamma_{\tau} \in \Gamma_o(\mu^*(\tau),\mu_{\tau})$. 

In the following lemma, we state a technical result showing that the estimate of  \eqref{eq:DiffIneqValue2} can be propagated back from the final time $T >0$ to $\tau \in \T$. Its proof being somewhat heavy and relying on ideas already explored in \cite{PMPWassConst,SetValuedPMP,PMPWass}, we postpone it to Appendix \ref{section:AppendixLem} below.

\begin{lem}[Time-constancy of two functionals]
\label{lem:Consistency}
Let $\Hpazo_1, \Hpazo_2 : [0,T] \rightarrow \R$ be respectively defined by  
\begin{equation}
\label{eq:Hpazo1Def}
\Hpazo_1(t) := \INTDom{\Big\langle r \, , \, \D_x \Phi^*_{(\tau,t)} \big( \Phi^*_{(t,\tau)}(x) \big) \Big( y - \Phi^*_{(t,\tau)}(x) \Big) + w_{\gamma_{\tau}} \big(t, \Phi^{u^*}_{(t,\tau)}(x) \big) \Big\rangle}{\R^{3d}}{\Bnu^*(t)(x,r,y)}, 
\end{equation}
and 
\begin{equation}
\label{eq:Hpazo2Def}
\Hpazo_2(t) := \INTDom{\big\langle r \, , \, \Psi_{\tau} \big( t,\Phi_{(t,\tau)}^*(x) \big) \big\rangle}{\R^{2d}}{\nu^*(t)(x,r)}, 
\end{equation}
for all times $t \in [0,T]$. Then, both maps $\Hpazo_1(\cdot),\Hpazo_2(\cdot)$ are absolutely continuous, with 
\begin{equation*}
\tderv{}{t}{} \Hpazo_1(t) = \tderv{}{t}{} \Hpazo_2(t) = 0,
\end{equation*}
for $\Lcal^1$-almost every $t \in [0,T]$, and are therefore constant over $[0,T]$.  
\end{lem}


\paragraph*{Step 4 : Proof of the sensitivity relation}

By combining the results of the previous three steps, we can prove the claims of Theorem \ref{thm:Sensitivity1}. 

\begin{proof}[Proof of Theorem \ref{thm:Sensitivity1}]
As a direct consequence of Lemma \ref{lem:Consistency}, and upon identifying terms in the right-hand side of \eqref{eq:DiffIneqValue2} using \eqref{eq:Hpazo1Def} and \eqref{eq:Hpazo2Def}, one has
\begin{equation}
\label{eq:DiffIneqValue3}
\begin{aligned}
\Vcal(\tau+h,\mu_{\tau}) - \Vcal(\tau,\mu^*(\tau)) & \leq -\Hpazo_1(T) - h \Hpazo_2(T) + o_{\Rpazo} \big( W_2(\mu^*(\tau),\mu_{\tau}) + h \big) \\
& = \, - \Hpazo_1(\tau) - h \Hpazo_2(\tau) + o_{\Rpazo} \big( W_2(\mu^*(\tau),\mu_{\tau}) + h \big).
\end{aligned}
\end{equation}
Now, by the definitions of $\D_x \Phi_{(\tau,\cdot)}^*(\cdot)$, $w_{\gamma_{\tau}}(\cdot,\cdot)$ and $\Psi_{\tau}(\cdot,\cdot)$ given in Proposition \ref{prop:LinSpace_Flows}, Theorem \ref{thm:FlowDiff} and Proposition \ref{prop:LinTime_Flows} respectively together with the absolute continuity of $\nu^*(\cdot)$ and $\Bnu^*(\cdot)$, one further has
\begin{equation}
\label{eq:ExpressionHpazoTau1}
\Hpazo_1(\tau) = \INTDom{\langle r , y-x \rangle}{\R^{3d}}{\Bnu^*(\tau)(x,r,y)} = -\INTDom{\langle r , y-x \rangle}{\R^{3d}}{\Big( \big( \pi^1,-\pi^2,\pi^3 \big)_{\#} \Bnu^*(\tau) \Big)(x,r,y)}, 
\end{equation}
and 
\begin{equation}
\label{eq:ExpressionHpazoTau2}
\Hpazo_2(\tau) = -\INTDom{\langle r , v(\tau,\mu^*(\tau),u^*(\tau),x) \rangle}{\R^{2d}}{\nu^*(\tau)(x,r)} = -\H(\tau,\nu^*(\tau),u^*(\tau)), 
\end{equation}
where we used the expression \eqref{eq:PMP_Hamiltonian} of the Hamiltonian $\H : [0,T] \times \Pcal_c(\R^{2d}) \times U \rightarrow \R$. Whence, upon combining \eqref{eq:DiffIneqValue3}, \eqref{eq:ExpressionHpazoTau1} and \eqref{eq:ExpressionHpazoTau2}, it holds 
\begin{equation}
\label{eq:DiffIneqValue3Bis}
\begin{aligned}
\Vcal(\tau+h,\mu_{\tau}) - \Vcal(\tau,\mu^*(\tau)) & \leq \INTDom{\langle r , y-x \rangle}{\R^{3d}}{\Big( \big( \pi^1,-\pi^2,\pi^3 \big)_{\#} \Bnu^*(\tau) \Big)(x,r,y)} \\
& \hspace{1.5cm} + h \, \H(\tau,\nu^*(\tau),u^*(\tau)) + o_{\Rpazo} \big( W_2(\mu^*(\tau),\mu_{\tau}) + h \big).
\end{aligned}
\end{equation}
Let $\xi \in C^0(\R^{2d},\R)$ and observe that by the construction of $\Bnu^*(\cdot)$ displayed in \eqref{eq:Bnu_TDef}, one has
\begin{equation*}
\begin{aligned}
\INTDom{\xi(x,r)}{\R^{3d}}{\Bnu^*(\tau)(x,r,y)} & = \INTDom{\xi \big( \Phi_{(T,\tau)}^*(x),r \big)}{\R^{3d}}{\Bnu_T^*(\tau)(x,r,y)} \\
& = \INTDom{\xi \big( \Phi_{(T,\tau)}^*(x),r \big)}{\R^{2d}}{\nu_T^*(\tau)(x,r)} = \INTDom{\xi(x,r)}{\R^{2d}}{\nu^*(\tau)(x,r)}, 
\end{aligned}
\end{equation*}
or equivalently $\pi^{1,2}_{\#} \Bnu^*(\tau) = \nu^*(\tau)$. By performing the same computations, one can also show that $\pi^{1,3}_{\#} \Bnu^*(\tau) = \gamma_{\tau} \in \Gamma_o(\mu^*(\tau),\mu_{\tau})$, which is equivalent to saying that $\Bnu^*(\tau) \in \BGamma_o^{1,3}(\nu^*(\tau),\mu_{\tau})$ following the notations introduced in Definition \ref{def:Generalised_Subdiff}. This fact together with \eqref{eq:DiffIneqValue3Bis} then yields
\begin{equation}
\label{eq:DiffIneqValue4}
\begin{aligned}
\Vcal(\tau+h,\mu_{\tau}) - \Vcal(\tau,\mu^*(\tau)) & \leq \INTDom{\langle r , y-x \rangle}{\R^{3d}}{\tilde{\Bnu}(\tau)(x,r,y)} \\
& \hspace{1.5cm} + h \, \H(\tau,\nu^*(\tau),u^*(\tau)) + o_{\Rpazo} \Big( W_2(\mu^*(\tau),\mu_{\tau}) + h \Big),
\end{aligned}
\end{equation}
for all times $\tau \in \T$, every $h \in \R$ such that $\tau + h \in [0,T]$ and any $\mu_{\tau} \in \Pcal(B(0,r'))$ for some $r' >0$, where $\tilde{\Bnu}(\tau) \in \BGamma_o^{1,3} \big( (\pi^1,-\pi^2)_{\#}\nu^*(\tau), \mu_{\tau} \big)$. Recalling the definition \eqref{eq:SuperdifferentialValue_Ineq} of the mixed superdifferential of the value function, and observing that $\T \subset (0,T)$ has full $\Lcal^1$-measure, we conclude from \eqref{eq:DiffIneqValue4} that
\begin{equation*}
\Big( \H(t,\nu^*(t),u^*(t)) \, , \, (\pi^1,-\pi^2)_{\#} \nu^*(t) \Big) \in \Bpartial^+ \Vcal(t,\mu^*(t)),
\end{equation*}
for $\Lcal^1$-almost every $t \in [0,T]$, which ends the proof of Theorem \ref{thm:Sensitivity1}. 
\end{proof}

By following the same strategy, we can also derive a sensitivity result which only involves the localised Fr\'echet superdifferential of the value function with respect to the measure variable. The main difference between this partial sensitivity relation and the total one exposed in Theorem \ref{thm:Sensitivity1} is that the former holds true for \textit{all times} $t \in [0,T]$, instead of $\Lcal^1$-almost every times.

\begin{prop}[Sensitivity relation with respect to the measure variable]
\label{prop:PointwiseSensitivity}
Let $\mu^0 \in \Pcal(B(0,r))$ for some $r>0$, and suppose that hypotheses \ref{hyp:OCP} hold. Given a minimiser $(\mu^*(\cdot),u^*(\cdot))$ for $(\Ppazo)$, denote by $\nu^*(\cdot) \in \AC([0,T],\Pcal_1(K' \times K'))$ the state-costate curve satisfying the PMP of Theorem \ref{thm:PMP}, with $K' := B(0,R_r')$ being given as in \eqref{eq:PMP_Est}. 

Then, the following sensitivity relation
\begin{equation}
\label{eq:Prop_PointwiseSensitivity}
\big( \pi^1 , -\pi^2 \big)_{\#} \nu^*(t) \in \Bpartial^+_{\mu} \Vcal(t,\mu^*(t)), 
\end{equation}
holds for all times $t \in [0,T]$. 
\end{prop}

\begin{proof}
To prove the statement of Proposition \ref{prop:PointwiseSensitivity}, one can repeat the arguments of Step 1 to 4 above while fixing $h=0$, and replacing the set $\T \subset (0,T)$ by the whole time interval $[0,T]$.
\end{proof}


\subsection{Sensitivity relations expressed in terms of Dini superdifferentials}
\label{subsection:GateauSensitivity}

In section \ref{subsection:FrechetSensitivity} above, we derived general sensitivity relations in terms of the Fr\'echet superdifferential of the value function, defined in the sense of \cite[Chapter 10]{AGS}. However in several control-theoretic applications, we will need to apply sensitivity results in cases where the test measures are of the form $\mu_{\tau} := (\Id + \F)_{\#} \mu^*(\tau)$, where $\F \in L^{\infty}(\R^d,\R^d;\mu)$ is not an optimal transport direction. This fact motivates the introduction of the following notion of \textit{Dini superdifferential}.

\begin{Def}[Dini superdifferential of the value function]
\label{def:Value_GateauSuperDiff}
Given $(\tau,\mu) \in [0,T] \times \Pcal_c(\R^d)$ and $(h,\F) \in \R \times L^{\infty}(\R^d,\R^d;\mu)$, we define the \textnormal{upper Dini derivative} of $\Vcal(\cdot,\cdot)$ at $(\tau,\mu)$ in the direction $(h,\F)$ as 
\begin{equation}
\label{eq:UpperDiniDerivative}
\dn^+ \Vcal(\tau,\mu)(h,\F) := \limsup_{\epsilon \rightarrow 0^+} \bigg[ \, \frac{\Vcal \big( \tau + \epsilon h, (\Id + \epsilon \F)_{\#} \mu \big) - \Vcal(\tau,\mu)}{\epsilon} \, \bigg].
\end{equation}
Then, we say that a pair $(\delta,\xi) \in \R \times L^2(\R^d,\R^d;\mu)$ belongs to the \textnormal{localised Dini superdifferential} $\eth^+ \Vcal(\tau,\mu)$ of $\Vcal(\cdot,\cdot)$ at $(\tau,\mu)$ if 
\begin{equation}
\label{eq:Value_GateauSuperDiff}
\dn^+ \Vcal(\tau,\mu)(h,\F) \leq \INTDom{\langle \xi(x) , \F(x) \rangle}{\R^d}{\mu(x)} + \delta \, h, 
\end{equation}
for all $(h,\F) \in \R \times L^{\infty}(\R^d,\R^d;\mu)$. We analogously say that $\xi \in L^2(\R^d,\R^d;\mu)$ belongs to the localised Dini superdifferential $\eth^+_{\mu} \Vcal(\tau,\mu)$ with respect to the measure variable of $\Vcal(\tau,\cdot)$ at $\mu$ if \eqref{eq:Value_GateauSuperDiff} holds for all $\F \in L^{\infty}(\R^d,\R^d;\mu)$ with $h = \delta = 0$.
\end{Def}

\begin{rmk}[On the definition of Dini superdifferentials]
The so-called Dini-Hadamard superdifferentials are usually defined in vector spaces by means of \textnormal{contingent directional derivatives} (see e.g. \cite[Section 6.1]{Aubin1984}). While the sensitivity relation expressed below could also be proven for an analogue of the Dini-Hadamard superdifferential stated on the space $[0,T] \times \Pcal_c(\R^d)$, we chose to use the simpler notion of Dini superdifferential introduced above to lighten the presentation.
\end{rmk}

In the following theorem, we state another central result of this manuscript, which provides Dini-type sensitivity relations involving non-optimal transport directions.

\begin{thm}[Dini-type sensitivity relations for Pontryagin costates]
\label{thm:Sensitivity2}
Suppose that hypotheses \ref{hyp:OCP} hold and let $K := B(0,R_r)$ be as in Lemma \ref{lem:AdmEst}. Given a minimiser $(\mu^*(\cdot),u^*(\cdot)) \in \AC([0,T],\Pcal_1(K)) \times \U$, denote by $\nu^*(\cdot) \in \AC([0,T],\Pcal_1(K' \times K'))$ the state-costate curve satisfying the PMP of Theorem \ref{thm:PMP}, with $K' := B(0,R_r')$ being given as in \eqref{eq:PMP_Est}. 

Then, the following sensitivity relation 
\begin{equation}
\label{eq:Sensitivity_ThmGateau1}
\Big( \H(t,\nu^*(t),u^*(t)) \, , \, -\bar{\nu}^*(t) \Big) \in \eth^+ \Vcal(t,\mu^*(t)), 
\end{equation}
holds for $\Lcal^1$-almost every $t \in [0,T]$, where $\bar{\nu}^*(t) \in L^{\infty}(\R^d,\R^d;\mu^*(t))$ denotes the barycentric projection of the state-costate curve $\nu^*(t)$ onto $\pi^1_{\#} \nu^*(t) = \mu^*(t)$ defined in the sense of \eqref{eq:Barycenter}. Moreover, the following sensitivity relation with respect to the measure variable
\begin{equation}
\label{eq:Sensitivity_ThmGateau2}
- \bar{\nu}^*(t) \in \eth^+_{\mu} \Vcal(t,\mu^*(t)), 
\end{equation}
holds for all times $t \in [0,T]$.
\end{thm}

\begin{proof}
We apply the same strategy as in Section \ref{subsection:FrechetSensitivity} above, up to some minor modifications. As before, we will use the condensed notation 
\begin{equation*}
(\Phi^*_{(\tau,t)}(\cdot))_{\tau,t \in [0,T]} := (\Phi_{(\tau,t)}^{u^*}[\mu^*(\tau)](\cdot))_{\tau,t \in [0,T]}, 
\end{equation*}
already introduced in \eqref{eq:CondensedNota} to refer to non-local flows defined along $(\mu^*(\cdot),u^*(\cdot))$.

First, fix $\tau \in \T$ and observe that for every $\F \in L^{\infty}(\R^d,\R^d;\mu^*(\tau))$, the following support estimate 
\begin{equation}
\label{eq:SupportPushSize}
\supp \Big( (\Id + \epsilon \F)_{\#} \mu^*(\tau) \Big) \subset B \Big(0,R_r \, + \NormL{\F(\cdot)}{\infty}{\mu^*(\tau)} \hspace{-0.1cm} \Big), 
\end{equation}
holds for every $\epsilon \in (0,1]$, since $\supp(\mu^*(\tau)) \subset K := B(0,R_r)$. Then, by applying Lemma \ref{lem:ValueFunctionDiffEst} and Proposition \ref{prop:RepresentationCostate} in the particular case where 
\begin{equation}
\label{eq:Gateau_Bmudef}
\mu_{\tau} := (\Id + \epsilon \F)_{\#} \mu^*(\tau) \qquad \text{and} \qquad \Bmu_{\tau} := (\Id,\Id + \epsilon\F)_{\#}\mu^*(\tau), 
\end{equation}
for some $\epsilon > 0$, we can derive the following differential estimate 
\begin{equation}
\label{eq:DiffIneqGateau1}
\begin{aligned}
\Vcal \Big( \tau + \epsilon h&, (\Id + \epsilon \F)_{\#} \mu^*(\tau) \Big) - \Vcal(\tau,\mu^*(\tau)) \\
& \leq \epsilon \INTDom{\Big\langle - r, \D_x \Phi^*_{(\tau,T)} \big( \Phi^*_{(T,\tau)}(x) \big) \F \big( \Phi^*_{(T,\tau)}(x) \big) + w_{\F} \big(T , \Phi^*_{(T,\tau)}(x) \big) \Big\rangle}{\R^{2d}}{\nu^*(T)(x,r)} \\
& \hspace{0.4cm} + \epsilon h \INTDom{\Big\langle - r , \Psi_{\tau} \big( T, \Phi_{(T,\tau)}^*(x) \big)\Big\rangle}{\R^{2d}}{\nu^*(T)(x,r)} + o_{h,\F}(\epsilon),
\end{aligned}
\end{equation}
similarly to \eqref{eq:DiffIneqValue2}. Here, we introduced the notation $w_{\F}(\cdot,\cdot) := \tfrac{1}{\epsilon} w_{\Bmu_{\tau}}(\cdot,\cdot)$, while the dependence on $\F(\cdot)$ of the remainder term $o_{h,\F}(\epsilon)$ appears as a consequence of \eqref{eq:SupportPushSize}, and since
\begin{equation*}
W_{2,\Bmu_{\tau}} \Big(\mu^*(\tau),(\Id + \epsilon \F)_{\#} \mu^*(\tau) \Big) \leq \epsilon \NormL{\F(\cdot)}{\infty}{\mu^*(\tau)}, 
\end{equation*}
for any $\epsilon > 0$. By Theorem \ref{thm:FlowDiff} applied with $\Bmu_{\tau}$ defined as in \eqref{eq:Gateau_Bmudef}, one can check that the map $t \in [0,T] \mapsto w_{\F}(t,x) \in \R^d$ defined for all $x \in K$ is the unique solution of the Cauchy problem
\begin{equation}
\label{eq:LinearisedCauchyMeasureF}
\left\{
\begin{aligned}
\partial_t w_{\F}(t,x) & = \D_x v \Big( t,\mu^*(t),\Phi_{(\tau,t)}^*(x) \Big) w_{\F}(t,x) \\
& \hspace{0.4cm} + \INTDom{\D_{\mu} v \Big( t,\mu^*(t),\Phi_{(\tau,t)}^*(x) \Big) \big( \Phi_{(\tau,t)}^*(y) \big) \Big(\D_x \Phi_{(\tau,t)}^*(y) \F(y) + w_{\F}(t,y) \Big)}{\R^d}{\mu^*(\tau)(y)}, \\
w_{\F}(\tau,x) & = 0, 
\end{aligned}
\right.
\end{equation}
where we used the fact that $w_{\Bmu_{\tau}}(\cdot,\cdot) = \epsilon w_{\F}(\cdot,\cdot)$ and the linearity of \eqref{eq:LinearisedMeasure_Cauchy}. Furthermore, by adapting the proof of Lemma \ref{lem:Consistency}, one can show that the mappings 
\begin{equation*}
\Hpazo_1 : t \in [0,T] \mapsto \INTDom{\Big\langle r \, , \, \D_x \Phi^*_{(\tau,t)} \big( \Phi^*_{(t,\tau)}(x) \big) \F \big( \Phi^*_{(t,\tau)}(x) \big) + w_{\F} \big( t , \Phi_{(t,\tau)}^*(x) \big) \Big\rangle}{\R^{2d}}{\nu^*(t)(x,r)}, 
\end{equation*}
and 
\begin{equation*}
\Hpazo_2 : t \in [0,T] \mapsto \INTDom{\big\langle r \, , \, \Psi_{\tau} \big( t,\Phi^*_{(t,\tau)}(x) \big) \big\rangle}{\R^{2d}}{\nu^*(t)(x,r)}, 
\end{equation*}
are constant over $[0,T]$, with 
\begin{equation}
\label{eq:Gateau_Hpazoval}
\Hpazo_1(\tau) = \INTDom{\langle r , \F(x) \rangle}{\R^{2d}}{\nu^*(\tau)(x,r)} \qquad \text{and} \qquad \Hpazo_2(\tau) = -\H(\tau,\nu^*(\tau),u^*(\tau)).
\end{equation}
Thus by combining \eqref{eq:DiffIneqGateau1} and \eqref{eq:Gateau_Hpazoval} while recalling the definition \eqref{eq:Barycenter} of barycentric projection, we recover 
\begin{equation}
\label{eq:DiffIneqGateau2}
\begin{aligned}
\Vcal \Big( \tau + \epsilon h, (\Id + \epsilon \F)_{\#} \mu^*(\tau) \Big) - \Vcal(\tau,\mu^*(\tau)) & \leq \epsilon \INTDom{\langle -\bar{\nu}^*(\tau,x) , \F(x) \rangle}{\R^{2d}}{\mu^*(\tau)(x)} \\
& \hspace{1.85cm} + \epsilon h \,  \H(\tau,\nu^*(\tau),u^*(\tau)) + o_{h,\F}(\epsilon).
\end{aligned}
\end{equation}
Dividing both sides of \eqref{eq:DiffIneqGateau2} by $\epsilon >0$, taking the limsup as $\epsilon \rightarrow 0^+$ and using the definition \eqref{eq:UpperDiniDerivative} of upper Dini derivative, we obtain 
\begin{equation*}
\dn^+ \Vcal(\tau,\mu^*(\tau))(h,\F) \leq \INTDom{\langle -\bar{\nu}^*(t,x), \F(x) \rangle}{\R^{2d}}{ \mu^*(\tau)(x)} + h \, \H(\tau,\nu^*(\tau),u^*(\tau)), 
\end{equation*}
for all $(h,\F) \in \R \times L^{\infty}(\R^d,\R^d;\mu^*(\tau))$. Since $\T \subset (0,T)$ has full $\Lcal^1$-measure and $\tau \in \T$ is arbitrary, this is in turn equivalent to the sensitivity relation \eqref{eq:Sensitivity_ThmGateau1}. 

Similarly to Proposition \ref{prop:PointwiseSensitivity}, one can repeat exactly the same proof strategy with $h=0$ and by replacing $\T$ by the whole time-interval $[0,T]$ in order to recover the sensitivity relation \eqref{eq:Sensitivity_ThmGateau2} expressed with respect to the measure variable.
\end{proof}

\begin{rmk}[Comparisons between Theorem \ref{thm:Sensitivity1} and Theorem \ref{thm:Sensitivity2}]
As previously mentioned, Theorem \ref{thm:Sensitivity1} is a sensitivity relation expressed in terms of Fr\'echet superdifferentials defined in the sense of \cite[Chapter 10]{AGS}. While the latter constitutes an insightful structural result, it is less suited to control problems than the Dini-type sensitivity relation of Theorem \ref{thm:Sensitivity2}, and of a slightly different nature. Indeed, when $\mu \in \Pcal_c(\R^d)$ and $\F \in L^{\infty}(\R^d,\R^d;\mu)$ are arbitrary, Theorem \ref{thm:Sensitivity2} transcribes a kind of strong superdifferentiability of the value function in the sense of \cite[Definition 10.3.1]{AGS}, but only for test measures expressed as perturbations of the identity (see also \cite[Remark 10.3.2]{AGS}). This property builds explicitly on the product structure of the disintegrations against $\mu^*(T)$ of the $3$-plans $\Bnu^*_T(\cdot)$ appearing in Step 3 and Step 4 of the proof of Theorem \ref{thm:Sensitivity1}. For this reason, it cannot be concluded from the approach we developed that $(\pi^1,-\pi^2)_{\#} \nu^*(t)$ is a strong Fr\'echet superdifferential of $\Vcal(t,\cdot)$ at $\mu^*(t)$ in general, so that Theorem \ref{thm:Sensitivity2} cannot be deduced directly from Theorem \ref{thm:Sensitivity1}.
\end{rmk}


\section{Applications in mean-field optimal control}
\label{section:Applications}

In this section, we investigate various interesting characterisations and properties of optimal trajectories for mean-field optimal control problems, some of which are established as a consequence of the semiconcavity and sensitivity properties of the value function explored in Section \ref{section:Semiconcavity} and Section \ref{section:Sensitivity}. 


\subsection{Propagation of Gateaux differentiability for the value function}
\label{subsection:PropagationReg}

In this section, we prove \textit{forward-time} propagation results along optimal trajectories for the Dini subdifferentials of the value function with respect to the measure variable. 

\begin{Def}[Dini subdifferential of the value function]
\label{def:Value_GateauSubDiff}
Given $(\tau,\mu) \in [0,T] \times \Pcal_c(\R^d)$ and $(h,\F) \in \R \times L^{\infty}(\R^d,\R^d;\mu)$, we define the \textnormal{lower Dini derivative} of $\Vcal(\cdot,\cdot)$ at $(\tau,\mu)$ in the direction $(h,\F)$ as 
\begin{equation}
\label{eq:LowerDiniDerivative}
\dn^- \Vcal(\tau,\mu)(h,\F) := \liminf_{\epsilon \rightarrow 0^+} \bigg[ \, \frac{\Vcal \big( \tau + \epsilon h, (\Id + \epsilon \F)_{\#} \mu \big) - \Vcal(\tau,\mu)}{\epsilon} \, \bigg].
\end{equation}
Then, we say that a pair $(\delta,\xi) \in \R \times L^2(\R^d,\R^d;\mu)$ belongs to the \textnormal{localised Dini subdifferential} $\eth^- \Vcal(\tau,\mu)$ of $\Vcal(\cdot,\cdot)$ at $(\tau,\mu)$ if 
\begin{equation}
\label{eq:Value_GateauSubDiff}
\dn^- \Vcal(\tau,\mu)(h,\F) \geq \INTDom{\langle \xi(x) , \F(x) \rangle}{\R^d}{\mu(x)} + \delta \, h, 
\end{equation}
for all $(h,\F) \in \R \times L^{\infty}(\R^d,\R^d;\mu)$. We analogously say that $\xi \in L^2(\R^d,\R^d;\mu)$ belongs to the localised Dini subdifferential $\eth^-_{\mu} \Vcal(\tau,\mu)$ with respect to the measure variable of $\Vcal(\tau,\cdot)$ at $\mu$ if \eqref{eq:Value_GateauSubDiff} holds for all $\F \in L^{\infty}(\R^d,\R^d;\mu)$ with $h = \delta = 0$.
\end{Def}

In the following proposition, we start by showing that the Dini subdifferential of the value function is single-valued whenever it is non-empty, and that the value function admits Gateaux derivatives at the corresponding points. 

\begin{prop}[Subdifferential and Gateaux derivatives of the value function]
\label{prop:UniqueSubdiff}
Assume that hypotheses \ref{hyp:OCP} hold and let $(\mu^*(\cdot),u^*(\cdot))$ be an optimal trajectory-control pair for $(\Ppazo)$. If it holds that $\eth^-_{\mu} \Vcal(\tau,\mu^*(\tau)) \neq \emptyset$ for some $\tau \in [0,T]$, then 
\begin{equation*}
\eth^-_{\mu} \Vcal(\tau,\mu^*(\tau)) = \{ -\bar{\nu}^*(\tau)\}, 
\end{equation*}
and the application $\mu \in \Pcal_c(\R^d) \mapsto \Vcal(t,\mu) \in \R$ admits \textnormal{directional derivatives} at $\mu^*(\tau)$, namely
\begin{equation}
\label{eq:GateauDerivativeValue}
\begin{aligned}
\dn \Vcal(\tau,\mu^*(\tau))(0,\F) := \, & \lim_{\epsilon \rightarrow 0^+} \bigg[ \frac{\Vcal(\tau,(\Id + \epsilon \F)_{\#} \mu^*(\tau)) - \Vcal(\tau,\mu^*(\tau))}{\epsilon} \bigg] \\
= \, & \INTDom{\langle -\bar{\nu}^*(\tau,x) , \F(x) \rangle}{\R^d}{\mu^*(\tau)(x)}, 
\end{aligned}
\end{equation}
for every $\F \in L^{\infty}(\R^d,\R^d;\mu^*(\tau))$.  
\end{prop}

\begin{proof}
Let $\xi_{\tau} \in \eth^-_{\mu} \Vcal(\tau,\mu^*(\tau))$ and $\F \in L^{\infty}(\R^d,\R^d;\mu^*(\tau))$ be arbitrary. By definition of the Dini sub and superdifferentials, together with the sensitivity relations of Theorem \ref{thm:Sensitivity2}, it holds
\begin{equation}
\label{eq:DoubleIneqSubSuperDiff}
\begin{aligned}
\INTDom{\langle -\bar{\nu}^*(\tau,x) , \F(x) \rangle}{\R^d}{\mu^*(\tau)(x)} & \geq \dn^+ \Vcal(\tau,\mu^*(\tau))(0,\F) \\
& \geq \dn^- \Vcal(\tau,\mu^*(\tau))(0,\F) \geq \INTDom{\langle \xi_{\tau}(x) , \F(x) \rangle}{\R^d}{\mu^*(t)(x)},
\end{aligned}
\end{equation}
which can be rewritten as 
\begin{equation*}
\INTDom{\langle \bar{\nu}^*(\tau,x) + \xi_{\tau}(x) , \F(x) \rangle}{\R^d}{\mu^*(\tau)(x)} \leq 0, 
\end{equation*}
for every $\F \in L^{\infty}(\R^d,\R^d;\mu^*(\tau))$. By a simple density argument, the previous inequality yields $\xi_{\tau} = -\bar{\nu}^*(\tau)$ in $L^2(\R^d,\R^d;\mu^*(\tau))$, which actually means that $\eth^-_{\mu} \Vcal(\tau,\mu^*(\tau)) = \{ -\bar{\nu}^*(\tau)\}$ because $\xi_{\tau}$ is an arbitrary Dini subdifferential. Upon using this fact together with \eqref{eq:DoubleIneqSubSuperDiff}, one further has
\begin{equation*}
\dn^- \Vcal(\tau,\mu^*(\tau))(0,\F) = \dn^+ \Vcal(\tau,\mu^*(\tau))(0,\F), 
\end{equation*}
which implies that the limit in \eqref{eq:GateauDerivativeValue} exists and is given by 
\begin{equation*}
\lim_{\epsilon \rightarrow 0^+} \bigg[ \frac{\Vcal(\tau,(\Id + \epsilon \F)_{\#} \mu^*(\tau)) - \Vcal(\tau,\mu^*(\tau))}{\epsilon} \bigg] = \INTDom{\langle -\bar{\nu}^*(\tau,x) , \F(x) \rangle}{\R^d}{\mu^*(\tau)(x)}, 
\end{equation*}
for every $\F \in L^{\infty}(\R^d,\R^d;\mu^*(\tau))$. 
\end{proof}

\begin{rmk}[On the directional derivatives of the value function]
The identity written in \eqref{eq:GateauDerivativeValue} can be assimilated to the claim that $-\bar{\nu}^*(\tau) \in L^2(\R^d,\R^d;\mu^*(\tau))$ is the Gateaux derivative of the value function $\Vcal(\tau,\cdot)$ at $\mu^*(\tau)$. 
\end{rmk}

In the following theorem, we build on Proposition \ref{prop:UniqueSubdiff} to prove that subdifferentiability -- and therefore Gateaux differentiability -- properties of the value function propagate forward in time along solutions of the Hamiltonian dynamics stemming from the PMP. 

\begin{thm}[Forward propagation of Dini subdifferentials and regularity of the value function]
\label{thm:ForwardSubdiff}
Let $\mu^0 \in \Pcal(B(0,r))$ for some $r > 0$, and assume that hypotheses \ref{hyp:OCP} hold. Given a minimiser $(\mu^*(\cdot),u^*(\cdot))$ for $(\Ppazo)$, let $\nu^*(\cdot)$ the state-costate curve satisfying the PMP of Theorem \ref{thm:PMP}. Finally, suppose that $\eth^-_{\mu} \Vcal(\tau,\mu^*(\tau)) \neq \emptyset$ for some $\tau \in [0,T]$. 

Then, any curve of measures $\vartheta^*(\cdot)$ solution of the forward Hamiltonian continuity equation
\begin{equation}
\label{eq:ForwardHamilDynThm}
\left\{
\begin{aligned}
\partial_t \vartheta^*(t) & + \Div_{(x,r)} \Big( \J_{2d} \nabla_{\nu} \H(t,\vartheta^*(t),u^*(t)) \vartheta^*(t) \Big) = 0, \\
\pi^1_{\#} \vartheta^*(t) & = \, \mu^*(t) \hspace{2.35cm} \text{for all times $t \in [\tau,T]$}, \\
-\bar{\vartheta}^*(\tau) & \in \eth^-_{\mu} \Vcal(\tau,\mu^*(\tau)),
\end{aligned}
\right.
\end{equation}
is such that $\bar{\vartheta}^*(t) = \bar{\nu}^*(t)$ for all times $t \in [\tau,T]$. In particular, the state-costate curve $\nu^*(\cdot)$ is a solution of \eqref{eq:ForwardHamilDynThm}, which also satisfies 
\begin{equation*}
\eth^-_{\mu} \Vcal(t,\mu^*(t)) \cap  \eth^+_{\mu} \Vcal(t,\mu^*(t)) = \{ - \bar{\nu}^*(t) \}, 
\end{equation*}
for all  times $t \in [\tau,T]$. Moreover, the application $\mu \in \Pcal_c(\R^d) \mapsto \Vcal(t,\mu) \in \R$ admits \textnormal{directional derivatives} at $\mu^*(t)$, with
\begin{equation}
\label{eq:GateauDerivativeValueRunning}
\dn \Vcal(t,\mu^*(t))(0,\F) = \INTDom{\langle -\bar{\nu}^*(t,x) , \F(x) \rangle}{\R^d}{\mu^*(t)(x)}, 
\end{equation}
for all times $t \in [\tau,T]$ and every $\F \in L^{\infty}(\R^d,\R^d;\mu^*(t))$. 
\end{thm}

\begin{proof}
Let $K:= B(0,R_r)$ with $R_r > 0$ being defined as in Lemma \ref{lem:AdmEst}, and choose arbitrary elements $t \in [\tau,T]$ and $\F \in L^{\infty}(\R^d,\R^d;\mu^*(t))$. Then for every $\epsilon \in (0,1]$, one has
\begin{equation}
\label{eq:TaylorFlowSubdiff1}
\begin{aligned}
\Vcal \big(t,(\Id + \epsilon \F)_{\#} & \mu^*(t) \big) - \Vcal(t,\mu^*(t)) \\
& \geq \Vcal \Big( \tau , \Phi^{u^*}_{(t,\tau)} \big[ (\Id + \epsilon \F)_{\#} \mu^*(t) \big] \circ (\Id + \epsilon \F)_{\#} \mu^*(t) \Big) - \Vcal(\tau,\mu^*(\tau)),
\end{aligned}
\end{equation}
since $t \in [\tau,T] \mapsto \Vcal(t,\mu(t)) \in \R$ is non-decreasing along solutions of the control system and constant along optimal trajectories. Observing that $\supp((\Id + \epsilon \F)_{\#} \mu^*(t)) \subset \K := B(0,\Rpazo)$ for all $\epsilon \in (0,1]$, where we defined $\Rpazo := R_r \, + \NormL{\F(\cdot)}{\infty}{\mu^*(t)}$, it follows from the proof of Corollary \ref{cor:TotalLinFlows} applied with $h=0$, $y = x + \epsilon \F(x) \in \K$ and $\Bmu_t = (\Id,\Id + \epsilon \F)_{\#} \mu^*(t)$ that
\begin{equation}
\label{eq:TaylorFlowSubdiff2}
\begin{aligned}
\Phi^{u^*}_{(t,\tau)} \big[ (\Id + \epsilon \F)_{\#} \mu^*(t) \big] \Big( x + \epsilon \F(x)\Big) = \Phi^*_{(t,\tau)}(x) & + \epsilon \F_t(\tau,x) + o_{t,x,\K}(\epsilon), 
\end{aligned}
\end{equation}
with $\sup_{(t,x) \in [0,T] \times \supp(\mu^*(t))} |o_{t,x,\K}(\epsilon)| = o_{K}(\epsilon)$. Here, the map $\F_t : [\tau,t] \times \K \rightarrow \R^d$ is defined by 
\begin{equation}
\label{eq:FDef}
\F_t(s,x) := \D_x \Phi_{(t,s)}^*(x) \F(x) + w_{\Bmu_t}(s,x), 
\end{equation}
for all times $s \in [\tau,t]$ and every $x \in \K$, and it is the unique solution of the linearised Cauchy problem
\begin{equation}
\label{eq:LinearisedCauchySubdiff}
\left\{
\begin{aligned}
\partial_s \F_t(s,x) & = \D_x v \Big( s , \mu^*(s) , \Phi_{(t,s)}^*(x) \Big) \F_t(s,x) \\
& \hspace{0.4cm} + \INTDom{\D_{\mu} v \Big( s , \mu^*(s) , \Phi_{(t,s)}^*(x) \Big) \big( \Phi_{(t,s)}^*(y) \big) \F_t(s,y)}{\R^d}{\mu^*(t)(y)}, \\
\F_t(t,x) & = \F(x),
\end{aligned}
\right.
\end{equation}
as a consequence of Proposition \ref{prop:LinSpace_Flows} and Theorem \ref{thm:FlowDiff}. Hence, by merging \eqref{eq:TaylorFlowSubdiff1}, \eqref{eq:TaylorFlowSubdiff2} and \eqref{eq:FDef}, we obtain 
\begin{equation}
\label{eq:TaylorFlowSubdiff3}
\Vcal \big(t,(\Id + \epsilon \F)_{\#} \mu^*(t) \big) - \Vcal(t,\mu^*(t)) \geq \Vcal \Big( \tau , \big(\Id + \epsilon \F_t \big(\tau,\Phi_{(\tau,t)}^*(\cdot)\big) + o(\epsilon) \big)_{\raisebox{4pt}{$\scriptstyle{\#}$}} \mu^*(\tau) \Big) - \Vcal(\tau,\mu^*(\tau)).
\end{equation}

Recall now that since $\eth^-_{\mu} \Vcal(\tau,\mu^*(\tau)) \neq \emptyset$, one has $\eth^-_{\mu} \Vcal(\tau,\mu^*(\tau)) = \{ - \bar{\nu}^*(\tau)\}$ by Proposition \ref{prop:UniqueSubdiff}. Thus, by construction, the state-costate curve $\nu^*(\cdot)$ is a solution of the forward Hamiltonian dynamics \eqref{eq:ForwardHamilDynThm}. 

Next, we consider any solution $\vartheta^*(\cdot)$ of \eqref{eq:ForwardHamilDynThm}. Then, by definition \eqref{eq:LowerDiniDerivative} of the lower Dini derivative, one has 
\begin{equation}
\label{eq:TaylorFlowSubdiff4}
\begin{aligned}
\dn^- \Vcal(\tau,\mu^*(\tau) \Big( 0 , \F_t \big(\tau, \Phi_{(\tau,t)}^*(\cdot) \big) \Big) & \geq \INTDom{\big\langle - \bar{\vartheta}^*(\tau,x) , \F_t \big(\tau, \Phi_{(\tau,t)}^*(x) \big) \big\rangle}{\R^d}{\mu^*(\tau)(x)} \\ 
& = - \INTDom{\big\langle r , \F_t \big(\tau, \Phi_{(\tau,t)}^*(x) \big) \big\rangle}{\R^{2d}}{\vartheta^*(\tau)(x,r)}.
\end{aligned} 
\end{equation}
By repeating the arguments of Appendix \ref{section:AppendixLem} with \eqref{eq:LinearisedCauchySubdiff}, one can show that the map 
\begin{equation*}
\Hpazo : s \in [\tau,t] \mapsto \INTDom{\big\langle r , \F_t \big(s , \Phi_{(s,t)}^*(x) \big) \big\rangle}{\R^{2d}}{\vartheta^*(s)(x,r)}, 
\end{equation*}
is constant over $[\tau,t]$, which in particular yields
\begin{equation}
\label{eq:TaylorFlowSubdiff5}
\INTDom{\big\langle r , \F_t \big(\tau , \Phi_{(\tau,t)}^*(x) \big) \big\rangle}{\R^{2d}}{\nu^*(\tau)(x,r)} = \INTDom{\langle r , \F(x) \rangle}{\R^{2d}}{\vartheta^*(t)(x,r)}. 
\end{equation}
Combining \eqref{eq:TaylorFlowSubdiff3}, \eqref{eq:TaylorFlowSubdiff4} and \eqref{eq:TaylorFlowSubdiff5} allows us to conclude 
\begin{equation*}
\dn^- \Vcal(t,\mu^*(t))(0,\F) \geq \INTDom{\langle -\bar{\nu}^*(t,x) , \F(x) \rangle}{\R^{2d}}{\mu^*(t)(x)}, 
\end{equation*}
for every $\F \in L^{\infty}(\R^d,\R^d;\mu^*(t))$, where we used that $\mu \in \Pcal_1(K) \mapsto \Vcal(\tau,\mu) \in \R$ is Lipschitz continuous by Proposition \ref{prop:LipRegValue}. 

By Definition \ref{def:Value_GateauSubDiff}, this latter inequality implies that $-\bar{\vartheta}^*(t) \in \eth^-_{\mu} \Vcal(t,\mu^*(t))$ for all times $t \in [\tau,T]$, which also yields $\bar{\vartheta}^*(t) = \bar{\nu}^*(t)$ and \eqref{eq:GateauDerivativeValueRunning} as a consequence of Proposition \ref{prop:UniqueSubdiff}. Furthermore, we also get that $\eth^-_{\mu} \Vcal(t,\mu^*(t)) \cap \eth^+_{\mu} \Vcal(t,\mu^*(t))$ is reduced to a singleton whenever it is non-empty. This allows us to conclude that $\eth^-_{\mu} \Vcal(t,\mu^*(t)) \cap \eth^+_{\mu} \Vcal(t,\mu^*(t)) = \{ -\bar{\nu}^*(t) \}$ for all times $t \in [\tau,T]$.
\end{proof}

\begin{rmk}[Link between Gateaux derivatives and Wasserstein gradients]
If a functional $\phi : \Pcal_c(\R^d) \rightarrow \R$ is locally differentiable at some $\mu \in \Pcal_c(\R^d)$ in the sense of Definition \ref{def:Classical_Diff}, then $\eth^- \phi(\mu) \cap \eth^+ \phi(\mu) = \{ \nabla \phi(\mu)\}$ where $\nabla \phi(\mu) \in \Tan_{\mu} \Pcal_2(\R^d)$ is the Wasserstein gradient of $\phi(\cdot)$ at $\mu$.
\end{rmk}


\subsection{Sufficient optimality conditions under Dini-type sensitivity relations}
\label{subsection:SufficientOpt}

In Theorem \ref{thm:Sensitivity2}, we have shown that if $(\mu^*(\cdot),u^*(\cdot))$ is an optimal pair for $(\Ppazo)$, then the unique state-costate curve $\nu^*(\cdot)$ satisfying the PMP is such that $(\H(t,\nu^*(t),u^*(t)),-\bar{\nu}^*(t)) \in \eth^+ \Vcal(t,\mu^*(t))$ for $\Lcal^1$-almost every $t \in [0,T]$. In what follows, we prove that the sensitivity relations also provide \textit{sufficient conditions} for the optimality of a pair associated with a state-costate curve satisfying the PMP. 

Throughout this section, we assume that hypotheses \ref{hyp:OCP} hold, fix $\mu^0 \in \Pcal(B(0,r))$ for some $r > 0$ and let $K := B(0,R_r)$ with $R_r > 0$ being given by Lemma \ref{lem:AdmEst}.
 
\begin{thm}[Sufficiency of the PMP under Dini-type sensitivity relations]
\label{thm:SufficiencyPMP}
Let $(\mu^*(\cdot),u^*(\cdot))$ be an admissible pair for $(\Ppazo)$, and suppose that there exists a state-costate curve $\nu^*(\cdot)$ satisfying the PMP of Theorem \ref{thm:PMP}, together with the sensitivity relation
\begin{equation}
\label{eq:SufficiencySensitivity}
\Big( \H(t,\nu^*(t),u^*(t)) \, , \, - \bar{\nu}^*(t) \Big) \in \eth^+ \Vcal(t,\mu^*(t)), 
\end{equation}
for $\Lcal^1$-almost every $t \in [0,T]$. Then, the pair $(\mu^*(\cdot),u^*(\cdot))$ is optimal for $(\Ppazo)$.
\end{thm} 
 
\begin{proof}
Recall first that since $(\mu^*(\cdot),u^*(\cdot))$ is an admissible pair for $(\Ppazo)$, the map $t \in [0,T] \mapsto \Vcal(t,\mu^*(t)) \in \R$ is non-decreasing, as well as absolutely continuous by Proposition \ref{prop:LipRegValue}. Let $\T \subset (0,T)$ be as in Section \ref{subsection:FrechetSensitivity} and fix $\tau \in \T$ such that both $\derv{}{t} \Vcal(\tau,\mu^*(\tau))$ exists and \eqref{eq:SufficiencySensitivity} hold at $\tau$.

Remark that by \eqref{eq:Barycenter} and Definition \ref{def:Value_GateauSuperDiff}, the inclusion \eqref{eq:SufficiencySensitivity} can be rewritten as 
\begin{equation*}
\begin{aligned}
& \limsup_{\epsilon \rightarrow 0^+} \bigg[ \, \frac{\Vcal \big( \tau + \epsilon, (\Id + \epsilon \F)_{\#} \mu^*(\tau) \big) - \Vcal(\tau,\mu^*(\tau))}{\epsilon} \, \bigg] \\ 
& \hspace{6.5cm} \leq - \INTDom{\langle r , \F(x) \rangle}{\R^{2d}}{\nu^*(\tau)(x,r)} + \H(\tau,\nu^*(\tau),u^*(\tau)), 
\end{aligned}
\end{equation*}
for any $\F \in L^{\infty}(\R^d,\R^d;\mu^*(\tau))$. Choosing in particular $\F(\cdot) = v(\tau,\mu^*(\tau),u^*(\tau),\cdot)$, one further has
\begin{equation}
\label{eq:SufficiencyEst2}
\limsup_{\epsilon \rightarrow 0^+} \bigg[ \, \frac{\Vcal \big( \tau + \epsilon, (\Id + \epsilon v(\tau,\mu^*(\tau),u^*(\tau)))_{\#} \mu^*(\tau) \big) - \Vcal(\tau,\mu^*(\tau))}{\epsilon} \, \bigg] \leq 0. 
\end{equation}
Under hypotheses \ref{hyp:OCP}-$(i),(ii)$, it can be shown by repeating the arguments in the proof of Proposition \ref{prop:LinSpace_Flows} below that the non-local flows $(\Phi^*_{(\tau,t)}(\cdot))_{\tau,t \in [0,T]}$ defined along $(\mu^*(\cdot),u^*(\cdot))$ as in \eqref{eq:CondensedNota} satisfy the following Taylor expansion in $C^0(K,\R^d)$
\begin{equation}
\label{eq:Sufficiency_TaylorFlow}
\Phi_{(\tau,\tau+\epsilon)}^*(\cdot) = \Id + \epsilon v(\tau,\mu^*(\tau),u^*(\tau),\cdot) + o_{\tau}(\epsilon),
\end{equation}
for every $\epsilon > 0$, since the elements of $\T$ are Lebesgue points of $t \in [0,T] \mapsto v(t,\cdot,u^*(t),\cdot) \in C^0(\Pcal(K) \times K,\R^d)$. Thus, we can deduce from \eqref{eq:Sufficiency_TaylorFlow} combined with the representation formula \eqref{eq:FlowRep} that
\begin{equation*}
\mu^*(\tau+\epsilon) = \Big( \Id + \epsilon v(\tau,\mu^*(\tau),u^*(\tau)) + o_{\tau}(\epsilon) \Big)_{\raisebox{4pt}{$\scriptstyle{\#}$}} \mu^*(\tau), 
\end{equation*}
for any small $\epsilon > 0$, which together with \eqref{eq:SufficiencyEst2} and Proposition \ref{prop:LipRegValue} allows us to obtain
\begin{equation*}
\limsup_{\epsilon \rightarrow 0^+} \bigg[ \, \frac{\Vcal \big( \tau + \epsilon, \mu^*(\tau+\epsilon) \big) - \Vcal(\tau,\mu^*(\tau))}{\epsilon} \, \bigg] \leq 0. 
\end{equation*}
Since we assumed that $t \in [0,T] \mapsto \Vcal(t,\mu^*(t)) \in \R$ is differentiable at $\tau \in \T$, this last estimate finally yields
\begin{equation}
\label{eq:SufficiencyEst3}
\tderv{}{t}{} \Vcal(\tau,\mu^*(\tau)) \leq 0,
\end{equation}
for every $\tau \in \T$ belonging to a subset of full $\Lcal^1$-measure. Whence, the map $\Vcal(\cdot,\mu^*(\cdot))$ is non-increasing and therefore constant, which implies that the pair $(\mu^*(\cdot),u^*(\cdot))$ is optimal.
\end{proof} 
 
In Theorem \ref{thm:SufficiencyPMP} above, we have proven that the necessary conditions of the PMP become sufficient when supplemented with the Dini-type sensitivity relation \eqref{eq:SufficiencySensitivity}. In the following proposition, we show that the latter sensitivity relation is still a sufficient optimality condition when $\nu^*(\cdot)$ is not a necessarily a state-costate curve solution of the Hamiltonian continuity equation \eqref{eq:PMP_Dynamics}. 

\begin{prop}[A more general sufficient optimality condition]
Let $(\mu^*(\cdot),u^*(\cdot))$ be admissible for $(\Ppazo)$, and suppose that for $\Lcal^1$-almost every $t \in [0,T]$ there exists $\nu^*_t \in \Pcal_c(\R^{2d})$ such that 
\begin{equation*}
\pi^1_{\#} \nu^*_t = \mu^*(t), \quad \Big( \H(t,\nu^*_t,u^*(t)) , -\bar{\nu}^*_t \Big) \in \eth^+ \Vcal(t,\mu^*(t)) \quad \text{and} \quad \H(t,\nu^*_t,u^*(t)) = \max_{u \in U} \, \H(t,\nu^*_t,u),
\end{equation*}
where $\bar{\nu}^*_t$ is the barycentric projection of $\nu^*_t$ onto $\mu^*(t)$. Then, $(\mu^*(\cdot),u^*(\cdot))$ is optimal for $(\Ppazo)$. 
\end{prop}

\begin{proof}
This result can be obtained by repeating the arguments in the proof of Theorem \ref{thm:SufficiencyPMP}. 
\end{proof}


\subsection{Optimal feedback}
\label{subsection:Feedback}

We end this series of applications of the semiconcavity and sensitivity relations in mean-field control by discussing optimal feedbacks for $(\Ppazo)$ and their regularity. These latter take the form of a \textit{set-valued map} $\G : [0,T] \times \Pcal_c(\R^d) \rightrightarrows C^0(\R^d,\R^d)$, defined using lower Dini derivatives of the value function.

In order to investigate optimal feedbacks, we recall the notion of \textit{upper-semicontinuity} for set-valued maps, for which we refer to \cite[Section 1.4]{Aubin1990}. We also adapt to the Wasserstein setting a standard result of non-smooth analysis stating that semiconcave mappings admit directional derivatives, and that they coincide with adequately chosen \textit{regularised lower derivatives} (see e.g. \cite[Theorem 3.9]{Cannarsa1991}). 

\begin{Def}[Upper-semicontinuous set-valued mappings]
Let $(\Scal,d_{\Scal})$ be a separable metric space, $(X,\Norm{\cdot}_X)$ be a separable Banach space and $\G : \Scal \rightrightarrows X$ be a \textnormal{set-valued map}, i.e. an application such that $\G(s) \subset X$ for every $s \in \Scal$. Then, $\G(\cdot)$ is said to be \textnormal{upper-semicontinuous} at a point $s \in \Scal$ such that $\G(s) \neq \emptyset$ if for every neighbourhood $\U \subset X$ of $\G(s)$, there exists $\eta > 0$ such that $\G(s') \subset \U$ for every $s' \in \Scal$ such that $d_{\Scal}(s,s') \leq \eta$. 
\end{Def}

\begin{prop}[Dini derivatives of semiconcave mappings]
\label{prop:DiniDerSemiconcave}
Let $\phi : \Pcal_c(\R^d) \rightarrow \R^d$ be locally strongly semiconcave. Then for every $\mu \in \Pcal_c(\R^d)$, the directional derivative $\dn \phi(\mu)(\F)$ exists for any $\F \in L^{\infty}(\R^d,\R^d;\mu)$. Moreover if $\F \in C^0(\R^d,\R^d)$, then for every $R > 0$ the latter coincides with the \textnormal{regularised $R$-lower derivative}, defined by 
\begin{equation}
\label{eq:RegularisedLowerDer}
\dn^-_{o,R} \, \phi(\mu)(\F) := \liminf_{\substack{\epsilon \rightarrow 0^+ \hspace{-0.05cm}, \, \nu \rightarrow \mu \\ \nu \in \Pcal(B_R(\mu))}} \bigg[ \frac{\phi \big( (\Id + \epsilon \F)_{\#} \nu \big) - \phi(\nu)}{\epsilon} \bigg].
\end{equation}
\end{prop}

\begin{proof}
We start by showing that $\phi(\cdot)$ admits Dini derivatives in every direction. Let $\mu \in \Pcal_c(\R^d)$ and $\F \in L^{\infty}(\R^d,\R^d;\mu)$, and denote by $K := B(0,R)$ a closed ball such that $\supp((\Id + \epsilon \F)_{\#} \mu) \subset K$ for every $\epsilon \in [0,1]$. Observe that for every pair $0 < \epsilon_1 \leq \epsilon_2 \leq 1$, it holds 
\begin{equation*}
\Id + \epsilon_1 \F = \big( 1 -\tfrac{\epsilon_1}{\epsilon_2}  \big) \Id + \tfrac{\epsilon_1}{\epsilon_2} \big( \Id + \epsilon_2 \F \big),  
\end{equation*}
which, by the local strong semiconcavity of $\phi(\cdot)$ applied along $\Bmu := (\Id,\Id + \epsilon_2 \F)_{\#} \mu$, yields
\begin{equation}
\label{eq:DiniSemiconcavity}
\frac{\phi \big((\Id + \epsilon_1 \F)_{\#} \mu \big) - \phi(\mu)}{\epsilon_1} \geq \frac{\phi \big((\Id + \epsilon_2 \F)_{\#} \mu \big) - \phi(\mu)}{\epsilon_2}  - \Cpazo_K \big( \epsilon_2 - \epsilon_1 \big) \NormL{\F(\cdot)}{2}{\mu}
\end{equation}
for a given constant $\Cpazo_K > 0$, where we used the fact that $W_{2,\Bmu}(\mu,(\Id + \epsilon_2 \F)_{\#} \mu) = \epsilon_2 \NormL{\F(\cdot)}{2}{\mu}$. Taking first the liminf as $\epsilon_1 \rightarrow 0^+$ and then the limsup as $\epsilon_2 \rightarrow 0^+$ in \eqref{eq:DiniSemiconcavity}, we obtain 
\begin{equation*}
\dn^- \phi(\mu)(\F) \geq \dn^+ \phi(\mu)(\F), 
\end{equation*}
which implies the existence of the directional derivative $\dn \phi(\mu)(\F)$.

Let $R > 0$ and suppose that $\F \in C^0(\R^d,\R^d)$. By definition \eqref{eq:RegularisedLowerDer} of the regularised $R$-lower derivative, it is clear that $\dn \phi(\mu)(\F) \geq \dn^-_{o,R} \, \phi(\mu)(\F)$. As we assumed that $\phi(\cdot)$ is locally strongly semiconcave, it is also continuous in the $W_2$-metric over $\Pcal(B_R(\mu))$, and for every $(\epsilon,\alpha) \in (0,1] \times \R_+^*$ there exists $\eta > 0$ such that 
\begin{equation}
\label{eq:DiniContinuity}
\frac{\phi \big((\Id + \epsilon \F)_{\#} \mu \big) - \phi(\mu)}{\epsilon} \leq \frac{\phi \big((\Id + \epsilon \F)_{\#} \nu \big) - \phi(\nu)}{\epsilon} + \alpha, 
\end{equation}
for any $\nu \in \Pcal(B_R(\mu))$ satisfying $W_2(\mu,\nu) \leq \eta$. By plugging \eqref{eq:DiniContinuity} into \eqref{eq:DiniSemiconcavity}, it further holds
\begin{equation*}
\frac{\phi \big((\Id + \epsilon \F)_{\#} \mu \big) - \phi(\mu)}{\epsilon} \leq \inf_{\nu \in \K_R^{\eta}(\mu) , \, \epsilon' \in (0,\epsilon]} \frac{\phi \big((\Id + \epsilon' \F)_{\#} \nu \big) - \phi(\nu)}{\epsilon'} + \, \Cpazo_K \epsilon \NormL{\F(\cdot)}{2}{\mu}  + \alpha
\end{equation*}
where $\K_R^{\eta}(\mu) := \{ \nu \in \Pcal(B_R(\mu)) ~\text{s.t.}~ W_2(\mu,\nu) \leq \eta \}$. Letting $\epsilon,\alpha,\eta \rightarrow 0^+$ in the previous expression thus yields $\D \phi(\mu)(\F) \leq \dn^-_{o,R} \, \phi(\mu)(\F)$, which concludes the proof since $R > 0$ is arbitrary.
\end{proof}

Building on these notions, we discuss in the following theorem the structure and regularity of optimal feedback mappings for $(\Ppazo)$. In the proof of this result, we will denote by $\Graph(\G(\cdot)) := \{ (s,x) \in \Scal \times X ~\text{s.t.}~ x \in \G(s) \}$ the \textit{graph} of a given set-valued mapping $\G : \Scal \rightrightarrows X$.

\begin{thm}[Optimal feedbacks]
\label{thm:Feedback}
Let $\mu^0 \in \Pcal(B(0,r))$ for some $r > 0$. Moreover, assume that hypotheses \ref{hyp:OCP} hold, and let $K := B(0,R_r)$ where $R_r > 0$ is given as in Lemma \ref{lem:AdmEst}. 

Then, a trajectory-control pair $(\mu^*(\cdot),u^*(\cdot))$ is optimal for $(\Ppazo)$ if and only if 
\begin{equation}
\label{eq:FeedbackInc}
v(t,\mu^*(t),u^*(t))_{|K} \in \G_{K}(t,\mu^*(t)),
\end{equation}
for $\Lcal^1$-almost every $t \in [0,T]$, where we defined the \textnormal{generalised feedback sets}
\begin{equation}
\label{eq:FeedbackSet}
\G_{K}(t,\mu):= \Big\{ \vb \in v(t,\mu,U)_{|K} ~\text{s.t.}~ \dn^- \Vcal(t,\mu)(1,\vb) \leq 0 \Big\},
\end{equation}
for all $(t,\mu) \in [0,T] \times \Pcal(K)$. The feedback sets $\G_K(t,\mu)$ are compact in $C^0(K,\R^d)$ for all $(t,\mu) \in [0,T] \times \Pcal(K)$, and if in addition \ref{hyp:SC2} and \ref{hyp:R} hold, then the set-valued map $\G_K : [0,T] \times \Pcal_1(K) \rightrightarrows C^0(K,\R^d)$ is upper-semicontinuous.
\end{thm}

\begin{proof}
The fact that $(\mu^*(\cdot),u^*(\cdot))$ is optimal if and only if \eqref{eq:FeedbackInc} is satisfied can be proven by repeating arguments analogous to those of the proof of Theorem \ref{thm:SufficiencyPMP}. Indeed let $(\mu^*(\cdot),u^*(\cdot))$ be an optimal pair for $(\Ppazo)$ and observe that by Proposition \ref{prop:LipRegValue}, one has
\begin{equation*}
\begin{aligned}
\Vcal  \big( t+\epsilon,(\Id + \epsilon v(t,\mu^*(t),u^*(t))_{\#} \mu^*(t) \big) - \Vcal(t,\mu^*(t)) & =  \Vcal  \big( t+\epsilon, \mu^*(t+\epsilon) \big) - \Vcal(t,\mu^*(t)) + o_t(\epsilon) \\
& = o_t(\epsilon),
\end{aligned}
\end{equation*}
for $\Lcal^1$-almost every $t \in [0,T]$, where we used the fact that $t \in [0,T] \mapsto \Vcal(t,\mu^*(t)) \in \R$ is constant. Thus, dividing by $\epsilon > 0$ and taking the liminf as $\epsilon \rightarrow 0^+$ in the previous inequality, we obtain 
\begin{equation*}
\dn^- \Vcal(t,\mu^*(t)) \Big( 1,v(t,\mu^*(t),u^*(t)) \Big) = 0,
\end{equation*}
which by \eqref{eq:FeedbackSet} implies that $v(t,\mu^*(t),u^*(t))_{|K} \in \G_{K}(t,\mu^*(t))$ for $\Lcal^1$-almost every $t \in [0,T]$. Conversely, let $(\mu^*(\cdot),u^*(\cdot))$ be an admissible pair for $(\Ppazo)$ such that $v(t,\mu^*(t),u^*(t))_{|K} \in \G_K(t,\mu^*(t))$ for $\Lcal^1$-almost every $t \in [0,T]$, and let $\tau \in [0,T]$ be a point of differentiability of $t \in [0,T]  \mapsto \Vcal(t,\mu^*(t)) \in \R$ at which \eqref{eq:FeedbackInc} holds and $W_1(\mu^*(\tau+\epsilon),(\Id+\epsilon v(\tau,\mu^*(\tau),u^*(\tau)))_{\#} \mu^*(\tau)) = o_{\tau}(\epsilon)$. Then, 
\begin{equation*}
\Vcal  \big( \tau +\epsilon, \mu^*(\tau+\epsilon) \big) - \Vcal(\tau,\mu^*(\tau)) = \Vcal  \big( \tau +\epsilon,(\Id + \epsilon v(\tau,\mu^*(\tau),u^*(\tau)))_{\#} \mu^*(\tau) \big) - \Vcal(\tau,\mu^*(\tau)) + o_{\tau}(\epsilon),
\end{equation*}
and upon dividing by $\epsilon > 0$ and taking the liminf as $\epsilon \rightarrow 0^+$ while recalling that $\Vcal(\cdot,\mu^*(\cdot))$ is differentiable at $\tau$ and $v(\tau,\mu^*(\tau),u^*(\tau))_{|K} \in \G_K(\tau,\mu^*(\tau))$, we obtain 
\begin{equation*}
\tderv{}{t} \Vcal(\tau,\mu^*(\tau))  \leq 0.
\end{equation*}
Because $t \in [0,T] \mapsto \Vcal(t,\mu^*(t)) \in \R$ is absolutely continuous and non-decreasing, and since the last inequality holds for $\Lcal^1$-almost every $\tau \in [0,T]$, we conclude that $(\mu^*(\cdot),u^*(\cdot))$ is optimal for $(\Ppazo)$.

We observe next that for every $(t,\mu ) \in [0,T] \times \Pcal(K)$, the sets
\begin{equation*}
\Big \{ \vb \in C^0(K,\R^d) ~\text{s.t.}~ \dn^- \Vcal(t,\mu )(1,\vb) \leq 0
\Big \},
\end{equation*}
are closed in $C^0(K,\R^d)$ by Proposition \ref{prop:LipRegValue}, which together with the compactness of $U$ implies that the sets $\G_K(t,\mu)$ are compact. We now prove that under hypotheses \ref{hyp:SC2} and \ref{hyp:R}, the set-valued map $\G_{K} : [0,T] \times \Pcal(K) \rightrightarrows C^0(K,\R^d)$ has closed graph and takes its values in a compact set. By classical results in set-valued analysis (see e.g. \cite[Corollary 1]{Aubin1984}), this will imply that it is upper-semicontinuous. Remark first that as a consequence of hypotheses \ref{hyp:R}, the elements of $v(t,\mu,U)_{|K} \subset C^0(K,\R^d)$ are equibounded and equi-Lipschitz continuous, uniformly with respect $(t,\mu) \in [0,T] \times \Pcal(K)$. Thus, the sets $\G_K(t,\mu)$ are contained within a compact subset of $C^0(K,\R^d)$ that is independent of $(t,\mu) \in [0,T] \times \Pcal(K)$. Assume now that \ref{hyp:SC2} hold, and consider a sequence $(t_n,\mu_n,\vb_n) \subset \Graph(\G_K(\cdot,\cdot))$ satisfying
\begin{equation*}
t_n ~\underset{n \rightarrow +\infty}{\longrightarrow}~t, \qquad W_1(\mu_n,\mu) ~\underset{n \rightarrow +\infty}{\longrightarrow}~ 0 \qquad \text{and} \qquad \NormC{\vb - \vb_n}{0}{K,\R^d} ~\underset{n \rightarrow +\infty}{\longrightarrow}~ 0, 
\end{equation*}
for some $(t,\mu,\vb) \in [0,T] \times \Pcal(K) \times C^0(K,\R^d)$. By Proposition \ref{prop:DiniDerSemiconcave}, we know that for every $R > 0$ the lower Dini derivative $\dn^- \Vcal(t,\mu)(1,\vb)$ coincides with the regularised $R$-lower derivative $\dn^-_{o,R} \Vcal(t,\mu)(1,\vb)$. By definition \eqref{eq:RegularisedLowerDer} of the latter, there exist two sequences $\epsilon_n,\delta_n \rightarrow 0^+$ such that 
\begin{equation*}
\frac{\Vcal \big(t + \epsilon_n,(\Id + \epsilon_n \vb_n)_{\#} \mu_n \big) - \Vcal(t,\mu_n)}{\epsilon_n} \leq \delta_n, 
\end{equation*}
for $n \geq 1$ large enough, which allows to obtain 
\begin{equation*}
\dn^- \Vcal(t,\mu^*(t))(1,\vb) = \dn^-_{o,R} \Vcal(t,\mu^*(t))(1,\vb) \leq \liminf_{n \rightarrow +\infty} \bigg[ \frac{\Vcal \big(t + \epsilon_n,(\Id + \epsilon_n \vb_n)_{\#} \mu \big) - \Vcal(t,\mu_n)}{\epsilon_n} \bigg] \leq 0, 
\end{equation*}
for $\Lcal^1$-almost every $t \in [0,T]$. One can also check that $(t,\mu) \in [0,T] \times \Pcal_1(K) \rightrightarrows v(t,\mu,U)_{|K}$ has closed graph under \ref{hyp:R}-$(i)$, which allows us to conclude that $\Graph(\G_{K}(\cdot,\cdot))$ is closed.
\end{proof}

\begin{rmk}[On the feedback mapping defined in \eqref{eq:FeedbackSet}]
It has been identified in set-valued analysis (see e.g. \cite[Chapter 2]{Aubin1984}) that the minimal regularity requirement needed to generalise Peano's existence theorem to differential inclusions is the upper-semicontinuity of the right-hand side with respect to the state variable, along with the convexity of the set of admissible velocities. In the present context, the existence of global solutions to the feedback-form continuity inclusion
\begin{equation*}
\left\{
\begin{aligned}
& \partial_t \mu^*(t) \in - \Div_x \big( \G(t,\mu^*(t)) \mu^*(t) \big), \\
& \mu^*(0) = \mu^0, 
\end{aligned}
\right.
\end{equation*}
is ensured in the absence of convexity by the existence of an optimal trajectory $\mu^*(\cdot)$ for $(\Ppazo)$.
\end{rmk}



\setcounter{section}{0} 
\renewcommand{\thesection}{A} 

\section{Linearisation formulas for non-local flows}
\label{section:AppendixFlowDiff}

\setcounter{Def}{0} \renewcommand{\thethm}{A.\arabic{Def}} 
\setcounter{equation}{0} \renewcommand{\theequation}{A.\arabic{equation}}

In this section, we prove several first-order linearisation formulas for the non-local flows $(\Phi_{(\tau,t)}[\mu_\tau](\cdot))_{t \in [0,T]}$ introduced in Definition \ref{def:NonlocalFlows}. Throughout this section, we consider a velocity field $v : [0,T] \times \Pcal_c(\R^d) \times \R^d \mapsto \R^d$ satisfying hypotheses \ref{hyp:CE} of Section \ref{subsection:NonLocalCE}, along with a closed ball $K := B(0,r)$ where $r > 0$. 

Before stating the first-order linearisation results, we start by absolute continuity and Lipschitz regularity estimates on non-local flows that will be useful throughout the remainder of this section. 

\begin{lem}[Regularity estimates on non-local flows]
\label{lem:LipFlow}
There exists a constant $C_K > 0$ such that for every $\tau_1,\tau_2 \in [0,T]$ with $\tau_1 \leq \tau_2$, any $\mu,\nu \in \Pcal(K)$ and all $x,y \in K$, it holds
\begin{equation*}
\big| \Phi_{(\tau_2,t)}[\nu](y) - \Phi_{(\tau_1,t)}[\mu](x) \big| \leq C_K \bigg( |x-y| + W_1(\mu,\nu) + \INTSeg{m(s)}{s}{\tau_1}{\tau_2} \bigg), 
\end{equation*}
for all times $t \in [0,T]$.  
\end{lem}

\begin{proof}
The proof of this result essentially relies on a bootstrapped application of Gronwall's lemma. To simplify the notations, we introduce 
\begin{equation*}
\mu(t) := \Phi_{(\tau_1,t)}[\mu](\cdot)_{\#} \mu \qquad \text{and} \qquad \nu(t) := \Phi_{(\tau_2,t)}[\nu](\cdot)_{\#} \nu, 
\end{equation*}
defined for all times $t \in [0,T]$. We consider only the case $\tau_2 \leq t \leq T$, the other scenario being identical up to changing the sign of relevant integrals.

Observe first that by Theorem \ref{thm:NonLocalCE}, there exists $R_r \geq r > 0$ such that $\supp(\mu(t)) \cup \supp(\nu(t)) \subset K'$ where $K':= B(0,R_r)$ for all times $t \in [0,T]$. By \eqref{eq:NonLocalFlow_Def}, it then holds for any $x,y \in K$ 
\begin{equation}
\label{eq:LipFlowEst1}
\begin{aligned}
\big| \Phi_{(\tau_2,t)}[\nu](y) - \Phi_{(\tau_1,t)}[\mu](x) \big| & \leq |x-y| + \INTSeg{ \big| v \big(s,\mu(s),\Phi_{(\tau_1,s)}[\mu](x) \big) \big|}{s}{\tau_1}{\tau_2} \\
& \hspace{0.4cm} + \INTSeg{\Big| v \big( s,\nu(s),\Phi_{(\tau_2,s)}[\nu](y) \big) - v \big(s,\mu(s),\Phi_{(\tau_1,s)}[\mu](x) \big) \Big|}{s}{\tau_2}{t} \\
& \leq |x-y| + \INTSeg{m_r(s)}{s}{\tau_1}{\tau_2} + \INTSeg{L_{K'}(s) W_1(\mu(s),\nu(s))}{s}{\tau_2}{t} \\
& \hspace{4.45cm} + \INTSeg{l_{K'}(s) \big| \Phi_{(\tau_2,s)}[\nu](y) - \Phi_{(\tau_1,s)}[\mu](x) \big|}{s}{\tau_2}{t}, 
\end{aligned}
\end{equation}
for all times $t \in [\tau_1,T]$, with $m_r(\cdot) := (1+2 R_r) m(\cdot)$ and where we used \ref{hyp:CE}-$(i),(ii)$. Integrating \eqref{eq:LipFlowEst1} against a $1$-optimal transport plan $\gamma \in \Gamma_o(\mu,\nu)$ and applying Fubini's theorem, we further obtain
\begin{equation}
\label{eq:LipFlowEst2}
\begin{aligned}
\INTDom{\big| \Phi_{(\tau_2,t)}[\nu](y) - \Phi_{(\tau_1,t)}[\mu](x) \big|}{\R^{2d}}{\gamma(x,y)} & \leq W_1(\mu,\nu) + \INTSeg{m_r(s)}{s}{\tau_1}{\tau_2} + \INTSeg{L_{K'}(s) W_1(\mu(s),\nu(s))}{s}{\tau_2}{t} \\
& \hspace{0.4cm} + \INTSeg{l_{K'}(s) \Big( \INTDom{\big| \Phi_{(\tau_2,s)}[\nu](y) - \Phi_{(\tau_1,s)}[\mu](x) \big|}{\R^{2d}}{\gamma(x,y)} \Big)}{s}{\tau_2}{t}, 
\end{aligned}
\end{equation}
which yields by an application of Gr\"onwall's lemma that 
\begin{equation}
\label{eq:LipFlowEst3}
\begin{aligned}
& \INTDom{\big| \Phi_{(\tau_2,t)}[\nu](y) - \Phi_{(\tau_1,t)}[\mu](x) \big|}{\R^{2d}}{\gamma(x,y)} \\
& \hspace{2cm} \leq \bigg( W_1(\mu,\nu) + \INTSeg{m_r(s)}{s}{\tau_1}{\tau_2} + \INTSeg{L_{K'}(s) W_1(\mu(s),\nu(s))}{s}{\tau_2}{t} \bigg) \exp \big( \Norm{l_{K'}(\cdot)}_1 \big),
\end{aligned}
\end{equation}
for all times $t \in [\tau_2,T]$. Observing now that $(\Phi_{(\tau_1,t)}[\mu] \circ \pi^1 , \Phi_{(\tau_2,t)}[\nu] \circ \pi^2)_{\#} \gamma \in \Gamma(\mu(t),\nu(t))$ for all times $t \in [0,T]$, we can deduce from \eqref{eq:LipFlowEst3} that
\begin{equation*}
W_1(\mu(t),\nu(t)) \leq \bigg( W_1(\mu,\nu) + \INTSeg{m_r(s)}{s}{\tau_1}{\tau_2} + \INTSeg{L_{K'}(s) W_1(\mu(s),\nu(s))}{s}{\tau_2}{t} \bigg) \exp \big( \Norm{l_{K'}(\cdot)}_1 \big),
\end{equation*}
which by another application of Gr\"onwall's lemma provides the distance estimate
\begin{equation}
\label{eq:LipFlowEst4}
W_1(\mu(t),\nu(t)) \leq C_K' \bigg( W_1(\mu,\nu) + \INTSeg{m_r(s)}{s}{\tau_1}{\tau_2} \bigg), 
\end{equation}
for all times $t \in [\tau_2,T]$, where $C_K' > 0$ is a constant which only depends on $K := B(0,r)$ via hypotheses \ref{hyp:CE}. Plugging \eqref{eq:LipFlowEst4} into \eqref{eq:LipFlowEst1} and applying yet again Gr\"onwall's lemma, we can finally conclude 
\begin{equation*}
\big| \Phi_{(\tau_2,t)}[\nu](y) - \Phi_{(\tau_1,t)}[\mu](x) \big| \leq C_K \bigg( |x-y| + W_1(\mu,\nu) + \INTSeg{m(s)}{s}{\tau_1}{\tau_2} \bigg), 
\end{equation*}
for all times $t \in [\tau_2,T]$ and some constant $C_K > 0$. The case $0 \leq t \leq \tau_2 \leq T$ being similar, this concludes the proof of our lemma. 
\end{proof}

In what follows, we introduce notions of measurability, integrability and Lebesgue points adapted to maps with values in $C^0(\Pcal_c(\R^d) \times \R^d,\R^d)$. 

\begin{Def}[Measurability and integrability of $C^0$-valued maps]
\label{def:IntegralC0}
A mapping $V : [0,T] \rightarrow C^0(\Pcal_c(\R^d) \times \R^d,\R^d)$ is said to be \textnormal{$\Lcal^1$-measurable} if for every compact set $K \subset \R^d$, its restriction
\begin{equation*}
V_{|K} : t \in [0,T] \mapsto V(t)_{|\Pcal(K) \times K} \in C^0(\Pcal(K) \times K,\R^d),
\end{equation*}
is $\Lcal^1$-measurable. Similarly, we say that $V : [0,T] \rightarrow C^0(\R^d,\R^d)$ is \textnormal{locally integrable} if its restrictions $V_{K} : [0,T] \rightarrow C^0(\Pcal(K) \times K,\R^d)$ are integrable in the sense of Bochner for every compact set $K \subset \R^d$. 
\end{Def}

Using this notion of integrability, we can derive the following useful result concerning the Lebesgue points of non-local velocity fields satisfying hypotheses \ref{hyp:CE}.  

\begin{lem}[Uniform Lebesgue points of non-local velocity fields]
\label{lem:LebesguePoints}
There exists a subset $\T^* \subset (0,T)$ of full $\Lcal^1$-measure such that every compact set $K \subset \R^d$, the elements of $\T^*$ are Lebesgue point of the maps $v_{|K} : t \in [0,T] \rightarrow C^0(\Pcal(K) \times K,\R^d)$, $\D_x v_{|K} : t \in [0,T] \rightarrow C^0(\Pcal(K) \times K,\R^{d \times d})$ and $\D_{\mu} v_{|K} : t \in [0,T] \rightarrow C^0(\Pcal(K) \times K \times K,\R^{d \times d})$. 
\end{lem}

\begin{proof}
First, observe that $\Pcal_c(\R^d) \times \R^d = \cup_{n \geq 1} \Pcal(B(0,n)) \times B(0,n)$, where the sets $\Pcal(B(0,n))$ are compact in the $W_1$-topology for any $n \geq 1$ as a consequence of \cite[Theorem 7.1.5]{AGS}. Hence, each of the sets $\Pcal(B(0,n)) \times B(0,n)$ is compact in the product $W_1 \times |\cdot|$-topology. Moreover, by \ref{hyp:CE}-$(i)$, it holds 
\begin{equation*}
|v(t,\mu,x)| \leq m(t) \Big(1 + |x| + \max_{y \in \supp(\mu)} |y| \Big), 
\end{equation*} 
for $\Lcal^1$-almost every $t \in [0,T]$ and any $(\mu,x) \in  \Pcal_c(\R^d) \times \R^d$. In particular, the restricted maps 
\begin{equation*}
v_{|B(0,n)} : t \in [0,T] \mapsto v(t,\cdot,\cdot)_{|\Pcal(B(0,n)) \times B(0,n)} \in C^0(\Pcal(B(0,n)) \times B(0,n),\R^d), 
\end{equation*} 
are $\Lcal^1$-measurable and Bochner integrable for every $n \geq 1$. Whence, there exists a subset $\T_v^n \subset (0,T)$ of full $\Lcal^1$-measure whose elements are Lebesgue points of $t \in [0,T] \mapsto v_{|B(0,n)}(t,\cdot,\cdot)$ in the sense of \eqref{eq:LebesguePoint_Banach}. Consider the full $\Lcal^1$-measure subset $\T_v := \cap_{n \geq 1} \T_v^n \subset (0,T)$ and an arbitrary compact set $K \subset \R^d$. Then, notice that for every $\tau \in \T_v$ and $h > 0$ with $\tau+h \in [0,T]$, it holds
\begin{equation*}
\begin{aligned}
& \frac{1}{h} \INTSeg{\NormC{v_{|K}(t) - v_{|K}(\tau)}{0}{\Pcal(K) \times K,\R^d}}{t}{\tau}{\tau+h} \\
& \hspace{2.25cm} \leq \frac{1}{h} \INTSeg{\NormC{v_{|B(0,N_K)}(t) - v_{|B(0,N_K)}(\tau)}{0}{\Pcal(B(0,N_K)) \times B(0,N_K),\R^d}}{t}{\tau}{\tau+h} ~\underset{h \rightarrow 0^+}{\longrightarrow}~ 0,
\end{aligned}
\end{equation*}
for every $N_K \geq 1$ such that $K \subset B(0,N_K)$. 

Thus, we have proven that the elements of $\T_v$ are Lebesgue points of $t \in [0,T] \mapsto v_{|K}(t) \in C^0(\Pcal(K) \times K,\R^d)$ for any compact set $K \subset \R^d$. Applying the same reasoning while using hypotheses \ref{hyp:CE}-$(ii),(iii),(iv)$, one can obtain the existence of two subsets $\T_{v_x} ,\T_{v_{\mu}} \subset (0,T)$ of full $\Lcal^1$-measures whose elements respectively are Lebesgues points of 
\begin{equation*}
\D_x v_{|K} : t \in [0,T] \mapsto \D_x v(t,\cdot,\cdot)_{| \Pcal(K) \times K} \in C^0(\Pcal(K) \times K,\R^{d \times d}), 
\end{equation*}
and
\begin{equation*}
\D_{\mu} v_{|K} : t \in [0,T] \mapsto \D_{\mu} v(t,\cdot,\cdot)(\cdot)_{| \Pcal(K) \times K \times K} \in C^0(\Pcal(K) \times K \times K,\R^{d \times d}), 
\end{equation*}
for every compact set $K \subset \R^d$. Here, the measurability of $\D_{\mu} v_{|K}(\cdot)$ can be deduced from the construction detailed e.g. in \cite[Definition 3.1 and Remark 3.2]{Gangbo2019} together with standard arguments relying on Pettis' theorem for mappings with values in separable Banach spaces \cite[Chapter II - Theorem 2]{DiestelUhl}. Thus, defining the subset of full $\Lcal^1$-measure $\T^* := \T_v \cap \T_{v_x} \cap \T_{v_{\mu}} \subset (0,T)$ concludes the proof.
\end{proof}

In the following proposition, we extend a well-known result about linearisations of flows with respect to the space variable. 

\begin{prop}[Space derivatives of non-local flows]
\label{prop:LinSpace_Flows}
For every $\mu \in \Pcal(K)$ and all $\tau,t \in [0,T]$, the map $x \in K \mapsto \Phi_{(\tau,t)}[\mu](x) \in \R^d$ is continuously Fr\'echet-differentiable. Moreover, for any $x,y \in K$, the following Taylor expansion holds
\begin{equation}
\label{eq:TaylorSpace}
\Phi_{(\tau,t)}[\mu](y) = \Phi_{(\tau,t)}[\mu](x) + \D_x \Phi_{(\tau,t)}[\mu](x) (y-x) + o_{\tau,t,\mu,x,K}(|x-y|), 
\end{equation}
where
\begin{equation*}
\sup_{(\tau,t,\mu,x) \in [0,T] \times [0,T] \times \Pcal(K) \times K} o_{\tau,t,\mu,x,K}(|x-y|) = o_K(|x-y|).
\end{equation*}
Here, for all $x \in K$, the map $t \in [0,T] \mapsto \D_x \Phi_{(\tau,t)}[\mu](x) \in \R^{d \times d}$ is the unique solution of the linearised Cauchy problem 
\begin{equation}
\label{eq:LinearisedSpace_Cauchy}
\left\{
\begin{aligned}
\partial_t w(t,x) & = \D_x v \Big( t , \mu(t) , \Phi_{(\tau,t)}[\mu](x) \Big) w(t,x),  \\
w(\tau,x) & = \Id,
\end{aligned}
\right.
\end{equation}
with $\mu(t) := \Phi_{(\tau,t)}[\mu](\cdot)_{\#} \mu$ for all times $t \in [0,T]$. 
\end{prop}

\begin{proof}
Observe first that by Lemma \ref{lem:LipFlow}, there exists a constant $C_K > 0$ such that $x \in K \mapsto \Phi_{(\tau,t)}[\mu](x) \in \R^d$ is $C_K$-Lipschitz for any $\tau,t \in [0,T]$. By the classical variational equation (see e.g. \cite[Theorem 2.3.2]{BressanPiccoli}), the map $x \in \R^d \rightarrow \Phi_{(\tau,t)}[\mu](x) \in \R^d$ is Fr\'echet-differentiable, and its differential $w(t,x) := \D_x \Phi_{(\tau,t)}[\mu](x) \in \R^{d \times d}$ is the unique solution of the linearised Cauchy problem
\begin{equation}
\label{eq:VariationalEquation}
\left\{
\begin{aligned}
\partial_t w(t,x) & = \D_x v \Big( t, \mu(t), \Phi_{(\tau,t)}[\mu](x) \Big) w(t,x), \\
w(\tau,x) & = \Id.
\end{aligned}
\right.
\end{equation}
Observe now that by \ref{hyp:CE}-$(ii)$ and $(iii)$, the map $y \in \R^d \mapsto \D_x v(t,\mu(t),y) \in \R^d$ is uniformly continuous on compact sets, and there exists $l_K'(\cdot) \in L^1([0,T],\R_+)$ such that 
\begin{equation}
\label{eq:BoundVelocityDerivative}
\max_{x \in K} \Big | \D_x v \Big( t,\mu(t),\Phi_{(\tau,t)}[\mu](x) \Big) \Big| \leq l_K'(t),
\end{equation}
for $\Lcal^1$-almost every $t \in [0,T]$. These facts together with the $C_K$-Lipschitz regularity of $x \in K \mapsto \Phi_{(\tau,t)}[\mu](x) \in \R^d$ imply, by Gr\"onwall's Lemma and Lebesgue's dominated convergence theorem, that $x \in K \mapsto \D_x \Phi_{(\tau,t)}[\mu](x) \in \R^d$ is continuous and bounded, uniformly with respect to $\tau,t \in [0,T]$.

Fix now $x,y \in K$ and observe that by \eqref{eq:NonLocalFlow_Def} together with \eqref{eq:VariationalEquation}, it holds for all times $0 \leq \tau \leq t \leq T$
\begin{equation}
\label{eq:TaylorIntegralIneq}
\begin{aligned}
& \Big| \Phi_{(\tau,t)}[\mu](y) - \Phi_{(\tau,t)}[\mu](x) - \D_x \Phi_{(\tau,t)}[\mu](x)(y-x) \Big|  \\
& \leq \int_{\tau}^t \Big| v \Big(s,\mu(s),\Phi_{(\tau,s)}[\mu](y)\Big) - v \Big(s,\mu(s),\Phi_{(\tau,s)}[\mu](x)\Big) \\
& \hspace{5.35cm} - \D_x v \Big(s,\mu(s),\Phi_{(\tau,s)}[\mu](x)\Big) \D_x \Phi_{(\tau,s)}[\mu](x)(y-x) \Big| \textnormal{d}s \\
& \leq \int_0^1  \int_{\tau}^t \Big| \D_x v \Big(s,\mu(s),\Phi_{(\tau,s)}[\mu] \big( x + \lambda(y-x) \big)\Big) \D_x \Phi_{(\tau,s)}[\mu] \big( x+\lambda(y-x) \big)\\
& \hspace{5.35cm} - \D_x v \Big(s,\mu(s),\Phi_{(\tau,s)}[\mu](x) \Big) \D_x \Phi_{(\tau,s)}[\mu](x) \Big| |x-y| \textnormal{d}s \textnormal{d} \lambda
\end{aligned}
\end{equation}
where we used the integral version of Taylor's theorem along with Fubini's theorem. Notice now that since $z \in K \mapsto \D_x v \big(s,\mu(s),\Phi_{(\tau,s)}[\mu](z) \big) \in \R^d$ and $z \in K \mapsto \D_x \Phi_{(\tau,s)}[\mu](z) \in \R^d$ are continuous, it holds for every $\lambda \in [0,1]$ 
\begin{equation*}
\begin{aligned}
\Big| \D_x v \Big(s,\mu(s),\Phi_{(\tau,s)}[\mu] & \big( x + \lambda(y-x) \big)\Big) \D_x \Phi_{(\tau,s)}[\mu] \big( x+\lambda(y-x) \big)  \\
& - \D_x v \Big(s,\mu(s),\Phi_{(\tau,s)}[\mu](x) \Big) \D_x \Phi_{(\tau,s)}[\mu](x) \Big| |x-y| \leq o_{s,\mu,x,K}(|x-y|),
\end{aligned}
\end{equation*}
for $\Lcal^1$-almost every $s \in [\tau,T]$. Moreover, we have $\INTSeg{\sup_{(\mu,x) \in \Pcal(K) \times K} |o_{s,\mu,x,K}(|x-y|)|}{s}{0}{T} = o_K(|x-y|)$ as a consequence of the $C_K$-Lipschitz regularity of $\Phi_{(\tau,s)}[\mu](\cdot)$ over $K$, \eqref{eq:VariationalEquation} and \eqref{eq:BoundVelocityDerivative}. Plugging this estimate into \eqref{eq:TaylorIntegralIneq} and applying Lebesgue's dominated convergence theorem, we can conclude that 
\begin{equation*}
\Phi_{(\tau,t)}[\mu](y) = \Phi_{(\tau,t)}[\mu](x) + \D_x \Phi_{(\tau,t)}[\mu](x)(y-x) + o_{\tau,t,\mu,x,K}(|x-y|),
\end{equation*}
for all times $0 \leq \tau \leq t \leq T$ and any $x,y \in K$, where 
\begin{equation*}
\sup_{(\tau,t,\mu,x) \in [0,T] \times [0,T] \times \Pcal(K) \times K} \big| o_{\tau,t,\mu,x,K}(|x-y|) \big| = o_K(|x-y|).
\end{equation*}
The case $0 \leq t \leq \tau \leq T$ being similar, this ends the proof of our proposition.
\end{proof}

In \cite[Proposition 5]{PMPWass}, an explicit formula was derived for directional derivatives of non-local flows $(\Phi_{(\tau,t)}[\mu](\cdot))_{t \in [0,T]}$ with respect to the measure variable. Therein however, only the particular case of perturbations of the identity induced by vector fields was considered, whereas in the present paper we need the following generalisation which takes into account perturbations induced by arbitrary transport plans.

\begin{thm}[Measure derivatives of non-local flows along transport plans]
\label{thm:FlowDiff}
For all $\mu,\nu \in \Pcal(K)$ and every $\Bmu \in \Gamma(\mu,\nu)$, the map $\mu \in \Pcal(K) \mapsto \Phi_{(\tau,\cdot)}[\mu](\cdot) \in C^0([0,T] \times K,\R^d)$ admits a derivative in the direction $\Bmu$ at $\mu$ for all $\tau \in [0,T]$. Moreover, the following Taylor expansion holds in $C^0(K,\R^d)$  
\begin{equation}
\label{eq:TaylorMeasure}
\Phi_{(\tau,t)}[\nu](\cdot) = \Phi_{(\tau,t)}[\mu](\cdot) + w_{\Bmu}(t,\cdot) + o_{\tau,t,K}(W_{2,\Bmu}(\mu,\nu)), 
\end{equation}
for all times $t \in [0,T]$ with $\sup_{\tau,t \in [0,T]} \NormC{o_{\tau,t,K}(W_{2,\Bmu}(\mu,\nu))}{0}{K,\R^d} = o_K(W_{2,\Bmu}(\mu,\nu))$. Here, for any $x \in K$, the map $t \in [0,T] \mapsto w_{\Bmu}(t,x) \in \R^d$ is the unique solution of the linearised Cauchy problem 
\begin{equation}
\label{eq:LinearisedMeasure_Cauchy}
\left\{ 
\begin{aligned}
\partial_t w_{\Bmu}(t,x) & = \D_x v \Big( t,\mu(t),\Phi_{(\tau,t)}[\mu](x) \Big)w_{\Bmu}(t,x) \\
& \hspace{-0.85cm} + \INTDom{\D_{\mu} v \Big(t , \mu(t) , \Phi_{(\tau,t)}[\mu](x) \Big) \big( \Phi_{(\tau,t)}[\mu](y) \big) \Big( \D_x \Phi_{(\tau,t)}[\mu](y)(z-y) + w_{\Bmu}(t,y) \Big)}{\R^{2d}}{\Bmu(y,z)}, \\
w_{\Bmu}(\tau,x) & = 0,
\end{aligned}
\right.
\end{equation}
where $\mu(t) := \Phi_{(\tau,t)}[\mu](\cdot)_{\#} \mu$ for all times $t \in [0,T]$, and the following Gr\"onwall-type estimate
\begin{equation}
\label{eq:w_BmuEst}
\NormC{w_{\Bmu}(\cdot,\cdot)}{0}{[0,T] \times K,\R^d} \leq C_K W_{2,\Bmu}(\mu,\nu), 
\end{equation}
holds with a constant $C_K > 0$ depending only on $K$. 
\end{thm}

\begin{proof}
The proof of this result being fairly long and relying on several preliminary steps, we expose it separately in Appendix \ref{section:AppendixThm}
\end{proof}

In the following proposition, we provide a uniform-in-space variant of the classical linearisation result with respect to the initial time for non-local flows. 

\begin{prop}[Derivatives of non-local flows with respect to the initial time]
\label{prop:LinTime_Flows}
Let $\T := \T^* \cap \T_m \subset (0,T)$ be the intersection of $\T^*$ as defined above with the set $\T_m$ of Lebesgue points of $m(\cdot) \in L^1([0,T],\R_+)$. Then for every $\tau \in \T$ and $\mu \in \Pcal(K)$, the map  $s \in \T \mapsto \Phi_{(s,\cdot)}[\mu](\cdot) \in C^0([0,T] \times K,\R^d)$ is differentiable at $\tau$. Moreover, given $h \in \R$ such that $\tau+h \in [0,T]$, the following Taylor expansion 
\begin{equation}
\label{eq:TaylorTime}
\Phi_{(\tau+h,t)}[\mu](\cdot) = \Phi_{(\tau,t)}[\mu](\cdot) + h \Psi_{\tau}(t,\cdot) + o_{\tau,t,K}(h), 
\end{equation}
holds in $C^0(K,\R^d)$ for all times $t \in [0,T]$, with $\sup_{t \in [0,T]} \NormC{o_{\tau,t,K}(h)}{0}{K,\R^d} = o_{\tau,K}(h)$. Here, for any $x \in K$, the map $t \in [0,T] \mapsto \Psi_{\tau}(t,x) \in \R^d$ is the unique solution of the linearised Cauchy problem 
\begin{equation}
\label{eq:LinearisedTime_Cauchy}
\left\{
\begin{aligned}
\partial_t \Psi_{\tau}(t,x) & = \D_x v \Big( t , \mu(t) , \Phi_{(\tau,t)}[\mu](x) \Big) \Psi_{\tau}(t,x) \\
& \hspace{0.4cm} + \INTDom{\D_{\mu} v \Big( t, \mu(t) ,\Phi_{(\tau,t)}[\mu](x)\Big) \big( \Phi_{(\tau,t)}[\mu](y) \big) \Psi_{\tau}(t,y)}{\R^{2d}}{\mu(y)}, \\
\Psi_{\tau}(\tau,x) & = - v(\tau,\mu,x),
\end{aligned}
\right. 
\end{equation} 
where $\mu(t) := \Phi_{(\tau,t)}[\mu](\cdot)_{\#} \mu$ for all times $t \in [0,T]$
\end{prop}

\begin{proof}
The fact that for any $\mu \in \Pcal(K)$, the map $s \in [0,T] \mapsto \Phi_{(s,\cdot)}[\mu](\cdot) \in C^0([0,T] \times K,\R^d)$ is differentiable at all $\tau \in \T$ can be proven by repeating the proof of Theorem \ref{thm:FlowDiff}. We denote by $\Psi_{\tau}(\cdot,\cdot) \in C^0([0,T] \times K,\R^d)$ the corresponding derivative.  

Let us now fix $h \in \R$ such that $\tau + h \in [0,T]$. Notice that by Lemma \ref{lem:AdmEst}, there exists $R'_r \geq r > 0$ such that  $\Phi_{(\tau+h,t)}[\mu](x) \in K' := B(0,R_r')$ for any $(t,x) \in [0,T] \times K$. Thus, denoting $\mu(s) := \Phi_{(\tau,s)}[\mu](\cdot)_{\#} \mu$ for any $s \in [\tau,\tau+h]$, it holds
\begin{equation}
\label{eq:ImportantDistEst}
\begin{aligned}
W_2 \Big( \mu(s) , \Phi_{(\tau+h,s)}[\mu](\cdot)_{\#} \mu \Big) & = W_2 \Big( \Phi_{(\tau+h,s)}[\mu(\tau+h)](\cdot)_{\#} \mu(\tau+h) \, , \, \Phi_{(\tau+h,s)}[\mu](\cdot)_{\#} \mu \Big) \\
& \leq W_2 \Big( \Phi_{(\tau+h,s)}[\mu(\tau+h)](\cdot)_{\#} \mu(\tau+h) \, , \, \Phi_{(\tau+h,s)}[\mu(\tau+h)](\cdot)_{\#} \mu \Big) \\
& \hspace{0.45cm} + W_2 \Big(\Phi_{(\tau+h,s)}[\mu(\tau+h)](\cdot)_{\#} \mu \, , \, \Phi_{(\tau+h,s)}[\mu](\cdot)_{\#} \mu \Big) \\
& \leq C_{K'} \INTSeg{\hspace{-0.1cm} m(s)}{s}{\tau}{\tau+h}, 
\end{aligned}
\end{equation}
where the constant $C_{K'} > 0$ is independent of $\tau \in \T$ and $\mu \in \Pcal(K)$, and given explicitly by 
\begin{equation*}
C_{K'} := (1+2R_r) \sup_{s \in [\tau,\tau+h]} \bigg( \max_{\nu \in \Pcal(K')} \Lip \big( \Phi_{(\tau+h,s)}[\nu](\cdot) ; K' \big) + \max_{x \in K'} ~ \Lip \big( \Phi_{(\tau+h,s)}[\cdot](x); \Pcal_1(K') \big) \bigg).
\end{equation*}
as a consequence of \eqref{eq:WassEst1} and \eqref{eq:WassEst2} together with Lemma \ref{lem:LipFlow} and \eqref{eq:SuppAC_Est}. 

We now focus on the first-order expansion \eqref{eq:TaylorTime}. Observe that by \eqref{eq:NonLocalFlow_Def}, it holds 
\begin{equation}
\label{eq:FlowLinTime1}
\begin{aligned}
\Phi_{(\tau+h,t)}[\mu](x) & = x + \INTSeg{v \Big(s, \Phi_{(\tau+h,s)}[\mu](\cdot)_{\#} \mu , \Phi_{(\tau+h,s)}[\mu](x) \Big)}{s}{\tau+h}{t} \\
& = x + \INTSeg{v \Big( s , \Phi_{(\tau+h,s)}[\mu](\cdot)_{\#} \mu , \Phi_{(\tau+h,s)}[\mu](x) \Big)}{s}{\tau}{t} \\
& \hspace{0.75cm} - \INTSeg{v \Big( s , \Phi_{(\tau+h,s)}[\mu](\cdot)_{\#} \mu , \Phi_{(\tau+h,s)}[\mu](x) \Big)}{s}{\tau}{\tau+h},
\end{aligned}
\end{equation}
for any $(t,x) \in [0,T] \times K$. We start by studying the behaviour as $h \rightarrow 0^+$ of the second integral in the right-hand side of \eqref{eq:FlowLinTime1}. Up to replacing each $v(s,\mu,\cdot)$ by its restriction to $K'$, one has 
\begin{equation}
\label{eq:FlowLinTime21}
\INTSeg{v(s,\mu,x)}{s}{\tau}{\tau+h} = h \, v(\tau,\mu,x) + o_{\tau,x,K}(h), 
\end{equation}
with $\sup_{x \in K} |o_{\tau,x,K}(h)| = o_{\tau,K}(h)$, since every $\tau \in \T$ is a Lebesgue point of $t \in [0,T] \mapsto v_{|K'}(t,\mu,\cdot) \in C^0(K',\R^d)$. In addition, it also holds
\begin{equation}
\label{eq:FlowLinTime22}
\begin{aligned}
& \INTSeg{\Big| v \Big( s, \mu(s) , \Phi_{(\tau+h,s)}[\mu](x) \Big) - v(s, \mu,x) \Big|}{s}{\tau}{\tau+h} \\
&\leq \Big( \INTSeg{l_{K'}(s)}{s}{\tau}{\tau+h} \Big) \hspace{-0.15cm} \sup_{s \in [\tau,\tau+h]} \big| \Phi_{(\tau+h,s)}[\mu](x) - x \big| + \Big( \INTSeg{L_{K'}(s)}{s}{\tau}{\tau+h} \Big) \hspace{-0.15cm} \sup_{s \in [\tau,\tau+h]} \NormC{\Phi_{(\tau,s)}[\mu](\cdot) - \Id}{0}{K,\R^d} \\
& \leq \bigg( \INTSeg{ \Big( l_{K'}(s) + L_{K'}(s) \Big)}{s}{\tau}{\tau+h} \bigg) \Big( \INTSeg{m_r(s)}{s}{\tau}{\tau+h} \Big) \leq o_{\tau,K}(h), 
\end{aligned}
\end{equation}
because $l_{K'}(\cdot),L_{K'}(\cdot) \in L^1([0,T],\R_+)$ and $\INTSeg{m_r(s)}{s}{\tau}{\tau+h} = O_{\tau,K}(h)$, since $m_r(\cdot) = (1+2R_r)m(\cdot)$ and we assumed that $\tau \in \T$ is a Lebesgue points of $m(\cdot)$. Combining \eqref{eq:FlowLinTime21} and \eqref{eq:FlowLinTime22} yields
\begin{equation}
\label{eq:FlowLinTime23}
\INTSeg{v \Big(s,\mu(s),\Phi_{(\tau+h,s)}[\mu](x) \Big)}{s}{\tau}{\tau+h} = h \, v(\tau,\mu,x) + o_{\tau,x,K}(h), 
\end{equation}
with $\sup_{x \in K} |o_{x,\tau,K}(h)| = o_{\tau,K}(h)$. Observe now that under hypotheses \ref{hyp:CE}-$(ii)$, it also holds 
\begin{equation}
\label{eq:FlowLinTime31}
\begin{aligned}
& \INTSeg{\Big| v \Big(s, \Phi_{(\tau+h,s)}[\mu](\cdot)_{\#} \mu , \Phi_{(\tau+h,s)}[\mu](x) \Big) - v \Big(s,\mu(s),\Phi_{(\tau+h,s)}[\mu](x) \Big) \Big|}{s}{\tau}{\tau+h}  \\
& \leq \INTSeg{L_{K'}(s) W_1 \Big( \mu(s) , \Phi_{(\tau+h,s)}[\mu](\cdot)_{\#} \mu \Big)}{s}{\tau}{\tau+h} \\
& \leq C_{K'} \Big( \INTSeg{L_{K'}(s)}{s}{\tau}{\tau+h} \Big) \Big( \INTSeg{m(s)}{s}{\tau}{\tau+h} \Big) = o_{\tau,K}(h),
\end{aligned}
\end{equation}
where we used the distance estimate \eqref{eq:ImportantDistEst} along with the facts that $L_{K'}(\cdot) \in L^1([0,T],\R_+)$ and $\tau \in \T$ is a Lebesgue point of $m(\cdot)$. Thus, by merging \eqref{eq:FlowLinTime23} and \eqref{eq:FlowLinTime31}, we can conclude 
\begin{equation}
\label{eq:FlowLinTime32}
\INTSeg{v \Big( s , \Phi_{(\tau+h,s)}[\mu](\cdot)_{\#} \mu , \Phi_{(\tau+h,s)}[\mu](x) \Big)}{s}{\tau}{\tau+h} = h v(\tau,\mu,x) + o_{\tau,x,K}(h),
\end{equation}
for any $\tau \in \T$ and all $h \in \R$ such that $\tau+h \in [0,T]$, where $\sup_{x \in K} |o_{\tau,x,K}(h)| = o_{\tau,K}(h)$. 

We now focus on the first integral in the right-hand side of \eqref{eq:FlowLinTime1}. Observe that for every $s \in [0,T]$, all $\nu \in \Pcal(K')$ and any $x \in K$, it holds 
\begin{equation}
\label{eq:FlowLinTime4}
\begin{aligned}
v \Big(s,\nu, \Phi_{(\tau+h,s)}[\mu](x) \Big) & = v \Big(s,\nu,\Phi_{(\tau,s)}[\mu](x) + h \Psi_{\tau}(s,x) + o_{s,\nu,x,K}(h) \Big) \\
& = v \Big(s, \nu, \Phi_{(\tau,s)}[\mu](x) \Big) + h \, \D_x v \big(s, \nu , \Phi_{(\tau,s)}[\mu](x) \big) \Psi_{\tau}(s,x) + o_{s,\nu,x,K}(h), 
\end{aligned}
\end{equation}
for $\Lcal^1$-almost every $s \in [\tau,t]$ as a consequence of \ref{hyp:CE}-$(iii)$, where 
\begin{equation*}
\INTSeg{\sup_{(\nu,x) \in \Pcal(K') \times K}|o_{s,\nu,x,K}(h)|}{s}{0}{T} = o_K(h). 
\end{equation*} 
Choose now $\nu := \Phi_{(\tau+h,s)}[\mu](\cdot)_{\#} \mu$ and recall that $\nu \in \Pcal_c(\R^d) \mapsto v(t,\nu,x) \in \R^d$ is locally differentiable as a consequence of hypothesis \ref{hyp:CE}-$(iv)$. Thus, by applying Corollary \ref{cor:DiffChainrule} with 
\begin{equation*}
\Bmu_{\tau}^s := \Big( \Phi_{(\tau,s)}[\mu](\cdot) \, , \, \Phi_{(\tau+h,s)}[\mu](\cdot)\Big)_{\raisebox{4pt}{$\scriptstyle{\#}$}} \mu, 
\end{equation*}
the right-hand side of \eqref{eq:FlowLinTime4} can be further expanded for $\Lcal^1$-almost every $s \in [0,T]$ and all $x \in K$ as
\begin{equation}
\label{eq:FlowLinTime5}
\begin{aligned}
& v \Big(s, \Phi_{(\tau+h,s)}[\mu](\cdot)_{\#} \mu, \Phi_{(\tau+h,s)}[\mu](x) \Big) \\
& = v \Big(s,\mu(s), \Phi_{(\tau,s)}[\mu](x) \Big) + h \, \D_x v \Big(s, \mu(s) , \Phi_{(\tau,s)}[\mu](x) \Big) \Psi_{\tau}(s,x) \\
& \hspace{2cm} + h \INTDom{ \bigg( \D_{\mu} v \Big(s, \mu(s), \Phi_{(\tau,s)}[\mu](x) \Big) \big( \Phi_{(\tau,s)}[\mu](y) \big) \Psi_{\tau}(s,y) \bigg)}{\R^{2d}}{\mu(y)} \\
& \hspace{2cm} + o_{\tau,s,x,K} \Big( W_{2,\Bmu_{\tau}^s} \big( \mu(s) , \Phi_{(\tau+h)}[\mu](\cdot)_{\#} \mu \big) \Big) \\
& = v \Big(s,\mu(s), \Phi_{(\tau,s)}[\mu](x) \Big) + h \, \D_x v \Big(s, \mu(s) , \Phi_{(\tau,s)}[\mu](x) \Big) \Psi_{\tau}(s,x) \\
& \hspace{2cm} + h \INTDom{ \bigg( \D_{\mu} v \Big(s, \mu(s), \Phi_{(\tau,s)}[\mu](x) \Big) \big( \Phi_{(\tau,s)}[\mu](y) \big) \Psi_{\tau}(s,y) \bigg)}{\R^{2d}}{\mu(y)} + o_{\tau,s,x,K}(h),
\end{aligned}
\end{equation}
with $\INTSeg{\sup_{x \in K}|o_{\tau,s,x,K}(h)|}{s}{0}{T} = o_{\tau,K}(h)$, and where we used again the distance estimate \eqref{eq:ImportantDistEst} together with the fact that $\tau \in \T$ is a Lebesgue point of $m(\cdot)$. Whence, by plugging \eqref{eq:FlowLinTime23} and \eqref{eq:FlowLinTime5} into \eqref{eq:FlowLinTime1} and identifying terms, we obtain 
\begin{equation*}
\begin{aligned}
\Psi_{\tau}(t,x) = - v(\tau,\mu,x) & + \INTSeg{\D_x v \Big(s, \mu(s) , \Phi_{(\tau,s)}[\mu](x) \Big) \Psi_{\tau}(s,x)}{s}{\tau}{t} \\
& + \INTSeg{\INTDom{\D_{\mu} v \Big(s,\mu(s),\Phi_{(\tau,s)}[\mu](x) \Big) \big( \Phi_{(\tau,s)}[\mu](y) \big) \Psi_{\tau}(s,y) }{\R^d}{\mu(y)}}{s}{\tau}{t},
\end{aligned}
\end{equation*}
for all $(t,x) \in [0,T] \times K$, which is equivalent to saying that $\Psi_{\tau}(\cdot,\cdot) \in C^0([0,T] \times K,\R^d)$ solves \eqref{eq:LinearisedTime_Cauchy}. The uniqueness of such solutions follows from an application of Gr\"onwall's Lemma. 
\end{proof}

We end this section about linearisations of non-local flows by establishing a total derivative formula with respect to the initial space, measure and time variables. 

\begin{cor}[Total derivative of non-local flows]
\label{cor:TotalLinFlows}
Let $\T \subset (0,T)$ be as in Proposition \ref{prop:LinTime_Flows}, and fix $(\tau,\mu,x) \in \T \times \Pcal(K) \times K$. Then for every $h  \in \R$ such that $\tau+h \in [0,T]$, any $\nu \in \Pcal(K)$ and all $y \in K$, the following Taylor expansion holds
\begin{equation}
\label{eq:TotalTaylor}
\begin{aligned}
\Phi_{(\tau+h,t)}[\nu](y) = \Phi_{(\tau,t)}[\mu](x) + \D_x \Phi_{(\tau,t)}[\mu](x)(y-x)  & + w_{\Bmu}(t,x) + h \Psi_{\tau}(t,x) \\
& + o_{\tau,t,x,K} \Big( |x-y| + W_{2,\Bmu}(\mu,\nu) + h\Big),
\end{aligned}
\end{equation}
for all times $t \in [0,T]$ and any $\Bmu \in \Gamma(\mu,\nu)$, where $\sup_{(\tau,t,x) \in \T \times [0,T] \times K} |o_{\tau,t,x,K}(r)| = o_K(r)$ as $r \rightarrow 0^+$. Here, the maps $\D_x \Phi_{(\tau,\cdot)}[\mu](\cdot)$, $w_{\Bmu}(\cdot,\cdot)$ and $\Psi_{\tau}(\cdot,\cdot)$ are defined as in Proposition \ref{prop:LinSpace_Flows}, Theorem \ref{thm:FlowDiff} and Proposition \ref{prop:LinTime_Flows} respectively. 
\end{cor}

\begin{proof}
The proof of \eqref{eq:TotalTaylor} is obtained by chaining the first-order expansions \eqref{eq:TaylorSpace}, \eqref{eq:TaylorMeasure} and \eqref{eq:TaylorTime} with respect to the time, space and measure variables. By Proposition \ref{prop:LinSpace_Flows}, it first holds 
\begin{equation}
\label{eq:TotalDerivative1}
\Phi_{(\tau+h,t)}[\nu](y) = \Phi_{(\tau+h,t)}[\nu](x) + \D_x \Phi_{(\tau+h,t)}[\nu](x)(y-x) + o_{\tau,t,h,\nu,x,K}(|x-y|), 
\end{equation}
with
\begin{equation*}
\sup_{(t,h) \in [0,T] \times (-\tau,T-\tau)} \bigg( \sup_{(\nu,x) \in \Pcal(K) \times K} o_{\tau,t,h,\nu,x,K}(r) \bigg) = o_{\tau,K}(r), 
\end{equation*}
as $r \rightarrow 0^+$. By applying arguments similar to those of the proof of Proposition \ref{prop:LinSpace_Flows} above, it can be shown under hypotheses \ref{hyp:CE}-$(ii),(iii)$ that 
\begin{equation*}
\NormC{\D_x \Phi_{(\tau+h,\cdot)}[\nu](\cdot) - \D_x \Phi_{(\tau,\cdot)}[\mu](\cdot)}{0}{[0,T] \times K,\R^{d \times d}} ~\underset{(h,\nu) \, \rightarrow \, (0,\mu)}{\longrightarrow}~ 0,
\end{equation*}
and we can deduce from \eqref{eq:TotalDerivative1} that
\begin{equation*}
\Phi_{(\tau+h,t)}[\nu](y) = \Phi_{(\tau+h,t)}[\nu](x) + \D_x \Phi_{(\tau,t)}[\mu](y)(y-x) + o_{\tau,t,h,\nu,x,K}(|x-y|),
\end{equation*}
with $\sup_{(t,h,\nu,x) \in [0,T] \times (-\tau,T-\tau) \times \Pcal(K) \times K} |o_{\tau,t,h,\nu,x,K}(|x-y|)| = o_{\tau,K}(|x-y|)$. By Theorem \ref{thm:FlowDiff}, we can further expand the first term in the right-hand side of this expression as
\begin{equation}
\label{eq:TotalDerivative2}
\Phi_{(\tau+h,t)}[\nu](x) = \Phi_{(\tau+h,t)}[\mu](x) + w_{\Bmu}^h(t,x) + o_{\tau,t,h,x,K} \Big( W_{2,\Bmu}(\mu,\nu) \Big), 
\end{equation}
for any $\Bmu \in \Gamma(\mu,\nu)$, with $\sup_{(t,h,x) \in [0,T] \times (-\tau,T-\tau) \times K}o_{\tau,t,h,x,K}(r) = o_{\tau,K}(r)$ as $r \rightarrow 0^+$, and where for any $x \in K$ the map $t \in [0,T] \mapsto w_{\Bmu}^h(t,x) \in \R^d$ is the unique solution of the linearised Cauchy problem
\begin{equation*}
\left\{ 
\begin{aligned}
& \partial_t w_{\Bmu}^h(t,x) = \D_x v \Big( t , \Phi_{(\tau+h,t)}[\mu](\cdot)_{\#} \mu , \Phi_{(\tau+h,t)}[\mu](x) \Big) w_{\Bmu}^h(t,x) \\
& \hspace{0.95cm} + \int_{\R^{2d}} \D_{\mu} v \Big( t , \Phi_{(\tau+h,t)}[\mu](\cdot)_{\#} \mu , \Phi_{(\tau+h,t)}[\mu](x) \Big) \big(\Phi_{(\tau+h,t)}[\mu](y) \big) \\
& \hspace{8.7cm} \Big( \D_x \Phi_{(\tau+h,t)}[\mu](y)(z-y) + w_{\Bmu}^h(t,y) \Big) \textnormal{d} \Bmu(y,z), \\
& w_{\Bmu}^h(\tau+h,x) = 0. 
\end{aligned}
\right.
\end{equation*}
Again, by invoking hypotheses \ref{hyp:CE}-$(ii),(iii),(iv)$ and repeating the Gr\"onwall-type estimates detailed above and in Appendix \ref{section:AppendixThm} while using the fact that elements of $\T^*$ are uniform Lebesgue points of $t \in [0,T] \mapsto \D_x v(t,\cdot,\cdot)$ and $t \in [0,T] \mapsto \D_{\mu} v(t,\cdot,\cdot)(\cdot)$ by construction, it can further be proven that
\begin{equation*}
w_{\Bmu}^h(t,x) = w_{\Bmu}(t,x) +o_{\tau,t,x,K}\Big( W_{2,\Bmu}(\mu,\nu) + h \Big) ,
\end{equation*}
with $\sup_{(t,x) \in [0,T] \times K} |o_{\tau,t,x,K}(r)| = o_{\tau,K}(r)$ as $r \rightarrow 0^+$. Thus combining \eqref{eq:TotalDerivative1} and \eqref{eq:TotalDerivative2}, we obtain
\begin{equation}
\label{eq:TotalDerivative3}
\Phi_{(\tau+h,t)}[\nu](y) = \Phi_{(\tau+h,t)}[\mu](x) + \D_x \Phi_{(\tau,t)}[\mu](x)(y-x) + w_{\Bmu}(t,x) + o_{\tau,t,h,x,K} \Big( |x-y| + W_{2,\Bmu}(\mu,\nu) + h \Big),
\end{equation}
where $\sup_{(t,h,x) \in [0,T] \times (-\tau,T-\tau) \times K} |o_{\tau,t,h,x,K}(r)| = o_{\tau,K}(r)$ as $r \rightarrow 0^+$. Finally since $\tau \in \T$ by assumption, it holds as a consequence of Proposition \ref{prop:LinTime_Flows} that 
\begin{equation}
\label{eq:TotalDerivative4}
\Phi_{(\tau+h,t)}[\mu](x) = \Phi_{(\tau,t)}[\mu](x) + h \Psi_{\tau}(t,x) + o_{\tau,t,x,K}(h), 
\end{equation}
with $\sup_{(t,x) \in [0,T] \times K} |o_{\tau,t,x,K}(h)| = o_{\tau,K}(h)$ as $ h \rightarrow 0$. Hence, by merging \eqref{eq:TotalDerivative3} and \eqref{eq:TotalDerivative4}, we recover the full Taylor expansion \eqref{eq:TotalTaylor}, which concludes the proof of our corollary.
\end{proof}


\setcounter{section}{0} 
\renewcommand{\thesection}{B} 

\section{Proof of Theorem \ref{thm:FlowDiff}}
\label{section:AppendixThm}

\setcounter{Def}{0} \renewcommand{\thethm}{B.\arabic{Def}} 
\setcounter{equation}{0} \renewcommand{\theequation}{B.\arabic{equation}}

In this section, we detail the proof of Theorem \ref{thm:FlowDiff}. The latter is inspired by that of \cite[Proposition 5]{PMPWass}, and is based on an application of the following parametrised version of Banach fixed point theorem.

\begin{thm}[Banach fixed point theorem with parameter]
\label{thm:Banach}
Let $(\Scal,d_{\Scal})$ be a metric space, $(X,\Norm{\cdot}_X)$ be a Banach space and $\Lambda : \Scal \times X \rightarrow X$ be a continuous mapping. Moreover, suppose that there exists a constant $\kappa \in (0,1)$ such that for every $s\in \Scal$, it holds
\begin{equation*}
\Norm{\Lambda(s,x) - \Lambda(s,y)}_X ~\leq \kappa \Norm{x-y}_X, 
\end{equation*}
for all $x,y \in X$. Then for any $s \in \Scal$, there exists a unique $x(s) \in X$ such that 
\begin{equation*}
\Lambda(s,x(s)) = x(s).
\end{equation*}
Moreover, the map $s \in \Scal \mapsto x(s) \in X$ is continuous, and such that the following estimate holds
\begin{equation}
\label{eq:FixedPointEstBanach}
\Norm{y - x(s)}_X \leq \tfrac{1}{1-\kappa} \Norm{y - \Lambda(s,y)}_X, 
\end{equation}
for any $(s,y) \in \Scal \times X$.
\end{thm}

Before moving to the proof of Theorem \ref{thm:FlowDiff}, we derive three additional technical lemmas that will be useful in the sequel. Given $\mu \in \Pcal(K)$ and $\tau \in [0,T]$, we will again use the condensed notation $\mu(t) := \Phi_{(\tau,t)}[\mu](\cdot)_{\#} \mu$ for all times $t \in [0,T]$, to lighten the computations throughout this section. 

\begin{lem}[Integral of a small-o of the distance]
\label{lem:Small-oEst}
Let $K \subset \R^d$ be a compact set, $\mu,\nu \in \Pcal(K)$ be given, $\Bmu \in \Gamma(\mu,\nu)$ and $(x,y) \in \R^{2d} \mapsto o_K(|x-y|) \in \R_+$ be a map in $L^{\infty}(\R^{2d},\R_+;\Bmu)$ such that 
\begin{equation}
\label{eq:UniformSmall-o}
\lim_{\substack{y \rightarrow x \\ y \in K}} \frac{o_K(|x-y|)}{|x-y|} = 0,
\end{equation}
uniformly with respect to $x \in K$. Then, the following integral estimate holds
\begin{equation*}
\INTDom{o_K(|x-y|)}{\R^{2d}}{\Bmu(x,y)} = o_K(W_{2,\Bmu}(\mu,\nu)).
\end{equation*}
\end{lem}

\begin{proof}
By definition of the small-o, the requirement \eqref{eq:UniformSmall-o} can be written as follows: for any $\epsilon > 0$, there exists $\eta > 0$ such that for every $x,y \in K$, we have $o_K(|x-y|) \leq \epsilon |x-y|$ whenever $|x-y| < \eta$. Let us fix such a pair $\epsilon,\eta > 0$, and observe that 
\begin{equation}
\label{eq:IntegralSmallOEst}
\begin{aligned}
\INTDom{o_K(|x-y|)}{\R^{2d}}{\Bmu(x,y)} & = \INTDom{o_K(|x-y|)}{|x-y| < \eta}{\Bmu(x,y)} + \INTDom{o_K(|x-y|)}{|x-y| \geq \eta}{\Bmu(x,y)} \\
& \leq \epsilon \INTDom{|x-y|}{\R^{2d}}{\Bmu(x,y)} + C_K \Bmu \Big( \big\{ (x,y) \in \R^{2d} ~\text{s.t.}~ |x-y| \geq \eta \big\} \Big).
\end{aligned}
\end{equation}
where $C_K > 0$ is a constant which only depends on $K$. By H\"older's inequality, one can estimate the first term in the right-hand side of \eqref{eq:IntegralSmallOEst} as 
\begin{equation}
\label{eq:IntegralSmallOEst1}
\INTDom{|x-y|}{\R^{2d}}{\Bmu(x,y)} \leq \Big( \INTDom{|x-y|^2}{\R^{2d}}{\Bmu(x,y)} \Big)^{1/2} = W_{2,\Bmu}(\mu,\nu). 
\end{equation}
Concerning the second term in the right-hand side of \eqref{eq:IntegralSmallOEst}, it holds by Chebyshev's inequality
\begin{equation}
\label{eq:IntegralSmallOEst2}
\Bmu \Big( \big\{ (x,y) \in \R^{2d} ~\text{s.t.}~ |x-y| \geq \eta \big\} \Big) \leq \tfrac{1}{\eta^2} \INTDom{|x-y|^2}{|x-y| \geq \eta}{\Bmu(x,y)} \leq \tfrac{1}{\eta^2} W_{2,\Bmu}^2(\mu,\nu).
\end{equation}
Whence, by plugging \eqref{eq:IntegralSmallOEst1} and \eqref{eq:IntegralSmallOEst2} into \eqref{eq:IntegralSmallOEst}, we obtain 
\begin{equation*}
\INTDom{o_K(|x-y|)}{\R^{2d}}{\Bmu(x,y)} \leq \epsilon W_{2,\Bmu}(\mu,\nu) + \tfrac{C_K}{\eta^2} W_{2,\Bmu}^2(\mu,\nu).
\end{equation*}
Thus for every $\epsilon':=2 \epsilon$, there exists $\eta' := \tfrac{\eta^2}{C_K} \epsilon$ such that $\INTDom{o_K(|x-y|)}{\R^{2d}}{\Bmu(x,y)} \leq \epsilon' W_{2,\Bmu}(\mu,\nu)$ whenever $W_{2,\Bmu}(\mu,\nu) \leq \eta'$, which concludes the proof of our claim.
\end{proof}

\begin{lem}[Uniform estimate on directional derivatives of non-local flows]
\label{lem:LinearisedFlowEst}
Let $\mu,\nu \in \Pcal(K)$, fix $\tau \in [0,T]$ and $\Bmu \in \Gamma(\mu,\nu)$. Then, the unique solution $w_{\Bmu}(\cdot,\cdot) \in C^0([0,T] \times K,\R^d)$ of the linearised Cauchy problem \eqref{eq:LinearisedMeasure_Cauchy} satisfies 
\begin{equation}
\label{eq:Lem_wBmuest}
\NormC{w_{\Bmu}(\cdot,\cdot)}{0}{[0,T] \times K,\R^d} \leq C_K W_{2,\Bmu}(\mu,\nu), 
\end{equation}
for some constant $C_K > 0$, along with the integral estimate
\begin{equation}
\label{eq:Integral_wEst}
\INTDom{|w_{\Bmu}(t,x) - w_{\Bmu}(t,y)|}{\R^{2d}}{\Bmu(x,y)} \leq o_{\tau,t,K}(W_{2,\Bmu}(\mu,\nu)),
\end{equation}
where $\sup_{\tau,t \in [0,T]} o_{\tau,t,K}(W_{2,\Bmu}(\mu,\nu)) = o_K(W_{2,\Bmu}(\mu,\nu))$.
\end{lem}

\begin{proof}
The estimate \eqref{eq:Lem_wBmuest} can be obtained under hypotheses \ref{hyp:CE} from a direct application of Gr\"onwall's Lemma to \eqref{eq:LinearisedMeasure_Cauchy}. By Theorem \ref{thm:NonLocalCE}, there exists a compact set $K' \subset \R^d$ depending on $K$ such that $\supp(\mu(t)) \subset K'$ and $\Phi_{(\tau,t)}[\mu](x) \in K'$ for all $(t,x) \in [0,T] \times K$. Thus, given $x,y \in K$, it can be checked by inserting suitable crossed terms in the integral expression of  \eqref{eq:LinearisedMeasure_Cauchy} that
\begin{equation}
\label{eq:TechnicalLem1}
\begin{aligned}
& |w_{\Bmu}(t,x) - w_{\Bmu}(t,y)| \\
& \leq \INTSeg{l_K'(s) |w_{\Bmu}(s,x) - w_{\Bmu}(s,y)|}{s}{\tau}{t} \\
& \hspace{0.45cm} + \bigg( \INTSeg{\Big| \D_x v \Big( s,\mu(s),\Phi_{(\tau,s)}[\mu](x)\Big) - \D_x v \Big( s,\mu(s),\Phi_{(\tau,s)}[\mu](y)\Big) \Big|}{s}{\tau}{t} \bigg) C_K W_{2,\Bmu}(\mu,\nu) \\
&  \hspace{0.45cm} + \bigg( \INTSeg{\max_{\sigma \in \Pcal(K')} \NormC{\D_{\mu} v \Big(s,\sigma,\Phi_{(\tau,s)}[\mu](x)\Big)(\cdot) - \D_{\mu} v \Big(s,\sigma,\Phi_{(\tau,s)}[\mu](y) \Big)(\cdot)}{0}{K',\R^d}}{s}{\tau}{t} \bigg)  \\
& \hspace{11.6cm} \times \Big( L_K' + C_K \Big) W_{2,\Bmu}(\mu,\nu), 
\end{aligned}
\end{equation}
where we used the notations
\begin{equation*}
\begin{aligned}
& l_K'(s) := \max_{\sigma \in \Pcal(K')} \NormC{ \D_x v(s,\sigma,\cdot)}{0}{K',\R^{d \times d}} \quad \text{and} \quad L'_K := \max_{\sigma \in \Pcal(K')} \NormC{\D_x \Phi_{(\tau,\cdot)}[\sigma](\cdot)}{0}{[0,T] \times K',\R^{d \times d}}.
\end{aligned}
\end{equation*}
Moreover, observe that as a consequence of \ref{hyp:CE}-$(iii),(iv)$, it holds
\begin{equation}
\label{eq:UniformContinuityIneq}
\left\{
\begin{aligned}
\Big| \D_x v \Big(s,\mu(s),\Phi_{(\tau,s)}[\mu](x)\Big) - \D_x v \Big(s,\mu(s),\Phi_{(\tau,s)}[\mu](y)\Big) \Big| \hspace{1.35cm} \, & \leq o_{s,x,K}(1), \\
\max_{\sigma \in \Pcal(K')} \NormC{\D_{\mu} v \Big(s,\sigma,\Phi_{(\tau,s)}[\mu](x)\Big)(\cdot) - \D_{\mu} v \Big(s,\sigma,\Phi_{(\tau,s)}[\mu](y) \Big)(\cdot)}{0}{K',\R^d} & \leq o_{s,x,K}(1),
\end{aligned}
\right.
\end{equation}
where
\begin{equation}
\INTSeg{\sup_{x \in K}|o_{s,x,K}(1)|}{s}{0}{T} = o_{K}(1), 
\end{equation}
as $y \rightarrow x$. Hence, we recover by applying Gr\"onwall's lemma to \eqref{eq:TechnicalLem1} together with \eqref{eq:UniformContinuityIneq} that
\begin{equation}
\label{eq:TechnicalLem2}
|w_{\Bmu}(t,x) - w_{\Bmu}(t,y)| \leq (L_K' + 2C_K) \exp(\Norm{l_K'(\cdot)}_1) o_{t,K}\Big( W_{2,\Bmu} \big( \mu,\nu \big) \Big).
\end{equation}
By integrating \eqref{eq:TechnicalLem2} against $\Bmu$ and repeating the steps of the proof of Lemma \ref{lem:Small-oEst}, we finally obtain 
\begin{equation*}
\INTDom{|w_{\Bmu}(t,x) - w_{\Bmu}(t,y)|}{\R^{2d}}{\Bmu(x,y)} \leq o_{t,K} \Big( W_{2,\Bmu}(\mu,\nu) \Big),
\end{equation*}
where $\sup_{t \in [0,T]} |o_{t,K}(W_{2,\Bmu}(\mu,\nu))| = o_K(W_{2,\Bmu}(\mu,\nu))$, which concludes the proof.
\end{proof}

Building on these preliminary estimates, we can move on to the proof of Theorem \ref{thm:FlowDiff}.

\begin{proof}[Proof of Theorem \ref{thm:FlowDiff}]
The proof of this result is based on a classical strategy already explored in a simpler setting in \cite[Proposition 5]{PMPWass}, and that relies on Theorem \ref{thm:Banach}. We start by showing that \eqref{eq:LinearisedMeasure_Cauchy} has a unique solution. Let us fix $\tau \in [0,T]$, an arbitrary plan $\Bmu \in \Gamma(\mu,\nu)$ and consider the operator $\Theta_{\Bmu} : C^0([0,T] \times K,\R^d) \rightarrow C^0([0,T] \times K,\R^d)$ defined by 
\begin{equation}
\label{eq:ThetaDef}
\begin{aligned}
& \Theta_{\Bmu}(w)(t,x) := \INTSeg{\D_x v \Big( s , \mu(s) , \Phi_{(\tau,s)}[\mu](x) \Big)w(s,x)}{s}{\tau}{t} \\
& \hspace{0.225cm} + \INTSeg{\INTDom{\hspace{-0.1cm} \D_{\mu} v \Big(s,\mu(s),\Phi_{(\tau,s)}[\mu](x) \Big) \big( \Phi_{(\tau,s)}[\mu](y) \big) \Big( \D_x \Phi_{(\tau,s)}[\mu](y)(z-y) + w(s,y) \Big)}{\R^{2d}}{\Bmu(y,z)}}{s}{\tau}{t}, 
\end{aligned}
\end{equation}
for any $w \in C^0([0,T] \times K, \R^d)$ and all $(t,x) \in [0,T] \times K$, where we recall that $\mu(s) :=: \Phi_{(\tau,s)}[\mu](\cdot)_{\#} \mu$. As a consequence of hypotheses \ref{hyp:CE}-$(iii),(iv)$, there exist $l_K'(\cdot),L_K'(\cdot) \in L^1([0,T],\R_+)$ such that 
\begin{equation*}
\max_{\sigma \in \Pcal(K')} \NormC{\D_x v (s,\sigma,\cdot)}{0}{K',\R^{d \times d}} \leq l'_K(s) \qquad \text{and} \qquad \max_{\sigma \in \Pcal(K')} \NormC{\D_{\mu} v (s,\sigma,\cdot)(\cdot)}{0}{K' \times K',\R^{d \times d}} \leq L'_K(s), 
\end{equation*}
for $\Lcal^1$-almost every $s \in [0,T]$, where $K' \subset \R^d$ is a compact set depending on $K$ such that $\supp(\mu(t)) \subset K'$ and $\Phi_{(\tau,t)}[\mu](x) \in K'$ for all $(t,x) \in [0,T] \times K$. Whence for all $w_1,w_2 \in C^0([0,T] \times K,\R^d)$, it holds for $t \in [\tau,T]$
\begin{equation}
\label{eq:ThetaEst1}
|\Theta_{\Bmu}(w_1)(t,x) - \Theta_{\Bmu}(w_2)(t,x)| \leq \INTSeg{ \Lpazo_K'(s) \max_{y \in K}|w_1(s,y) - w_2(s,y)|}{s}{\tau}{t}, 
\end{equation}
where $\Lpazo_K'(\cdot) := l'_K(\cdot) + L_K'(\cdot)$, with a similar estimate for $t \in [0,\tau]$. We now consider the following weighted $C^0$-norm, defined by
\begin{equation*}
\NormC{w(\cdot,\cdot)}{0}{[\tau,T] \times K,\R^d}^{\Lpazo_K'} := \max_{(t,x) \in [\tau,T] \times K} e^{-2 \, \INTSeg{\Lpazo_K'(s)}{s}{0}{t}} |w(t,x)|, 
\end{equation*}
and notice that it is equivalent to the usual $C^0$-norm over $[\tau,T] \times K$. Then, \eqref{eq:ThetaEst1} implies that
\begin{equation*}
\begin{aligned}
|\Theta_{\Bmu}(w_1)(t,x) - \Theta_{\Bmu}(w_2)(t,x)| & \leq  \Big( \INTSeg{\Lpazo_K'(s) e^{2 \, \INTSeg{\Lpazo_K'(\zeta)}{\zeta}{0}{s}}}{s}{\tau}{t} \Big) \NormC{w_1 - w_2}{0}{[0,T] \times K,\R^d}^{\Lpazo_K'} \\
& = \frac{1}{2} \Big( e^{2\INTSeg{\Lpazo_K'(s)}{s}{0}{t}} - e^{2\INTSeg{\Lpazo_K'(s)}{s}{0}{\tau}} \Big) \NormC{w_1 - w_2}{0}{[0,T] \times K,\R^d}^{\Lpazo_K'}, 
\end{aligned}
\end{equation*}
for all times $t \in [\tau,T]$, which in turn yields for any $w_1,w_2 \in C^0([0,T] \times K,\R^d)$ that 
\begin{equation*}
\NormC{\Theta_{\Bmu}(w_1) - \Theta_{\Bmu}(w_2)}{0}{[\tau,T] \times K,\R^d}^{\Lpazo_K'} \leq \tfrac{1}{2} \NormC{w_1 - w_2}{0}{[0,T] \times K,\R^d}^{\Lpazo_K'}.
\end{equation*}
In a similar way, we show that the same inequality also holds when $[\tau,T]$ is replaced by $[0,\tau]$. Whence, the operator $\Theta_{\Bmu}(\cdot)$ is contracting with respect to $\NormC{\cdot}{0}{[0,T] \times K,\R^d}^{\Lpazo_K'}$, and by Theorem \ref{thm:Banach} there exists a unique map $w_{\Bmu} \in C^0([0,T] \times K,\R^d)$ such that $\Theta_{\Bmu}(w_{\Bmu}) = w_{\Bmu}$, namely  $w_{\Bmu}(\cdot,\cdot)$ is the unique solution of \eqref{eq:LinearisedMeasure_Cauchy}. Recall also that by Lemma \ref{lem:LinearisedFlowEst}, this mapping is such that 
\begin{equation}
\label{eq:wgamma_est}
\NormC{w_{\Bmu}(\cdot,\cdot)}{0}{[0,T] \times K,\R^d} \leq C_K W_{2,\Bmu}(\mu,\nu), 
\end{equation}
for some constant $C_K > 0$, depending only on $K$.

Consider now the operator $\Lambda_{\tau} : \Pcal(K) \times C^0([0,T] \times K,\R^d) \rightarrow C^0([0,T] \times K,\R^d)$, defined by 
\begin{equation}
\label{eq:LambdaDef}
\Lambda_{\tau} \big( \nu,\Phi \big)(t,x) := x + \INTSeg{v \Big(s,\Phi(s,\cdot)_{\#} \nu,\Phi(s,x) \Big)}{s}{\tau}{t},
\end{equation}
for any $\nu \in \Pcal(K)$ and all $\Phi \in C^0([0,T]\times K,\R^d)$. It can be checked that $\nu \mapsto \Lambda_{\tau}(\nu,\Phi)$ is continuous with respect to the $W_1$-metric for any $\Phi \in C^0([0,T] \times K,\R^d)$, and by repeating the same steps as above,\footnote{But this time with $\Lpazo_K(\cdot) := l_K(\cdot) + L_K(\cdot)$ being given by \ref{hyp:CE}-$(ii)$.} it can also be proven that $\Phi \mapsto \Lambda_{\tau}(\nu,\Phi)$ is contracting with respect to $\NormC{\cdot}{0}{[0,T] \times K,\R^d}^{\Lpazo_K}$ independently from $\nu \in \Pcal(K)$. Whence, by Theorem \ref{thm:Banach}, there exists for any $\nu \in \Pcal(K)$ a unique continuous map $\Phi_{(\tau,\cdot)}[\nu](\cdot) \in C^0([0,T] \times K,\R^d)$ such that 
\begin{equation*}
\Lambda_{\tau} \big( \nu, \Phi_{(\tau,\cdot)}[\nu](\cdot) \big) = \Phi_{(\tau,\cdot)}[\nu](\cdot). 
\end{equation*}
Observing now that $\NormC{\cdot}{0}{[0,T]\times K,\R^d}^{\Lpazo_K}$ is equivalent to the standard $C^0$-norm, there exists by \eqref{eq:FixedPointEstBanach} a constant $C'_K > 0$ such that 
\begin{equation}
\label{eq:FixedPoint_Ineq}
\begin{aligned}
& \NormC{\Phi_{(\tau,\cdot)}[\nu](\cdot) - \Phi_{(\tau,\cdot)}[\mu](\cdot) - w_{\Bmu}(\cdot,\cdot)}{0}{[0,T] \times K ,\R^d} \\
& \hspace{5.3cm} \leq C'_K \NormC{ \big( \Lambda_{\tau}(\nu,\cdot) - \Id \big) \Big( \Phi_{(\tau,\cdot)}[\mu](\cdot) + w_{\Bmu}(\cdot,\cdot) \Big)}{0}{[0,T] \times K ,\R^d}, 
\end{aligned}
\end{equation}
for every $\nu \in \Pcal(K)$. Hence, to complete the proof of Theorem \ref{thm:FlowDiff}, there remains to show that the right-hand side of \eqref{eq:FixedPoint_Ineq} is a $o_K(W_{2,\Bmu}(\mu,\nu))$. 

Let $K' \subset \R^d$ be a compact set such that 
\begin{equation*}
\Phi_{(\tau,t)}[\nu](x) + w_{\Bmu}(t,x) \in K', 
\end{equation*}
for any $\nu \in \Pcal(K)$ and all $(t,x) \in [0,T] \times K$, and fix an element $z \in K'$. It then holds as a consequence of hypothesis \ref{hyp:CE}-$(iv)$ together with Proposition \ref{prop:GradientChainrule} and Corollary \ref{cor:DiffChainrule} that 
\begin{equation}
\label{eq:TaylorVelocity1}
\begin{aligned}
& v \Big(s, ( \Phi_{(\tau,s)}[\mu](\cdot) + w_{\Bmu}(s,\cdot))_{\#} \nu , z \Big) \\
& = v \Big(s, \big( \Phi_{(\tau,s)}[\mu](\cdot) + w_{\Bmu}(s,\cdot) \big) \circ \pi^2 \big)_{\#} \Bmu , z \Big) \\
& = v(s, \mu(s) , z)  \\
& \hspace{0.4cm} + \INTDom{\D_{\mu} v(s,\mu(s),z) \Big( \Phi_{(\tau,s)}[\mu](x) \Big) \Big( \Phi_{(\tau,s)}[\mu](y) + w_{\Bmu}(s,y) - \Phi_{(\tau,s)}[\mu](x) \Big)}{\R^{2d}}{\Bmu(x,y)} \\
& \hspace{0.4cm} + o_{\tau,s,z,K} \bigg( W_{2,\Bmu^s_{\tau}} \Big( \mu(s),(\Phi_{(\tau,s)}[\mu](\cdot) + w_{\Bmu}(s,\cdot))_{\#} \nu \Big) \bigg), 
\end{aligned}
\end{equation}
with $\INTSeg{\sup_{(\tau,z) \in [0,T] \times K'}|o_{\tau,s,z,K}(r)|}{s}{0}{T} = o_K(r)$ as $r \rightarrow 0^+$, and where we chose the particular test plans $\Bmu_{\tau}^s \in \Pcal_c(\R^{2d})$ given for all times $s \in [0,T]$ by 
\begin{equation*}
\Bmu_{\tau}^s := \Big( \Phi_{(\tau,s)}[\mu] \circ \pi^1 \, , \, \Big(\Phi_{(\tau,s)}[\mu](\cdot) + w_{\Bmu}(s,\cdot) \Big) \circ \pi^2 \Big)_{\raisebox{4pt}{$\scriptstyle \#$}} \Bmu.
\end{equation*}
By the triangle inequality of the $L^2(\R^{2d},\R;\Bmu)$-norm together with the Lipschitz regularity of $x \in K \mapsto \Phi_{(\tau,s)}[\mu](x) \in \R^d$ and \eqref{eq:wgamma_est}, it can be checked that 
\begin{equation*}
W_{2,\Bmu_{\tau}^s} \bigg( \mu(s),(\Phi_{(\tau,s)}[\mu](\cdot) + w_{\Bmu}(s,\cdot))_{\#} \nu \Big) \leq \Big( \sup_{t \in [0,T]} \Lip \big(\Phi_{(\tau,t)}[\mu](\cdot) ; K \big) + C_K \bigg)  W_{2,\Bmu}(\mu,\nu), 
\end{equation*}
for all times $s \in [0,T]$, which in turn allows us to recover the asymptotic estimate 
\begin{equation}
\label{eq:SmalloEst1}
\INTSeg{\sup_{(\tau,z) \in [0,T] \times K'} o_{\tau,s,z,K} \bigg( W_{2,\Bmu_{\tau}^s} \Big( \mu(s),(\Phi_{(\tau,s)}[\mu](\cdot) + w_{\Bmu}(s,\cdot))_{\#} \nu \Big) \bigg)}{s}{0}{T} = o_K \Big( W_{2,\Bmu}(\mu,\nu) \Big).
\end{equation}
By Lemma \ref{lem:LinearisedFlowEst} above, it also holds
\begin{equation}
\label{eq:SmalloEst2}
\INTDom{|w_{\Bmu}(s,y) - w_{\Bmu}(s,x)|}{\R^{2d}}{\Bmu(x,y)} = o_{\tau,s,K} \Big( W_{2,\Bmu}(\mu,\nu) \Big),
\end{equation} 
with $\sup_{\tau,s \in [0,T]} |o_{\tau,s,K}(W_{2,\Bmu}(\mu,\nu))| = o_K(W_{2,\Bmu}(\mu,\nu))$. In addition, by Proposition \ref{prop:LinSpace_Flows}, observe that given  $x,y \in K$, one has for all $s \in [\tau,t]$
\begin{equation}
\label{eq:TaylorFlowProof1}
\Phi_{(\tau,s)}[\mu](y) = \Phi_{(\tau,s)}[\mu](x) + \D_x \Phi_{(\tau,s)}[\mu](x)(y-x) + o_{\tau,s,K}(|x-y|), 
\end{equation}
where $\sup_{\tau,s \in [0,T]}|o_{\tau,s,K}(|x-y|)| = o_K(|x-y|)$. This further implies by Lemma \ref{lem:LipFlow} and Lemma \ref{lem:Small-oEst} that
\begin{equation}
\label{eq:SmalloEst3}
\INTDom{o_K(|x-y|)}{\R^{2d}}{\Bmu(x,y)} = o_K \Big( W_{2,\Bmu}(\mu,\nu) \Big), 
\end{equation}
since $\mu,\nu \in \Pcal(K)$. Thus, plugging \eqref{eq:SmalloEst1}, \eqref{eq:SmalloEst2}, \eqref{eq:TaylorFlowProof1} and \eqref{eq:SmalloEst3} into \eqref{eq:TaylorVelocity1} yields
\begin{equation}
\label{eq:TaylorVelocity2}
\begin{aligned}
& v \Big(s, ( \Phi_{(\tau,s)}[\mu](\cdot) + w_{\Bmu}(s,\cdot))_{\#} \nu , z \Big) \\
& =  v(s,\mu(s),z) + \INTDom{\D_{\mu} v(s,\mu(s),z) \Big( \Phi_{(\tau,s)}[\mu](x) \Big) \Big( \D_x \Phi_{(\tau,s)}[\mu](x)(y-x) + w_{\Bmu}(s,y) \Big)}{\R^{2d}}{\Bmu(x,y)} \\
& \hspace{12.15cm} + o_{\tau,s,z,K} \Big( W_{2,\Bmu}(\mu,\nu) \Big),
\end{aligned}
\end{equation}
for $\Lcal^1$-almost every $s \in [0,T]$ and any $z \in K'$.

Now, fix $z := \Phi_{(\tau,s)}[\mu](x) + w_{\Bmu}(s,x) \in K'$ for some $(s,x) \in [0,T] \times K$. Observe first that as a consequence of the continuity of $(\nu,x,y) \in \Pcal_1(K') \times K' \times K' \mapsto \D_{\mu} v(t,\nu,x)(y) \in \R^{d \times d}$ for $\Lcal^1$-almost every $t \in [0,T]$ imposed in \ref{hyp:CE}-$(iv)$, one has 
\begin{equation}
\label{eq:SmalloEst4}
\begin{aligned}
& \D_{\mu} v \Big(s,\mu(s), \Phi_{(\tau,s)}[\mu](x) + w_{\Bmu}(s,x) \Big) \big( \Phi_{(\tau,s)}[\mu](y) \big) = \D_{\mu} v \Big(s,\mu(s), \Phi_{(\tau,s)}[\mu](x) \Big) \big( \Phi_{(\tau,s)}[\mu](y) \big) \\
& \hspace{13.5cm}+ o_{s,x,K}(1),
\end{aligned}
\end{equation}
for $\Lcal^1$-almost every $s \in [0,T]$ , with $\INTSeg{\sup_{x \in K}|o_{s,x,K}(1)|}{s}{0}{T} = o_{K}(1)$ as $W_{2,\Bmu}(\mu,\nu)) \rightarrow 0^+$. Moreover by hypothesis \ref{hyp:CE}-$(iii)$, the map $x \in K' \mapsto v(t,\mu(t),x) \in \R^d$ is continuously differentiable, so that 
\begin{equation}
\label{eq:TaylorVelocity3}
\begin{aligned}
v \Big(s,\mu(s), \Phi_{(\tau,s)}[\mu](x) + w_{\Bmu}(s,x) \Big) & = v \Big(s,\mu(s), \Phi_{(\tau,s)}[\mu](x) \Big) \\
& \hspace{0.4cm} + \D_x v \Big(s,\mu(s), \Phi_{(\tau,s)}[\mu](x) \Big) w_{\Bmu}(s,x) + o_{s,x,K} \Big( W_{2,\Bmu}(\mu,\nu) \Big), 
\end{aligned}
\end{equation}
with $\INTSeg{\sup_{x \in K}|o_{s,x,K}(W_2(\mu,\nu)|}{s}{0}{T} = o_{K}(W_2(\mu,\nu))$, where we used \eqref{eq:wgamma_est}. Whence, by merging \eqref{eq:TaylorVelocity2}, \eqref{eq:SmalloEst4} and \eqref{eq:TaylorVelocity3} and applying Lebesgue's dominated convergence theorem, we recover
\begin{equation*}
\begin{aligned}
& \Lambda_{\tau} \Big(\nu,\Phi_{(s,\cdot)}[\mu](\cdot) + w_{\Bmu}(\cdot,\cdot) \Big)(t,x) \\
& = x + \INTSeg{v \Big(s,\mu(s), \Phi_{(\tau,s)}[\mu](x) \Big)}{s}{\tau}{t} + \int_{\tau}^t \bigg( \D_x v \Big(s,\mu(s), \Phi_{(\tau,s)}[\mu](x) \Big) w_{\Bmu}(s,x) \\
& \hspace{0.75cm} + \INTDom{ \hspace{-0.15cm} \D_{\mu} v \Big(s,\mu(s), \Phi_{(\tau,s)}[\mu](x) \Big) \big( \Phi_{(\tau,s)}[\mu](y) \big) \Big( \D_x \Phi_{(\tau,s)}[\mu](y)(z-y) + w_{\Bmu}(s,y) \Big) }{\R^{2d}}{\Bmu(y,z)} \bigg) \textnormal{d}s \\
& \hspace{0.75cm} + o_{\tau,t,x,K} \Big( W_{2,\Bmu}(\mu,\nu) \Big),
\end{aligned}
\end{equation*}
with $\sup_{(\tau,t,x) \in [0,T] \times [0,T] \times K} |o_{\tau,t,x,K}(W_{2,\Bmu}(\mu,\nu))| = o_K(W_{2,\Bmu}(\mu,\nu))$, and where we used Lemma \ref{lem:Small-oEst}. Recalling that $(t,x) \in [0,T] \times K \mapsto \Phi_{(\tau,t)}[\mu](x) \in \R^d$ is the unique solution of \eqref{eq:NonLocalFlow_Def} and that $(t,x) \in [0,T] \times K \mapsto w_{\Bmu}(t,x) \in \R^d$ solves \eqref{eq:LinearisedMeasure_Cauchy}, we finally obtain that
\begin{equation*}
\NormC{\big( \Lambda_{\tau}(\nu,\cdot) - \Id \big) \Big( \Phi_{(\tau,\cdot)}[\mu](\cdot) + w_{\Bmu}(\cdot,\cdot) \Big)}{0}{[0,T] \times K,\R^d} \leq o_K \Big( W_{2,\Bmu}(\mu,\nu) \Big).
\end{equation*}
for all times $\tau \in [0,T]$. Combining this last estimate with \eqref{eq:FixedPoint_Ineq}, we can therefore conclude 
\begin{equation*}
\Phi_{(\tau,t)}[\nu](\cdot) = \Phi_{(\tau,t)}[\mu](\cdot) + w_{\Bmu}(t,\cdot) + o_{\tau,t,K} \big( W_{2,\Bmu}(\mu,\nu) \big), 
\end{equation*}
with $\sup_{\tau,t \in [0,T]} \NormC{o_{\tau,t,K}(W_{2,\Bmu}(\mu,\nu))}{0}{K,\R^d} = o_{K}(W_{\Bmu}(\mu,\nu))$.
\end{proof}


\setcounter{section}{0} 
\renewcommand{\thesection}{C} 

\section{Proof of the estimate \eqref{eq:TimeShiftedInterp}}
\label{section:AppendixGronwall}

\setcounter{Def}{0} \renewcommand{\thethm}{C.\arabic{Def}} 
\setcounter{equation}{0} \renewcommand{\theequation}{C.\arabic{equation}}

In this section, we detail the arguments subtending an estimate involved in the proof of Theorem \ref{thm:Semiconcavity3} above, stating that there exists a constant $C_K > 0$ such that 
\begin{equation*}
\begin{aligned}
& \INTDom{\Big| (1-\lambda) x + \lambda \Phi^{u_{\epsilon} \circ \vt_{\lambda}}_{(\tau_2,\vt_{\lambda}^{-1}(s))}[\mu_2](y) - \Phi_{(\tau_{\lambda},s)}^{u_{\epsilon}}[\gamma^{1 \rightarrow 2}_{\lambda}] \big( (1-\lambda)x + \lambda y \big) \Big|}{\R^{2d}}{\gamma(x,y)} \\
& \hspace{1.9cm} \leq C_K \lambda(1-\lambda) \Big( |\tau_2 - \tau_1|^2 + W_2^2(\mu_1,\mu_2) \Big) + \ell_K \INTSeg{W_1 \Big( \tilde{\gamma}_{\lambda}^{1 \rightarrow 2}(\zeta) , \Phi_{(\tau_{\lambda},\zeta)}^{u_{\epsilon}}[\gamma^{1 \rightarrow 2}_{\lambda}](\cdot)_{\#} \gamma^{1 \rightarrow 2}_{\lambda}
\Big)}{\zeta}{\tau_{\lambda}}{s},
\end{aligned}
\end{equation*}
where all the relevant quantities are defined in Section \ref{subsection:TimeMeasSemiconcavity}. 

First using \eqref{eq:NonLocalFlow_Def}, we can derive the following ODE characterisation 
\begin{equation}
\label{eq:ODE_Characterisation}
\begin{aligned}
\Phi^{u_{\epsilon} \circ \vt_{\lambda}}_{(\tau_2,\vt_{\lambda}^{-1}(s))}[\mu_2](y) & = y + \INTSeg{v \Big( \xi , \tilde{\mu}_2(\xi) , u_{\epsilon} \circ \vt_{\lambda}(\xi) , \Phi^{u_{\epsilon} \circ \vt_{\lambda}}_{(\tau_2,\xi)}[\mu_2](y) \Big) }{\xi}{\tau_2}{\vt_{\lambda}^{-1}(s)} \\
& = y + \frac{1}{\lambda} \INTSeg{v \Big( \vt_{\lambda}^{-1}(\zeta) , \tilde{\mu}_2 \circ \vt_{\lambda}^{-1}(\zeta) , u_{\epsilon}(\zeta) , \Phi^{u_{\epsilon} \circ \vt_{\lambda}}_{(\tau_2,\vt_{\lambda}^{-1}(\zeta))}[\mu_2](y) \Big)}{\zeta}{\tau_{\lambda}}{s}
\end{aligned}
\end{equation}
for the time-shifted flows, where we used the change of variable $\zeta := \vt_{\lambda}(\xi)$ to go from the first line to the second one. Thus as a consequence of \ref{hyp:R}-$(ii)$ , it holds for every $s \in [\tau_{\lambda},\tau_1]$ and any $x,y \in K$
\begin{equation}
\label{eq:AppGronwall1}
\begin{aligned}
& \Big| (1-\lambda) x + \lambda \Phi^{u_{\epsilon} \circ \vt_{\lambda}}_{(\tau_2,\vt_{\lambda}^{-1}(s))}[\mu_2](y) - \Phi_{(\tau_{\lambda},s)}^{u_{\epsilon}}[\gamma^{1 \rightarrow 2}_{\lambda}] \big( (1-\lambda)x + \lambda y \big) \Big| \\
& \leq \int_{\tau_{\lambda}}^s \Big| v \Big( \vt_{\lambda}^{-1}(\zeta) , \tilde{\mu}_2 \circ \vt_{\lambda}^{-1}(\zeta) , u_{\epsilon}(\zeta) , \Phi^{u_{\epsilon} \circ \vt_{\lambda}}_{(\tau_2,\vt_{\lambda}^{-1}(\zeta))}[\mu_2](y) \Big) \\
& \hspace{3.5cm} - v \Big( \zeta , \Phi_{(\tau_{\lambda},\zeta)}^{u_{\epsilon}}[\gamma^{1 \rightarrow 2}_{\lambda}](\cdot)_{\#} \gamma^{1 \rightarrow 2}_{\lambda} , u_{\epsilon}(\zeta) , \Phi_{(\tau_{\lambda},\zeta)}^{u_{\epsilon}}[\gamma^{1 \rightarrow 2}_{\lambda}] \big( (1-\lambda)x + \lambda y \big) \Big) \Big| \textnormal{d} \zeta \\
& \leq \INTSeg{\ell_K \bigg( |\vt_{\lambda}^{-1}(\zeta) - \zeta| + W_1 \Big( \tilde{\mu}_2 \circ \vt_{\lambda}^{-1}(\zeta), \tilde{\gamma}_{\lambda}^{1 \rightarrow 2}(\zeta) \Big) + (1-\lambda) \Big| \Phi^{u_{\epsilon} \circ \vt_{\lambda}}_{(\tau_2,\vt_{\lambda}^{-1}(\zeta))}[\mu_2](y) - x  \Big| \bigg)}{\zeta}{\tau_{\lambda}}{s} \\
& \hspace{0.4cm} + \INTSeg{\ell_K \Big| (1-\lambda)x + \lambda \Phi^{u_{\epsilon} \circ \vt_{\lambda}}_{(\tau_2,\vt_{\lambda}^{-1}(\zeta))}[\mu_2](y) - \Phi_{(\tau_{\lambda},\zeta)}^{u_{\epsilon}}[\gamma^{1 \rightarrow 2}_{\lambda}] \big( (1-\lambda)x + \lambda y \big) \Big|}{\zeta}{\tau_{\lambda}}{s} \\
& \hspace{0.4cm} + \INTSeg{\ell_K W_1 \Big( \tilde{\gamma}_{\lambda}^{1 \rightarrow 2}(\zeta) , \Phi_{(\tau_{\lambda},\zeta)}^{u_{\epsilon}}[\gamma^{1 \rightarrow 2}_{\lambda}](\cdot)_{\#} \gamma^{1 \rightarrow 2}_{\lambda}
\Big)}{\zeta}{\tau_{\lambda}}{s}, \\
\end{aligned}
\end{equation}
where we inserted crossed terms involving $\tilde{\gamma}^{1 \rightarrow 2}_{\lambda}(\zeta)$ for $\zeta \in [\tau_{\lambda},s]$. Our next goal is to estimate first integral in the right-hand side of \eqref{eq:AppGronwall1}. 

By definition \eqref{eq:Reparametrisation} of $\vt_{\lambda}(\cdot)$, one has $\vt_{\lambda}^{-1}(\zeta) - \zeta = \tfrac{(1-\lambda)}{\lambda}(\zeta - \tau_1)$ for every $\zeta \in [\tau_{\lambda},\tau_1]$, so that 
\begin{equation}
\label{eq:AppGronwall2}
\INTSeg{|\vt_{\lambda}^{-1}(\zeta) - \zeta|}{\zeta}{\tau_{\lambda}}{s} \leq \tfrac{(1-\lambda)}{\lambda} |\tau_{\lambda} - \tau_1|^2 = \lambda(1-\lambda)|\tau_2 - \tau_1|^2, 
\end{equation}
for every $s \in [\tau_{\lambda},\tau_1]$. Besides, by Theorem \ref{thm:NonLocalCE} combined with \ref{hyp:R}-$(i)$, it can be checked that $\tilde{\mu}_2(\cdot)$ is Lipschitz in the $W_1$-metric over $[\tau_2,\tau_1]$ with constant $m_r := (1+2R_r)m$, which implies 
\begin{equation}
\label{eq:AppGronwall3}
\begin{aligned}
\INTSeg{W_1 \Big( \tilde{\mu}_2 \circ \vt_{\lambda}^{-1}(\zeta), \tilde{\gamma}_{\lambda}^{1 \rightarrow 2}(\zeta) \Big)}{\zeta}{\tau_{\lambda}}{s} & \leq \INTSeg{ \bigg(W_1 \Big( \tilde{\mu}_2 \circ \vt_{\lambda}^{-1}(\zeta) , \tilde{\mu}_2(\zeta) \Big) +  W_2 \Big( \tilde{\mu}_2(\zeta), \tilde{\gamma}_{\lambda}^{1 \rightarrow 2}(\zeta) \Big) \bigg)}{\zeta}{\tau_{\lambda}}{s} \\
& \leq (1+2R_r) m \INTSeg{|\vt_{\lambda}^{-1}(\zeta) - \zeta|}{\zeta}{\tau_{\lambda}}{s} + \INTSeg{(1-\lambda)W_2(\mu_1,\tilde{\mu}_2(\zeta))}{\zeta}{\tau_{\lambda}}{s} \\
& \leq 3(1+2R_r)m \, \lambda(1-\lambda) |\tau_2 - \tau_1| \Big( |\tau_2 - \tau_1| + W_2(\mu_1,\mu_2) \Big)
\end{aligned}
\end{equation}
where we used the estimates of \eqref{eq:ValueFunctionTimeEst5} and \eqref{eq:AppGronwall2}, along with the fact that $(\tilde{\gamma}^{1\rightarrow2}_{\lambda}(\zeta))_{\lambda \in [0,1]}$ is a constant speed $W_2$-geodesic between $\mu_1$ and $\tilde{\mu}_2(\zeta)$ for every $\zeta \in [\tau_{\lambda},\tau_1]$. Finally by applying Gr\"onwall's Lemma to the ODE characterisation \eqref{eq:ODE_Characterisation} and using again \ref{hyp:R}-$(i)$, one can show that 
\begin{equation}
\label{eq:AppGronwall4}
\INTSeg{(1-\lambda)\Big| \Phi^{u_{\epsilon} \circ \vt_{\lambda}}_{(\tau_2,\vt_{\lambda}^{-1}(\zeta))}[\mu_2](y) - x  \Big|}{\zeta}{\tau_{\lambda}}{s} \leq \lambda (1-\lambda) |\tau_2 - \tau_1|\Big( m|\tau_2 - \tau_1| + |x-y| \Big) \exp \big(\ell_K |\tau_2-\tau_1| \big), 
\end{equation}
holds for every $s \in [\tau_{\lambda},\tau_1]$ and every $x,y \in K$. Thus, by plugging \eqref{eq:AppGronwall2}, \eqref{eq:AppGronwall3} and \eqref{eq:AppGronwall4} into \eqref{eq:AppGronwall1}, integrating the resulting estimate against $\gamma \in \Gamma_o(\mu_1,\mu_2)$ while applying Fubini's theorem and Gr\"onwall's lemma, we finally obtain
\begin{equation*}
\begin{aligned}
& \INTDom{\Big| (1-\lambda) x + \lambda \Phi^{u_{\epsilon} \circ \vt_{\lambda}}_{(\tau_2,\vt_{\lambda}^{-1}(s))}[\mu_2](y) - \Phi_{(\tau_{\lambda},s)}^{u_{\epsilon}}[\gamma^{1 \rightarrow 2}_{\lambda}] \big( (1-\lambda)x + \lambda y \big) \Big|}{\R^{2d}}{\gamma(x,y)} \\
& \hspace{1.9cm} \leq C_K \lambda(1-\lambda) \Big( |\tau_2 - \tau_1|^2 + W_2^2(\mu_1,\mu_2) \Big) + \ell_K \INTSeg{W_1 \Big( \tilde{\gamma}_{\lambda}^{1 \rightarrow 2}(\zeta) , \Phi_{(\tau_{\lambda},\zeta)}^{u_{\epsilon}}[\gamma^{1 \rightarrow 2}_{\lambda}](\cdot)_{\#} \gamma^{1 \rightarrow 2}_{\lambda}
\Big)}{\zeta}{\tau_{\lambda}}{s},
\end{aligned}
\end{equation*}
where $C_K > 0$ is a constant depending only on $m,\ell_K,T,R_r$. 


\setcounter{section}{0} 
\renewcommand{\thesection}{D} 

\section{Proof of Lemma \ref{lem:Consistency}}
\label{section:AppendixLem}

\setcounter{Def}{0} \renewcommand{\thethm}{D.\arabic{Def}} 
\setcounter{equation}{0} \renewcommand{\theequation}{D.\arabic{equation}}

In this section, we provide the proof of Lemma \ref{lem:Consistency} above, which is a key step in the establishment of the sensitivity relation of Theorem \ref{thm:Sensitivity1}.

\begin{proof}[Proof of Lemma \ref{lem:Consistency}]
We will only prove the statements for the map $\Hpazo_1(\cdot)$, the arguments transposing almost verbatim to $\Hpazo_2(\cdot)$. For the sake of readability, we will use the condensed notation 
\begin{equation}
\label{eq:LemmaProof_Notations}
\Phi_{(\tau,t)}^*(x) := \Phi_{(\tau,t)}^{u^*}[\mu^*(\tau)](x) \qquad \text{and} \qquad \F_{\tau}(t,x,y) := \D_x \Phi^*_{(\tau,t)}(x) (y - x) + w_{\Bmu_{\tau}}(t,x), 
\end{equation}
for all $(t,x,y) \in [0,T] \times \supp(\Bmu_{\tau})$. By Proposition \ref{prop:LinSpace_Flows} and Theorem \ref{thm:FlowDiff}, it can be checked that $t \in [0,T] \mapsto \F_{\tau}(t,x,y) \in \R^d$ is the unique solution of the linearised Cauchy problem
\begin{equation}
\label{eq:CauchyProblemF}
\left\{
\begin{aligned}
\partial_t \F_{\tau}(t,x,y) &  = \D_x v \Big(t,\mu^*(t), \Phi_{(\tau,t)}^*(x)\Big) \F_{\tau}(t,x,y) \\
& \hspace{0.4cm} + \INTDom{\D_{\mu} v \Big(t,\mu^*(t), \Phi_{(\tau,t)}^*(x)\Big)\big( \Phi_{(\tau,t)}^*(z_1) \big) \F_{\tau}(t,z_1,z_2)}{\R^{2d}}{\Bmu_{\tau}(z_1,z_2)} \\
\F_{\tau}(\tau,x,y) & = y-x.
\end{aligned}
\right.
\end{equation}
Recalling that $\supp(\Bmu_{\tau}) \subset \K \times \K$, there exist $\Rpazo_{\F} > 0$ and $m_{\F}(\cdot) \in L^1([0,T],\R_+)$ such that 
\begin{equation*}
|\F_{\tau}(t,x,y)| \leq \Rpazo_{\F} \qquad \text{and} \qquad |\F_{\tau}(t,x,y) - \F_{\tau}(s,x,y)| \leq \INTSeg{m_{\F}(\zeta)}{\zeta}{s}{t}, 
\end{equation*}
for all times $0 \leq s \leq t \leq T$ and any $x,y \in \K$. Let us now fix $t \in [0,T]$, and observe that 
\begin{equation}
\label{eq:Hpazo1Expression}
\Hpazo_1(t) = \INTDom{\INTDom{\Big\langle r \, , \, \F_{\tau} \big(t, \Phi_{(T,\tau)}^*(x) , y \big) \Big\rangle}{\R^d}{\sigma_x^*(t)(r)}}{\R^{2d}}{\gamma_{\tau}^T(x,y)}, 
\end{equation}
by the construction of $\Bnu^*(\cdot)$ introduced in \eqref{eq:Bnu_TDef}, which further yields
\begin{equation*}
\begin{aligned}
|\Hpazo_1(t) - \Hpazo_1(s)| & \leq \INTDom{\INTDom{\Big| \big\langle r \, , \, \F_{\tau} \big(t, \Phi_{(T,\tau)}^*(x) , y \big) - \F_{\tau} \big(s, \Phi_{(T,\tau)}^*(x) , y \big) \big\rangle \Big| }{\R^d}{\sigma_x^*(t)(r)}}{\R^{2d}}{\gamma_{\tau}^T(x,y)} \\
& \hspace{0.4cm} + \bigg| \INTDom{\INTDom{\Big\langle r , \F_{\tau} \big(s, \Phi_{(T,\tau)}^*(x) , y \big) \Big\rangle}{\R^d}{\Big( \sigma_x^*(t) - \sigma_x^*(s) \Big)(r)}}{\R^{2d}}{\gamma_{\tau}^T(x,y)} \bigg| \\
& \leq \INTSeg{\Big( \Rpazo \, m_{\F}(\zeta) + \Rpazo_{\F} \, m^{\sigma}_r(\zeta) \Big)}{\zeta}{s}{t}, 
\end{aligned}
\end{equation*}
for all times $0 \leq s \leq t \leq T$, with $\Rpazo >0$ and $m^{\sigma}_r(\cdot) \in L^1([0,T],\R_+)$ being respectively defined as in steps 1 and 2 of the proof of Theorem \ref{thm:Sensitivity1}, and where we used Kantorovich-Rubinstein's duality formula \eqref{eq:KantorovichDuality}. Thus, $\Hpazo_1(\cdot) \in \AC([0,T],\R)$, and it is differentiable $\Lcal^1$-almost everywhere. 

Let $t \in [0,T]$ be a differentiability point of $\Hpazo_1(\cdot)$. Using \eqref{eq:Hpazo1Expression}, we can compute by applying Lebesgue's theorem for differentiating under the integral sign
\begin{equation}
\label{eq:HpazoDerivative1}
\begin{aligned}
\tderv{}{t}{} \Hpazo_1(t) & = \INTDom{\INTDom{\Big\langle r \, , \, \partial_t \F_{\tau} \big(t, \Phi_{(T,\tau)}^*(x) , y \big) \Big\rangle}{\R^d}{\sigma_x^*(t)(r)}}{\R^{2d}}{\gamma_{\tau}^T(x,y)} \\
& \hspace{0.4cm} + \INTDom{\INTDom{\Big\langle \Wpazo_x(t,\sigma^*_x(t),r) \, , \, \F_{\tau} \big(t,\Phi_{(T,\tau)}^*(x) , y \big) \Big\rangle}{\R^d}{\sigma_x^*(t)(r)}}{\R^{2d}}{\gamma_{\tau}^T(x,y)}, 
\end{aligned}
\end{equation}
with $\Wpazo_x : [0,T] \times \Pcal_c(\R^d) \times \R^d \rightarrow \R^d$ being defined as in \eqref{eq:DualVelocityField}, and where we used the distributional characterisation \eqref{eq:NonLocal_Distrib2} of the fact that $\sigma^*_x(\cdot)$ solves \eqref{eq:BackwardCauchy} for $\mu^*(T)$-almost every $x \in \R^d$. Observe now that as a consequence of \eqref{eq:CauchyProblemF}, the first term in the right-hand side of \eqref{eq:HpazoDerivative1} can be expressed as
\begin{equation}
\label{eq:HpazoDerivative2}
\begin{aligned}
& \INTDom{\INTDom{\Big\langle r \, , \, \partial_t \F_{\tau} \big(t, \Phi_{(T,\tau)}^*(x) , y \big) \Big\rangle}{\R^d}{\sigma_x^*(t)(r)}}{\R^{2d}}{\gamma_{\tau}^T(x,y)} \\
& \hspace{0.4cm} = \INTDom{\INTDom{\Big\langle r \, , \, \D_x v \Big( t,\mu^*(t),\Phi_{(T,t)}^*(x) \Big) \F_{\tau} \big(t,\Phi_{(T,\tau)}^{u^*}(x) , y \big) \Big\rangle}{\R^d}{\sigma_x^*(t)(r)}}{\R^{2d}}{\gamma_{\tau}^T(x,y)} \\
& \hspace{0.8cm} + \INTDom{\INTDom{\Big\langle r \, , \, \INTDom{\D_{\mu} v \Big( t,\mu^*(t),\Phi_{(T,t)}^*(x) \Big) \big( \Phi_{(\tau,t)}^*(z_1) \big) \F_{\tau} \big(t,z_1,z_2 \big)}{\R^{2d}}{\gamma_{\tau}(z_1,z_2)} \Big\rangle}{\R^d}{\sigma_x^*(t)(r)}}{\R^{2d}}{\gamma_{\tau}^T(x,y)}.
\end{aligned}
\end{equation}
By linearity of the integral, the second term in the right-hand side of \eqref{eq:HpazoDerivative2} can be rewritten as 
\begin{equation}
\label{eq:HpazoDerivative3}
\begin{aligned}
& \INTDom{\INTDom{\Big\langle r \, , \, \INTDom{\D_{\mu} v \Big( t,\mu^*(t),\Phi_{(T,t)}^*(x) \Big) \big( \Phi_{(\tau,t)}^*(z_1) \big) \F_{\tau} \big(t,z_1,z_2 \big)}{\R^{2d}}{\gamma_{\tau}(z_1,z_2)} \Big\rangle}{\R^d}{\sigma_x^*(t)(r)}}{\R^{2d}}{\gamma_{\tau}^T(x,y)} \\
&  \hspace{0.4cm} = \INTDom{\hspace{-0.15cm} \Big\langle \INTDom{ \hspace{-0.1cm} \D_{\mu} v \Big( t,\mu^*(t),\Phi_{(T,t)}^*(x) \Big) \big( \Phi_{(T,t)}^*(z_1) \big)^{\top}  \hspace{-0.05cm} r\,}{\R^{2d}}{\nu^*_T(t)(x,r)} , \F_{\tau} \Big(t,\Phi_{(T,\tau)}^*(z_1) ,z_2 \Big) \Big\rangle}{\R^{2d}}{\gamma_{\tau}^T(z_1,z_2)}, \\ 
\end{aligned}
\end{equation}
where we applied Fubini's theorem and used the fact that $\nu_T^*(t) = \INTDom{\sigma^*_x(t)}{\R^d}{\mu^*(T)(x)}$. By plugging \eqref{eq:HpazoDerivative3} into \eqref{eq:HpazoDerivative2} and recalling the analytical expression of $\Wpazo_x : [0,T] \times \Pcal_c(\R^d) \times \R^d \rightarrow \R^d$ given in \eqref{eq:DualVelocityField}, we conclude that 
\begin{equation*}
\tderv{}{t}{} \Hpazo_1(t) = 0, 
\end{equation*}
for $\Lcal^1$-almost every $t \in [0,T]$, and the same analysis can be performed for $\Hpazo_2(\cdot)$. 
\end{proof}


\smallskip

\begin{flushleft}
{\small{\bf  Acknowledgement.}  This material is based upon work supported by the Air Force Office of Scientific Research under award number FA9550-18-1-0254.}
\end{flushleft}


\bibliographystyle{plain}
{\footnotesize
\bibliography{../../ControlWassersteinBib}
}

\end{document}